\documentclass{article}
\usepackage{amsthm, amssymb, amsmath}
\usepackage{enumerate}
\usepackage{amsthm}
\usepackage{graphicx}
\numberwithin{equation}{section}
\newtheorem{theorem}{Theorem}[section]

\newtheorem{lemma}[theorem]{Lemma}

\newtheorem{proposition}[theorem]{Proposition}

\huge

\date{}

\begin{document}
\title{On the mean convexity of a space-and-time neighborhood of generic singularities formed by mean curvature flow}
\author{Zhou Gang\footnote{gzhou@math.binghamton.edu, partly supported by NSF grant DMS-1443225 and DMS-1801387.}}
\maketitle
\centerline{Department of Mathematical Sciences, Binghamton University, Binghamton, NY, 13850}
\setlength{\leftmargin}{.1in}
\setlength{\rightmargin}{.1in}
\normalsize \vskip.1in
\setcounter{page}{1} \setlength{\leftmargin}{.1in}
\setlength{\rightmargin}{.1in}
\large

\date

\setlength{\leftmargin}{.1in}
\setlength{\rightmargin}{.1in}
\normalsize \vskip.1in
\setcounter{page}{1} \setlength{\leftmargin}{.1in}
\setlength{\rightmargin}{.1in}
\large

\section*{Abstract}
We consider one of the generic regimes of formation of singularities. We obtain a detailed description of a possibly small, but fixed, neighborhood of the blowup point, up to (and including) the blowup time, and find that it is mean convex. This confirms a conjecture by Ilmanen. 
And we find that the singularity is isolated from the other ones. 

\tableofcontents

\section{Introduction}
Here we study mean curvature flow (MCF) for a $n-$dimensional hypersurface embedded in $\mathbb{R}^{n+1}$:
\begin{align}\label{eq:MeanCur}
\partial_{t}X_{t}=-h(X_{t}),
\end{align} where $X_{t}$ is the immersion at time $t$, $h(X_{t})$ is the mean curvature vector at the point $X_{t}.$ 

We are mainly interested in the generic blowups, and a small, but fixed, space-and-time neighborhood of the singularity. The objective is to find a detailed description of the neighborhood, and prove it is mean convex in certain regimes.

It was proved by Colding and Minicozzi in \cite{ColdingMiniUniqueness} that for the generic singularity,
suppose that the singularity is at time $T$ and at spatial $0$, then under the scaling $X_{t}\rightarrow \frac{1}{\sqrt{T-t}}X_t$, the manifold will converge to a unique cylinder $\mathbb{R}^{n-k}\times \mathbb{S}_{\sqrt{2k}}^{k},\ k=1,2,\cdots, n-1,$ or $\mathbb{S}^{n}_{\sqrt{2n}}$, here $\mathbb{S}^{k}_{\sqrt{2k}}$ is $k$-dimensional torus with radius $\sqrt{2k}.$

Here we choose to study the regime where the limit cylinder is $\mathbb{R}^3\times \mathbb{S}^{1}_{\sqrt{2}}$. One motivation is that such singularities are not understood as well as $\mathbb{R}\times \mathbb{S}^{n-1}_{\sqrt{2(n-1)}}.$ See the results in \cite{HuiskenSur2009, BrHui2016, MR3602529, MR3662439}. For the other related works, see \cite{Hamilton1997, AltAngGiga1995, sesum2008, AnDaSE15, AnDaSE18}.  We expect that the techniques work equally well for all the generic blowups.

Next we present our results.

Suppose that the blowup point is the origin and the time is $T$, and the limit cylinder is $\mathbb{R}^{3}\times \mathbb{S}^{1}_{\sqrt{2}}$, then for $t$ sufficiently close to the blowup time $T$, a neighborhood of the blowup point can be parametrized by some positive function $u$ as
\begin{align}\label{eq:origPara1}
\Psi(z,t)= \left[
\begin{array}{ccc}
z\\
u(z, \theta, t)\cos\theta\\
u(z,\theta,t) \sin \theta
\end{array}
\right],\ z\in \mathbb{R}^3, \ |z|\leq c(t)\ \text{for some}\ c(t)>0,\ \text{and}\ \theta\in \mathbb{T}.
\end{align}
Here we study the function $u$. We prove that in certain regime there exists a fixed $\epsilon>0$ such that when $t\in [T-\epsilon,T]$, the part of hypersurface in the ball $B_{\epsilon}(0)\subset \mathbb{R}^5$ is of the form \eqref{eq:origPara1}. Moreover we obtain a detailed description of the function $u$, see Theorem \ref{prob:converge} below.

Now we discuss some technical aspects. We study the problem in two steps. In the first step we consider the rescaled MCF, namely $X_{t}\rightarrow \frac{1}{\sqrt{T-t}},$ and define a new function $v$ by
\begin{align}
u(z,\theta,t)=&\sqrt{T-t} \ v(y,\theta, \tau)\label{eq:rescaIni}
\end{align} here $y$ and $\tau$ are the rescaled space and time variables defined as
\begin{align}
y:=\frac{z}{\sqrt{T-t}},\ \tau:=-ln(T-t).
\end{align} 
Then the part parametrized in \eqref{eq:origPara1} becomes
\begin{align}
\sqrt{T-t}\left[
\begin{array}{c}
y\\
v(y,\theta,\tau)cos\theta\\
v(y,\theta,\tau) sin\theta
\end{array}
\right].
\end{align}

We manage to prove that for $|y|\leq 3\tau^{\frac{1}{2}+\frac{1}{20}}$, the dominant part of $v$ is $\sqrt{2+\tau^{-1}y^{T}\tilde{B}y}$ in the sense that
\begin{align}
\Big|v-\sqrt{2+\tau^{-1}y^{T}\tilde{B}y}\Big|\leq \tau^{-\frac{3}{10}},\label{eq:tent1}
\end{align}
where $\tilde{B}$ is a $3\times 3$ diagonal matrix
\begin{align}
\tilde{B}=diag[b_1,b_2,b_3], \ \text{with}\ b_{k}=0\ \text{or}\ 1.
\end{align} And we have estimates on the derivatives of $v$ in this region, see Theorem \ref{THM:TwoReg1} below. 

Compare to the other works, here we do not study Huisken's monotonicity formula since from that it is difficult to control the solution in the region $|y|\leq 3\tau^{\frac{1}{2}+\frac{1}{20}}$. To be seen on an intuitive level, suppose that we have $\|e^{-\frac{1}{8}|y|^2}f\|_{2}\leq \tau^{-10}$, one can not derive decay estimate for $f$ in the region $|y|=\mathcal{O} (\tau^{\frac{1}{2}})$ as $\tau\rightarrow \infty$, since the function $e^{-\frac{1}{8}|y|^2}$ decays too rapidly.

The situation changes if one considers our chosen weighted $L^{\infty}$ norm. Suppose that $\|\langle y\rangle^{-3}f\|_{\infty}\leq \tau^{-2}$, then when $|y|\leq 3\tau^{\frac{1}{2}+\frac{1}{20}}$, we have a decay estimate for $f$,
\begin{align}
|f(y)|\leq \langle y\rangle^{3}\tau^{-2}\lesssim \tau^{-\frac{7}{20}}.
\end{align} To use such weighted $L^{\infty}$-norms we apply propagator estimates, see Lemma \ref{LM:propagator} below. An easy, but essential, version is that, there exists a constant $C$ such that for any $t\geq 0,$
\begin{align}
\|\langle y\rangle^{-3} e^{-tL}f\|_{\infty}\leq Ce^{-\frac{3}{2}t} \|\langle y\rangle^{-3}f\|_{\infty}
\end{align} where $L:=-\Delta+\frac{1}{2}y\cdot \nabla_y,$ and $f$ is any function satisfying $f\perp e^{-\frac{1}{4}|y|^2}, e^{-\frac{1}{4}|y|^2}y_k, e^{-\frac{1}{4}|y|^2}y_k y_l$, $k,l=1,2,3.$ Note that the conjugate of $L$, $e^{-\frac{1}{8}|y|^2}Le^{\frac{1}{8}|y|^2}$ is the Harmonic oscillator.

In the second step we consider the regime where 
\begin{align}
b_1=b_2=b_3=1,\label{eq:ThreeB}
\end{align} which is arguably the most most generic regime in the sense that if one of the $b_k$ is zero, then after a generic perturbation at a large time $\tau_1,$ all these $b_k$ will become 1 as $\tau\rightarrow \infty$. 

Here we consider a new (un-rescaled) MCF from some fixed time $t_1=t(\tau_1)$, with $\tau_1$ sufficiently large and hence $T-t_1$ sufficiently small, with initial condition provided by $v(\cdot,\tau_1)$ (thus is a rescaled version of $u(\cdot,t_1)$ by \eqref{eq:rescaIni}). This makes the new flow of the form
\begin{align}
\left[
\begin{array}{ccc}
z\\
q(z, \theta, s)\cos\theta\\
q(z,\theta,s) \sin \theta
\end{array}
\right],
\end{align} where $q$ is related to $v$ and $u$ by
\begin{align*}
q(z,\theta,s)=&\sqrt{1-s}\ v\big(\frac{z}{\sqrt{1-s}},\theta, -\ln(1-s)+\tau_1\big)\\
=&\frac{1}{\lambda} u(\lambda z, \theta, \lambda^2 s+t_1).\label{eq:quId}
\end{align*}
and $s$ and $\lambda$ are defined as 
\begin{align}
s:=\frac{t-t_1}{\lambda^2}, \ \lambda:=\sqrt{T-t_1}.
\end{align} 
Obviously the new flow will blowup at time $s=1.$

By the identity in \eqref{eq:quId} we have that estimates for $q$ implies those for $u.$ Moreover since the MCF is scaling invariant, the one parametrized by $q$ is also a MCF.

Now we study the region where $|z|\tau_1^{-\frac{1}{2}-\frac{1}{20}}\approx 1$. By the estimate in \eqref{eq:tent1}, it is close to a cylinder with a large radius $\approx \tau_1^{\frac{1}{20}}\gg 1$. Thus, at least intuitively, as $s\rightarrow 1$ i.e. the blowup time, under mean curvature flow, this part will stay close to a cylinder with a large radius. In the proof of Theorem \ref{prob:converge} we will make this rigorous, together with regularity estimate and the technique of local smooth extension, see e.g. \cite{Eckerbook}.

Thus it is critically important that we have a good control on the rescaled MCF in the ball $B_{0}(\tau^{\frac{1}{2}+\epsilon})$ for some $\epsilon>0.$

For the regimes where at least one of $b_k$ in \eqref{eq:ThreeB} is zero, we will address them in subsequent papers.

The present paper depends on our paper \cite{GZ2017}, where, among other results, we proved that the $3\times 3$ symmetric matrix $B(\tau)$, in \eqref{eq:decPr} below, is close to being semi-positive definite. This allows us to consider a large region of the rescaled flow, otherwise $\sqrt{2+y^{T}B y}$ might not be well defined. Some of the key techniques were devised in \cite{GaKnSi, GaKn20142, GS2008}, see also \cite{BrKu}.

The paper is organized as the following: In Section \ref{sec:MainTHM} we state two Main Theorems \ref{THM:TwoReg1}
and \ref{prob:converge}. The Theorem \ref{THM:TwoReg1} is reformulated in Section \ref{sec:refor}. And the results there will be proved in Sections \ref{sec:ReforWeightLInf} and \ref{sec:regul}. The two Main Theorems will be proved in Sections \ref{sec:ProTwoReg1} and \ref{sec:flowThr} respectively. In Sections \ref{sec:estM1}-\ref{sec:M3} we prove some technical results. 

The estimates in Section \ref{sec:estM1} are carried out in detail, with new ideas explained. Some of the details in Sections \ref{sec:M3} and \ref{sec:estM2} will be skipped if they are similar to those in Section \ref{sec:estM1}.

\section*{Acknowledgement}
The author would like to thank Shengwen Wang for pointing out that the proved estimates imply mean convexity, and their importances.

\section{Main Theorem}\label{sec:MainTHM}

As defined in \cite{CoMi2012}, $\lambda(\Sigma)$ is the supremum of the Gaussian surface areas of hypersurface $\Sigma$ over all centers and scales. 
Under the condition that there exists a constant $\lambda_0>0$ such that for any $\tau>0$,
\begin{align}
\lambda(\Sigma_{\tau})\leq \lambda_0,\label{eq:generic}
\end{align} it was proved in \cite{CoMi2012, CIM13} the only possible singularities are cylinders. Then it was proved in \cite{ColdingMiniUniqueness} the limit cylinder is unique. Here $\Sigma_{\tau}$ is the rescaled hypersurface at time $\tau$.

We suppose that the blowup point is the origin and the blowup time is $T>0,$
and suppose that the limit cylinder is $\mathbb{R}^{3}\times \mathbb{S}^{1}_{\sqrt{2}},$ parameterized by, 
\begin{align}\label{eq:limitCylin}
\left[
\begin{array}{ccc}
y\\
\sqrt{2}cos\theta\\
\sqrt{2}sin\theta
\end{array}
\right],\ y:=(y_{1},\ y_2,\  y_{3})^{T}\in \mathbb{R}^3.
\end{align}

Then we have that, in a (possibly shrinking) neighborhood of the singularity, MCF takes the form
\begin{align}\label{eq:origPara}
\Psi(z,t)= \left[
\begin{array}{ccc}
z\\
u(z, \theta, t)\cos\theta\\
u(z,\theta,t) \sin \theta
\end{array}
\right],\ z\in \mathbb{R}^3, \ |z|\leq c(t)\ \text{for some}\ c(t)>0,
\end{align}
where $u$ is periodic in $\theta$.

Then we define a rescaled MCF by rescaling $u$ into $v$, specifically
\begin{align}
u(z,\theta,t)=&\sqrt{T-t} \ v(y,\theta, \tau)\label{eq:rescale}
\end{align} here $y$ and $\tau$ are the rescaled space and time variables defined as
\begin{align}
y:=\frac{z}{\sqrt{T-t}},\ \tau:=-ln(T-t).\label{eq:scaling}
\end{align} 
Then the part parametrized in \eqref{eq:origPara} becomes
\begin{align}\label{eq:rescaledOri}
\sqrt{T-t}\left[
\begin{array}{c}
y\\
v(y,\theta,\tau)cos\theta\\
v(y,\theta,\tau) sin\theta
\end{array}
\right].
\end{align}

In the present paper we consider the following region:
\begin{align}
|y|\leq \Omega(\tau),\ \text{and}\ \tau\geq \xi_0 \ \text{with}\ \xi_0\ \text{sufficiently large}
\end{align} with $\Omega$ defined as
\begin{align}
\Omega(\tau):=\sqrt{100 \ln \tau+9 (\tau-\xi_0)^{\frac{11}{10}}}.\label{eq:defOmega}
\end{align}

The result is the following:  
\begin{theorem}\label{THM:TwoReg1}
Suppose the condition \eqref{eq:generic} holds, the blowup point is the origin and the blowup time is $T$, and the limit cylinder is the one parametrized by \eqref{eq:limitCylin}.

Then there exists a large time $\xi_0$ such that for $\tau\geq \xi_0$ and for $|y|\leq \Omega(\tau)$, the rescaled MCF is parametrized by
\begin{align}\label{eq:representation}
\sqrt{T-t} \left[
\begin{array}{ccc}
y\\
v(y, \theta,\tau)cos\theta\\
v(y, \theta,\tau)sin\theta
\end{array}
\right],
\end{align} with $v,\ y,\ \tau$ defined in \eqref{eq:rescale} and \eqref{eq:scaling}, and, up to a rotation $y\rightarrow U y$, one has that, for $|y|\leq \Omega(\tau),$
\begin{align}
v(y, \theta,\tau)=\sqrt{\frac{2+y^{T}B(\tau) y}{2a(\tau)}}+\eta(y, \theta,\tau),\label{eq:decPr}
\end{align} where, for some $C>0,$ the parameter $a$ satisfies the estimate
\begin{align}
|a(\tau)-\frac{1}{2}|\leq C\tau^{-1},
\end{align} and the symmetric $3\times 3$ matrix $B$ satisfies the estimates
\begin{align}
B(\tau)=\tau^{-1}\left[
\begin{array}{lll}
b_1&0&0\\
0&b_2&0\\
0&0&b_3
\end{array}
\right]+\mathcal{O}(\tau^{-2}),\ \text{with}\ b_{k}=0\ \text{or}\ 1,\  k=1,2,3,
\end{align}
and
the function $\eta$ satisfies the estimates
\begin{align}
\|e^{-\frac{1}{8}|y|^2}1_{\Omega}\eta(\cdot,\tau)\|_2\leq C \tau^{-2},\label{eq:estEta}
\end{align} and 
\begin{align}\label{eq:InftyEst}
\begin{split}
\|\langle y\rangle^{-3}1_{\Omega}\partial_{\theta}^{l}\eta(\cdot,\tau)\|_{\infty}\leq &C (\tau^{-2}+\Omega^{-4}),\ l=0,1,2,\\
\|\langle y\rangle^{-2}1_{\Omega}\partial_{\theta}^{l}\nabla_{y}\eta(\cdot,\tau)\|_{\infty}\leq &C \Omega^{-3},\ l=0,1,\\
\|\langle y\rangle^{-1}1_{\Omega}\nabla_{y}^{k}\eta(\cdot,\tau)\|_{\infty}\leq &C \Omega^{-2},\ |k|=2.
\end{split}
\end{align}

\end{theorem}
The theorem will be reformulated in Section \ref{sec:refor}, and proved in Section \ref{sec:ProTwoReg1}.

Here $1_{\leq \Omega}$ is the Heaviside function defined as
\begin{align}
1_{\leq \Omega}(y)=\left[
\begin{array}{ll}
1\ &\ \text{if}\ |y|\leq \Omega,\\
0\ & \ \text{otherwise.}
\end{array}
\right.
\end{align}

Next we discuss the original MCF. 

\begin{theorem}\label{prob:converge} Suppose that in \eqref{eq:Best} of Theorem \ref{THM:TwoReg} we have 
\begin{align}
b_1=b_2=b_3=1 .
\end{align}
Then there exist constants $\epsilon_1, \ \epsilon_2>0$, such that when $0\leq T-t\leq \epsilon_1$ and $|z|\leq \epsilon_2$, the manifold is parameterized as in \eqref{eq:origPara}, and $u$ is continuous in all the variables. 

And in the same space and time intervals, except at $(z,t)=(0,T)$ where $u(0,\theta,T)=0$, the following two statements hold:
\begin{itemize}
\item[(A)]
the function $u$ is positive, smooth in all variables, and is strictly decreasing in $t,$
\item[(B)] For any fixed $N$, there 
 there exists some positive constant $\delta(|z|,t)$ satisfying $$\lim_{|z|\rightarrow 0,\ t\rightarrow T}\delta(|z|,t)= 0,$$ such that for $|m|+n=1,2,\cdots, N,$
\begin{align}
u^{|m|-1}\big|\partial_{\theta}^{n} \nabla_{z}^{m}u(z, \theta,t)\big|\leq \delta(|z|,t),\label{eq:oriz}
\end{align} 
and the normal direction ${\bf{n}}(z,\theta,t)$ satisfies 
\begin{align}
\big|{\bf{n}}(z,\theta,t)-(0,0,0, -sin\theta, cos\theta)^{T}\big|\leq \delta(|z|,t).\label{eq:oriz3}
\end{align}

And in the considered set the hypersurface is mean convex.
\end{itemize}

\end{theorem}
The theorem will be proved in Section \ref{sec:flowThr}. The fact that the manifold is mean convex is directly implied by the estimates in \eqref{eq:oriz}, to be shown below. For the importance of mean convexity, namely non-fattening, we refer to the results in \cite{HerWhite2017}.

Now we prove that a neighborhood of singularity is mean convex. Observe that if we rescale the flow such that $u(z,\theta,t)$ is rescaled into a new function $u_1$ defined by the identity
\begin{align}
u(z,\theta,t)=\lambda u_1(\lambda^{-1} z, \theta, \lambda^{-2}t), 
\end{align} then we have that, since these functions in \eqref{eq:oriz} are ``rescaling invariant",
\begin{align}
u_1^{|m|-1}|\partial_{\theta}^{n} \nabla_{z}^{m}u_1(z, \theta,t)|=u^{|m|-1}|\partial_{\theta}^{n} \nabla_{x}^{m}u(x, \theta,s)|_{x=\lambda^{-1}z, s=\lambda^{-2}t}\leq \delta(|x|,s).\label{eq:oriz4}
\end{align} Now we choose $\lambda$ to make $u_1(z_0,\theta_0,t_0)=1$ for some $z_0,\ \theta_0$ and $t_0$, then by the estimates in \eqref{eq:oriz4}, we find that an increasingly large (as $\lambda\rightarrow 0$) neighborhood of $(z_0,\theta_0)$ will become increasingly resemble to a cylinder with radius $1$, up to any order of derivatives. Hence this part of the hypersurface is mean convex, i.e. has positive mean curvature. 

The flow through singularities will be addressed in subsequent papers. 

\section{Reformulation of Theorem \ref{THM:TwoReg1}}\label{sec:refor}

In what follows we derive equations for the function $u$ in \eqref{eq:origPara}.

Recall that we suppose the blowup point is the origin, and the blowup time is $T$, and the limit cylinder is the one parametrized by \eqref{eq:limitCylin}.

Then for $t<T,$ there exists some $\epsilon(t)>0$ such that in the region $|z|\leq \epsilon(t)$, the manifold can be parameterized as in \eqref{eq:origPara}.
And the function $u$ satisfies the parabolic differential equation, by the mean curvature equation \eqref{eq:MeanCur} and the results in \cite{MR1770903},
\begin{align}
\begin{split}\label{eq:MCF}
\partial_{t}u=&\frac{1}{1+|\nabla_x u|^2+(\frac{\partial_{\theta}u}{u})^2}\sum_{k=1}^{3}\Big[1+|\nabla_x u|^2- (\partial_{x_k}u)^{2}+(\frac{\partial_{\theta}u}{u})^2\Big]\partial^{2}_{x_k}u\\
&+u^{-2}\frac{1+|\nabla_x u|^2}{1+|\nabla_x u|^2+(\frac{\partial_{\theta}u}{u})^2}\partial_{\theta}^2 u+
u^{-2}\frac{2\partial_{\theta}u}{1+|\nabla_x u|^2+(\frac{\partial_{\theta}u}{u})^2}\sum_{l=1}^{3}\partial_{x_l}u\partial_{x_l}\partial_{\theta}u\\
&+\frac{1}{1+|\nabla_x u|^2+(\frac{\partial_{\theta}u}{u})^2}\frac{(\partial_{\theta}u)^2}{u^3}-\sum_{i\not= j}\frac{\partial_{x_i} u \partial_{x_j} u}{1+|\nabla_x u|^2+(\frac{\partial_{\theta}u}{u})^2}\partial_{x_i}\partial_{x_j}u-\frac{1}{u}.
\end{split}
\end{align}

Now we rescale solution as in \eqref{eq:rescale}, and derive an equation for the function $v$, 
\begin{align}
\partial_{\tau}v=\Delta_{y} v+v^{-2}\partial_{\theta}^2 v-\frac{1}{2}y\cdot\nabla_{y}v+\frac{1}{2}v-\frac{1}{v}+N_1(v)\label{eq:scale1}
\end{align} with $N_1(v)$ defined as
\begin{align}
\begin{split}\label{eq:defFu}
N_1(v):=&-\frac{\sum_{k=1}^{3}(\partial_{y_k}v)^{2} \partial^{2}_{y_k}v}{1+|\nabla_y v|^2+(\frac{\partial_{\theta}v}{v})^2}-\sum_{i\not= j}\frac{\partial_{y_i} v \partial_{y_j} v}{1+|\nabla_y v|^2+(\frac{\partial_{\theta}v}{v})^2}\partial_{y_i}\partial_{y_j}v\\
&+v^{-2}\frac{2\partial_{\theta}v}{1+|\nabla_y v|^2+(\frac{\partial_{\theta}v}{v})^2}\sum_{l=1}^{3}\partial_{y_l}v\partial_{y_l}\partial_{\theta}v-v^{-2} \frac{(v^{-1}\partial_{\theta}v)^{2} \partial^{2}_{\theta}v}{1+|\nabla_y v|^2+(\frac{\partial_{\theta}v}{v})^2}\\
&+\frac{1}{1+|\nabla_y v|^2+(\frac{\partial_{\theta}v}{v})^2}\frac{(\partial_{\theta}v)^2}{v^3}.
\end{split}
\end{align}

Now we present the general strategy in proving Theorem \ref{THM:TwoReg1}. For the details we refer to the proof of Theorem \ref{THM:TwoReg1} in Section \ref{sec:ReforWeightLInf}.

We will proceed by bootstrap arguments. Specifically, under assumption of some regularity estimates, specifically \eqref{eq:SmCon} below, in a time interval $[\xi_0,\tau_1]$ with $\tau_1>\xi_0$, we prove Theorem \ref{THM:TwoReg} below. And the estimates in Theorem \ref{THM:TwoReg}, in turn, will make Lemma \ref{LM:regularity} applicable in a larger interval $[\xi_0,\ \tau_2]$ with $\tau_2>\tau_1$. The estimates in Lemma \ref{LM:regularity} imply \eqref{eq:SmCon}, and hence in turn, makes Theorem \ref{THM:TwoReg}, hold in a larger interval.

To initiate the bootstrap arguments, we need the estimates from the previous paper \cite{GZ2017}.

In the first part of the bootstrap arguments, we prove the following results.
\begin{theorem}\label{THM:TwoReg}

Suppose that $\xi_0$ in \eqref{eq:defOmega} is a sufficiently large constant. 

There exists a small constant $\delta$, such that if in the region $|y|\leq (1+\epsilon)\Omega(\tau)$ and in the time interval $\tau\in [\xi_0,\ \tau_1]$ the following estimates hold, for $ m, \ |k|+l=1,\cdots,5$, and $|k|\geq 1,$
\begin{align}\label{eq:SmCon}
\Big|\frac{v(\cdot,\tau)}{\sqrt{2+\tau^{-1}y^{T}\tilde{B}y}}-1\Big|,\ v^{-1}(\cdot,\tau)|\partial_{\theta}^m v(\cdot,\tau)|,\ |\nabla_{y}^{k}\partial_{\theta}^{l} v(\cdot,\tau)|\leq \delta,
\end{align} where $\tilde{B}$ is a $3\times 3$ diagonal matrix 
\begin{align}
\tilde{B}=\left[
\begin{array}{lll}
b_{1}&0&0\\
0&b_2&0\\
0&0&b_3
\end{array}
\right],\ b_{k}=1\ \text{or}\ 0,\ k=1,2,3,
\end{align}
then the following statements hold in the time interval $[\xi_0,\ \tau_1]$.

There exist unique parameters $a,\ \alpha_1$, $\alpha_2$, a $3\times 3$ symmetric real matrix $B$ and 3-dimensional vectors $\vec\beta_k,\ k=1,2,3,$ such that
\begin{align}
\begin{split}
v(y,\ \theta,\tau)=V_{a(\tau), B(\tau)}(y)+\vec\beta_1(\tau)\cdot y&+{\vec\beta}_2(\tau)\cdot y cos\theta + {\vec\beta}_3(\tau)\cdot y sin\theta\\
& +\alpha_1(\tau) cos\theta +\alpha_2(\tau) sin\theta+w(y,\ \theta,\tau),\label{eq:decomVToW}
\end{split}
\end{align} and 
the function $e^{-\frac{1}{8}|y|^2}\chi_{\Omega}w$ satisfies the following orthogonality conditions,
\begin{align}
\begin{split}\label{eq:orthow}
e^{-\frac{1}{8}|y|^2}w\chi_{\Omega}\perp  &\ e^{-\frac{1}{8} |y|^2},\ e^{-\frac{1}{8} |y|^2} cos\theta,\ e^{-\frac{1}{8} |y|^2} sin\theta,  \\ 
&e^{-\frac{1}{8} |y|^2} y_k,\ e^{-\frac{1}{8} |y|^2}(\frac{1}{2}y_k^2-1),\ e^{-\frac{1}{8} |y|^2} y_k cos\theta,\ e^{-\frac{1}{8} |y|^2} y_k sin\theta,\ k=1,2,3, \\
&\  e^{-\frac{1}{8} |y|^2} y_m y_n, \ m\not=n, \ \ m,n=1,2,3.
\end{split}
\end{align} Here $V_{a,B}$ and $\chi_{\Omega}$ are two functions to be defined in \eqref{eq:defVaB} and \eqref{eq:defChi3} below. 

Up to a unitary rotation, $y\rightarrow U y$, the $3\times 3$ real symmetric matrix $B$ is ``almost" semi-positive definite,
\begin{align}\label{eq:Best}
B(\tau)=\tau^{-1}\left[
\begin{array}{ccc}
b_1&0&0\\
0&b_2&0\\
0&0&b_3
\end{array}
\right]+\mathcal{O}(\tau^{-2}),\ \text{and}\ b_{k}=0\ \text{or}\ 1,\ k=1,2,3, 
\end{align} 
The other parameters and vectors satisfy the estimates,  for some $C>0,$
\begin{align}\label{eq:10parame}
|a(\tau)-\frac{1}{2}|\leq C\tau^{-1},
\end{align} and
\begin{align}
|\vec\beta_1(\tau)|\leq C \tau^{-2},\ |\vec\beta_2(\tau)|,&\ |\vec\beta_3(\tau)|,\ |\alpha_1(\tau)|,\ |\alpha_2(\tau)|\leq C \tau^{-3}.\label{eq:betaA}
\end{align}
And they satisfy the equations,
\begin{align}\label{eq:scalarEqn}
\begin{split}
|\frac{d}{d\tau}B+ B^{T}B|\leq &C\tau^{-3}\\
|\Big(\frac{1}{a}\frac{d}{d\tau}-2\Big)\Big(a-\frac{1}{2}-\frac{1}{2}(b_{11}+b_{22}+b_{33})\Big)|\leq &C\tau^{-2},\\
|\frac{d}{d\tau} \vec{\beta}_1-a(1+\mathcal{O}(|B|))\vec\beta_1|\leq &C\tau^{-3},\\
|\frac{d}{d\tau}\vec{\beta}_2|,\ |\frac{d}{d\tau}\vec{\beta}_3|,\ 
|\frac{d}{d\tau}\alpha_1-\frac{1}{2} \alpha_1|,\ |\frac{d}{d\tau}\alpha_2-\frac{1}{2}\alpha_2|\leq &C\tau^{-3}.
\end{split}
\end{align}
The remainder $w$ satisfies the estimates, in weighted $L^2$ norms,
\begin{align}
\sum_{|k|+l=0,1,2}\|e^{-\frac{1}{8}|y|^2}\nabla_{y}^{k}\partial_{\theta}^{l}\big(\chi_{\Omega}w(\cdot,\tau)\big)\|_2\leq C &\tau^{-2}, \label{eq:expWei}
\end{align}
and in the weighted $L^{\infty}-$norms, for some constant $\kappa(\epsilon)$ to be defined in \eqref{eq:defKappa} below,
\begin{align}
\|\langle y\rangle^{-3}\partial_{\theta}^{m}\chi_{\Omega} w(\cdot,\tau)\|_{L^{\infty}}\leq &C \big(\tau^{-2}+\kappa(\epsilon) \Omega^{-4}\big), \ \text{with}\ m=0,1,2,\label{eq:weighL1}\\
\|\langle y\rangle^{-2}\nabla_y\partial_{\theta}^{m} \chi_{\Omega}w(\cdot,\tau)\|_{L^{\infty}}\leq &C \kappa(\epsilon) \Omega^{-3},\ \text{with}\ m=0,1.\label{eq:weighL2}\\
\|\langle y\rangle^{-1}\nabla_y^{l} \chi_{\Omega}w(\cdot,\tau)\|_{L^{\infty}}\leq &C \kappa(\epsilon) \Omega^{-2},\ \text{with}\ |l|=2.\label{eq:weighL3}
\end{align} 
\end{theorem}

Here $V_{a,B}$ in \eqref{eq:decomVToW} is a function defined as, for $a\in \mathbb{R}^{+}$ and $3\times 3$ symmetric matrix $B$,
\begin{align}
V_{a,B}(y):=\sqrt{\frac{2+y^{T}B y}{2a}}.\label{eq:defVaB}
\end{align} Before defining the cutoff function $\chi_{\Omega}$, we define a smooth, spherically symmetric cutoff function $\chi$,
\begin{align}\label{eq:defChi3}
\chi(z)=\chi(|z|)=\ \Big[
\begin{array}{lll}
1,\ \text{if}\ |z|\leq 1,\\
0,\ \text{if}\ |z|\geq 1+\epsilon.
\end{array}
\
\end{align} We require it is decreasing in $|z|$, and there exist constants $M_k=M_k(\epsilon), \ k=0, 1,\cdots, 5,$ such that for any $z$ satisfying $0\leq 1+\epsilon-|z|\ll 1$, $\chi$ satisfies the estimates
\begin{align}
\begin{split}\label{eq:properties}
\frac{d^{k}}{d|z|^{k}}\chi(|z|)=&M_k\Big||z|-1-\epsilon\Big|^{20-k}+\mathcal{O}\big(\Big||z|-1-\epsilon\Big|^{21-k}\big),\ k=0,1,2,3,4.
\end{split}
\end{align}

Such a function is easy to construct, we skip the details here. 

Now we define the cutoff functions $\chi_{\Omega}$, for any $\Omega>0,$ as 
\begin{align}
\chi_{\Omega}(y):=\chi(\frac{y}{\Omega}).\label{eq:reCutoff}
\end{align}

The constant $\kappa(\epsilon)$ is defined to control terms produced by the cutoff function $\chi$,
\begin{align}
\kappa(\epsilon):=\sum_{k=1}^{5}\sup_{|z|}\Big|\frac{d^{k}}{d|z|^{k}} \chi(|z|)\Big|+\sup_{z}\sum_{|l|=1,2,3}|\chi^{-\frac{3}{4}}(z)\ \nabla^{l}_{z}\chi |<\infty.\label{eq:defKappa}
\end{align} To justify that $\kappa(\epsilon)$ is finite, we have that the first term is finite since the function is smooth; the second is also finite since the conditions in \eqref{eq:properties} imply that $\big|\nabla^{l}_{z}\chi(|z|)\big|$, $|l|\leq 3,$ approach to zero faster than $\chi^{\frac{3}{4}}(|z|)$ as $|z| \rightarrow 1+\epsilon,$ and recall that $\chi(|z|)>0$ when $|z|< 1+\epsilon.$

Next we formulate the second part of bootstrapping argument.

To make Theorem \ref{THM:TwoReg} applicable one needs to verify its conditions \eqref{eq:SmCon}.  
For that purpose, we need the following results, recall that the definition of $\Omega$ in \eqref{eq:defOmega} depends on $\xi_0,$
\begin{lemma}\label{LM:regularity}
Suppose that when $|y|\leq \Omega(\tau)$ and $\tau\in [\xi_0, \tau_1]$ with $\tau_1\geq \xi_0+20$, the graph function $v$ of the rescaled MCF satisfies the estimates, 
\begin{align}
\Big|v(\cdot,\tau)-\sqrt{2+\tau^{-1}y^{T}\tilde{B}y}\Big|\leq \Omega^{-\frac{2}{5}}(\tau),\ \text{and}\ |\nabla_{y}^{k}\partial_{\theta}^{l} v(\cdot,\tau)|\leq \Omega^{-\frac{9}{10}}(\tau),\ \ |k|+l=1,2,\label{eq:PointSmall}
\end{align} 
and $\tilde{B}$ takes the form
\begin{align}
\tilde{B}=\left[
\begin{array}{lll}
b_{1}&0&0\\
0&b_2&0\\
0&0&b_3
\end{array} 
\right],\ b_{k}=1\ \text{or}\ 0,\ k=1,2,3.
\end{align}

Then there exists a constant $\delta=\delta(\xi_0)>0$,  such that 
\begin{align}
\delta(\xi_0)\rightarrow 0,\ \text{as}\ \xi_0\rightarrow \infty,
\end{align} and the following estimates hold.

There exist some constant $C$, independent of $\delta$, and some small constant $\kappa=\kappa(\delta)>0$ such that at time
$
\tau=\tau_1+10\kappa$ and in the region 
\begin{align}
 |y|\leq (1+5\kappa)\Big(\Omega(\tau)-C\sup_{|y|\leq \Omega(\tau)}\sqrt{2+\tau^{-1}y^{T}\tilde{B}y}\Big), \label{eq:SmoExte}
\end{align}
$v$ satisfies the estimates, for $m=1,2,\cdots,5$ and $|k|+l=1,2,\cdots,5$ and $|k|\geq 1$,
\begin{align}
\Big|\frac{v(\cdot,\tau)}{\sqrt{2+\tau^{-1}y^{T}\tilde{B}y}}-1\Big|,\ v^{-1}\Big|\partial_{\theta}^m v\Big|,\ v^{|k|-1}\Big|\nabla_{y}^{k}\partial_{\theta}^{l} v(\cdot,\tau)\Big|\leq \delta.\label{eq:extenSmooth}
\end{align} 
\end{lemma}

The lemma will be proved in Section \ref{sec:regul}.

Assuming Theorem \ref{THM:TwoReg} and Lemma \ref{LM:regularity}, we are ready to prove Theorems \ref{THM:TwoReg1} and \ref{prob:converge}.

\section{Proof of Main Theorem \ref{THM:TwoReg1}}\label{sec:ProTwoReg1}
As said earlier, we plan to prove Theorem \ref{THM:TwoReg1} by bootstrapping arguments. To start it, we need to verify certain estimates for the interval $\tau\in [\xi_0, \ \xi_0+40]$ with $\xi_0$ being the initial time and chosen to be large. They will be provided by our previous paper \cite{GZ2017}.

Recall that in \cite{GZ2017} we defined a function $R(\tau)$ as, for some $\tau_0\gg 1,$
\begin{align}
R(\tau):=\sqrt{\frac{26}{5}\ln\tau+100 \ln(1+\tau-\tau_0)}.\label{eq:defRTau}
\end{align} And we derived estimates for $v$ when $|y|\leq R(\tau).$

The definition of $\Omega(\tau)$ in \eqref{eq:defOmega} implies that, $$R(\tau)\geq \Omega(\tau)\ \text{for}\ \tau\in [\xi_0, \xi_0+40],\ \text{if}\ \xi_0\gg \tau_0.$$ Thus the results proved in \cite{GZ2017} is applicable in this interval, most importantly they fulfill the conditions \eqref{eq:PointSmall} to make Lemma \ref{LM:regularity} applicable. Now we fix the constant $\epsilon$ in Theorem \ref{THM:TwoReg} to be the constant $\kappa$ in Lemma \ref{LM:regularity}, choose $\xi_0$ to be sufficiently large such that $\delta$ in Lemma \ref{LM:regularity} is sufficiently small. Then by Lemma \ref{LM:regularity} the conditions in \eqref{eq:SmCon} hold in an interval $[\xi_0,\xi_0+40+\kappa_1]$ for some $\kappa_1>0$. This makes Theorem \ref{THM:TwoReg} applicable in the same interval.

Results in Theorem \ref{THM:TwoReg} fulfills the condition \eqref{eq:PointSmall}, and hence makes Lemma \ref{LM:regularity} applicable in the same interval. Lemma \ref{LM:regularity}, in turn, makes the condition \eqref{eq:SmCon}, hence Theorem \ref{THM:TwoReg} hold in an even larger interval $[\xi_0,\ \xi_{0}+40+\kappa_2]$ with $\kappa_2> \kappa_1$. 

By bootstrapping these arguments, we find that the results in Theorem \ref{THM:TwoReg} hold in the interval $[\xi_0,\ \infty)$.

We complete the proof of Theorem \ref{THM:TwoReg1} by taking the estimates from Theorem \ref{THM:TwoReg}.


\section{Proof of Main Theorem \ref{prob:converge} }\label{sec:flowThr}

To facilitate later discussions we fix a large time $\tau_1$ of rescaled MCF, then define a new MCF. The property we exploit is that MCF is scaling invariant, in the present situation it means that if $u(x,\theta,t)$ is a solution to \eqref{eq:MCF}, then so is $\lambda^{-1}u(\lambda x,\theta,\lambda^2 t)$ for any $\lambda>0.$ 

We define the new flow such that the part of the original flow parametrized by $u$ is parametrized by a new function $p$, from the time $t_1=t(\tau_1),$ specifically
 \begin{align}\label{eq:flow2}
 \left[
\begin{array}{ccc}
z\\
p(z, \theta, s)\cos\theta\\
p(z,\theta,s) \sin \theta
\end{array}
\right],
\end{align}
with $p(z,\theta,s)$ defined in terms of the functions $u$, and hence of the function $v$ through the identity \eqref{eq:rescale},
\begin{align}
\begin{split}\label{eq:restart}
p(z,\theta,s):
=&\frac{1}{\lambda} u(\lambda z, \theta,  \lambda^2 s+t_1)\\
=&\frac{\sqrt{1-\tau_1^{\frac{1}{10}}s}}{\tau_{1}^{\frac{1}{20}}}\ v\big(\frac{\tau_1^{\frac{1}{20}} z}{\sqrt{1-\tau_1^{\frac{1}{10}}s}},\theta, -\ln(1-\tau_1^{\frac{1}{10}}s)+\tau_1\big),
\end{split}
\end{align} where $s$, $t_1$ $\lambda$ are defined as \begin{align}
s:=\frac{t-t_1}{\lambda^2}, \ t_1:=t(\tau_1),\ \lambda:=\tau^{\frac{1}{20}}_1\sqrt{T-t_1}.
\end{align} Recall that $\tau$ is defined by a bijection $\tau=-\ln (T-t)$, $t_1$ is the time $t$ when $\tau(t)=\tau_1.$
The blowup time here is $\tau_{1}^{-\frac{1}{10}}$ since as $s\rightarrow \tau_{1}^{-\frac{1}{10}}$, $\tau=-ln(1-\tau_{1}^{\frac{1}{10}}s)+\tau_1\rightarrow \infty.$

We plan to study this flow, or equivalently the function $p$, in the small positive time interval $s\in [0,\tau_1^{-\frac{1}{10}}]$ and $z $ in the set $|z|\in \big[\sqrt{2}\tau_1^{\frac{1}{2}}-\tau_1^{\frac{1}{4}}, \sqrt{2}\tau_1^{\frac{1}{2}}+\tau_1^{\frac{1}{4}}\big]$. Then we transfer the estimates to the corresponding part for $u$ through the second identity in \eqref{eq:restart}, and obtain the desired estimates.

We derive estimates for the function $p$ by those for $v$, and the latter is estimated in detail in Theorem \ref{THM:TwoReg}. Provided that $\tau$ is sufficiently large, then in the region $|y|\leq 3\tau^{\frac{1}{2}+\frac{1}{20}}$, \eqref{eq:weighL1} implies that, 
\begin{align}
|w(y,\sigma,\tau)|\leq \langle y\rangle^{3} \|\langle y\rangle^{-3}w(\cdot,\tau)\|_{\infty}\lesssim &
\langle y\rangle^{3} \tau^{-2}\lesssim \tau^{-\frac{7}{20}},
\end{align} and similarly from the estimates in \eqref{eq:weighL2} and \eqref{eq:weighL3},
\begin{align}\label{eq:higherDer}
\sum_{k+|l|=1}|\partial_{\theta}^{k}\nabla_{y}^{l}w(\cdot,\tau)| \lesssim \tau^{-\frac{1}{2}},\ \ 
  \sum_{k+|l|=2}|\partial_{\theta}^{k}\nabla_{y}^{l}w(\cdot,\tau)| \lesssim  \tau^{-\frac{1}{2}}.
\end{align}
These, together with the decomposition of $v$ and the estimates in \eqref{eq:Best}-\eqref{eq:betaA}, make,
\begin{align}
\Big|v(y,\theta,\tau)-\sqrt{2+\tau^{-1}|y|^2}\Big|\lesssim \tau^{-\frac{7}{20}},\ \sum_{k+|l|=1,2}|\partial_{\theta}^{k}\nabla_{y}^{l}v(\cdot,\tau)| \lesssim \tau^{-\frac{1}{2}},\ \text{for} \ |y|\leq  3\tau^{\frac{1}{2}+\frac{1}{20}}.\label{eq:estV}
\end{align}

At $s=0$ and $|z|\in \big[\sqrt{2}\tau_1^{\frac{1}{2}}-\tau_1^{\frac{1}{4}}, \sqrt{2}\tau_1^{\frac{1}{2}}+\tau_1^{\frac{1}{4}}\big]$, the last identity in \eqref{eq:restart} makes the estimates on $v$ in \eqref{eq:estV} applicable, thus
\begin{align}
|p(z,\theta,0)-\sqrt{2}|\lesssim \tau_1^{-\frac{1}{20}},
\end{align}
and 
\begin{align}
|\nabla_{z}^{k}\partial_{\theta}^{l}p(z, \theta,0)| \leq \tau_1^{-\frac{2}{5}},\ |k|+l=1,2.\label{eq:oriz2}
\end{align}

Now we consider the negative time interval $s\in [-1,0]$. This, by the identity in \eqref{eq:restart}, corresponds to the time interval $\tau\in \big[-ln(1+|s|\tau_1^{\frac{1}{10}})+\tau_1,\ \tau_1\big]$ for the rescaled MCF. The estimates on $v$ imply that, for $s\in[-1,0] $ and $|z|\in \big[\sqrt{2}\tau_1^{\frac{1}{2}}-\tau_1^{\frac{1}{4}}, \sqrt{2}\tau_1^{\frac{1}{2}}+\tau_1^{\frac{1}{4}}\big],$
\begin{align}
1\leq p\leq 10,\ | \nabla_{z}^{k}\partial_{\theta}^{l}p|\leq \tau_1^{-\frac{3}{10}},\ |k|+l=1,2.
\end{align} 

The estimates above make the techniques of local smooth extension applicable. This, together with the regularity estimate, and interpolation between the estimates of the derivatives, implies that in the set 
\begin{align}
s\in [0,\ \tau_1^{-\frac{1}{10}}]\ \text{and}\ |z|\in \big[\sqrt{2}\tau_1^{\frac{1}{2}}-\frac{1}{2}\tau_1^{\frac{1}{4}}, \sqrt{2}\tau_1^{\frac{1}{2}}+\frac{1}{2}\tau_1^{\frac{1}{4}}\big],
\end{align} we have
\begin{align}
|p(z,\theta,s)-\sqrt{2}|\leq C \tau_1^{-\frac{1}{20}}; \label{eq:fiGap}
\end{align} and for any fixed $N\in \mathbb{N}$, there exists a constant $\epsilon(\tau_1)$ satisfying $\displaystyle\lim_{\tau\rightarrow \infty}\epsilon(\tau)= 0$, such that
\begin{align*}
 |\nabla_{z}^{m}\partial_{\theta}^n p(z,\theta,s)|\leq \epsilon(\tau_1), \ 1\leq |m|+n\leq N,
\end{align*} and hence, there exists a constant $C_N$ such that
\begin{align}
p^{|m|-1}|\partial_{\theta}^{n}\nabla_{y}^{m}p(z,\theta,s)|\leq C_{N}\epsilon(\tau_1).\label{eq:invar}
\end{align}

These directly enable us to estimate the function $u$ through the second identity in \eqref{eq:restart}. Indeed, for any $z$ and $t$ satisfying 
\begin{align}
 \tau^{\frac{1}{20}}_1\sqrt{T-t(\tau_1)}\Big[\sqrt{2}\tau_1^{\frac{1}{2}}-\frac{1}{2}\tau_1^{\frac{1}{4}}\Big]\leq |z|&\leq \tau^{\frac{1}{20}}_1\sqrt{T-t(\tau_1)}\Big[\sqrt{2}\tau_1^{\frac{1}{2}}+\frac{1}{2}\tau_1^{\frac{1}{4}}\Big],\label{eq:setZ}
\end{align} and
\begin{align}\label{eq:setT}
t\in & [t(\tau_1),\ T],
\end{align}
we have that, for example, if $|k|=3$, then by \eqref{eq:invar},
\begin{align}
u^2|\nabla_{z}^{k}u(z, \theta,t)|=p^2\Big|\nabla_{x}^{k}p(x,\theta,s)\Big|_{z=\tau_1^{\frac{1}{20}}\sqrt{T-t(\tau_1)}x, \ s=t_1+\tau_1^{\frac{1}{10}}(T-t(\tau_1))}\leq C_{N}\epsilon(\tau_1) , \label{eq:invar2}
\end{align} and the other estimates in \eqref{eq:oriz} will be derived similarly. The estimate in \eqref{eq:oriz3} will be derived easily, hence will be skipped.

Until now we only proved part of Item B of Theorem \ref{prob:converge}, specifically we need to prove the estimates also hold for time $t\in [t_1, T]$ and for any $z$ satisfying $|z|\leq \sqrt{2}\tau^{\frac{1}{20}}_1\sqrt{T-t(\tau_1)} \tau_1^{\frac{1}{2}},$ which is obviously larger than that in \eqref{eq:setZ} and \eqref{eq:setT}.

It turns out that what is left can be obtained similarly and more easily. The reason is that for each fixed small $z\not=0$, by the rescaling $y=\frac{1}{\sqrt{T-t(\tau)}}z=e^{\frac{1}{2}\tau}z$ and that $e^{\frac{1}{2}\tau}$ grows much faster than $\tau$, it will become a $y$ with $|y|=2\tau^{\frac{1}{2}+\frac{1}{20}}$ at some time $t(\tau)\geq t(\tau_1).$ After the time $t(\tau)$ we apply the arguments above. 
Before $t(\tau)$, i.e. $|y|^2\leq 2\tau^{\frac{1}{2}+\frac{1}{20}}$, we derive the desired estimates for $u$ from that of $v$, for the latter we derived very detailed estimates, see \eqref{eq:decPr} and \eqref{eq:InftyEst}.
For example, for $|k|=2$, we have, by \eqref{eq:higherDer},
\begin{align}
u(z,\theta,t)\Big|\nabla_{z}^{k}u(z, \theta,t)\Big|=v(y,\theta,\tau) \Big|\nabla_{y}^{k}v(y,\theta,\tau(t))\Big|_{z=e^{-\frac{1}{2}\tau}y}\lesssim  \tau^{-\frac{2}{5}}.
\end{align} The other estimates in \eqref{eq:oriz}, for $|k|+l=1,2$ will be derived similarly. For $|k|+l>2$ we follow the steps in analyzing the regime $|y|\geq 2\tau^{\frac{1}{2}+\frac{1}{20}}$, namely defining a new flow as in \eqref{eq:restart} and choosing $\lambda$ according to the considered region, and extending the solution by a small time interval. The process is tedious, besides a sophisticated version will be used in the proof of Lemma \ref{LM:regularity}. We choose to skip the details here.

This completes the proof of Item B.

Now we prove Item A.

To see that $u(\cdot,t)$ is decreasing in $t$, we derive, from the equation for $u$ in \eqref{eq:MCF} and the estimates in \eqref{eq:oriz}, 
\begin{align}
\partial_{t}u=-\frac{1}{u}\Big(1+\mathcal{O}\big(\delta(|z|,t)
\big)\Big).
\end{align} By the smallness of $\delta$ and the positivity of $u$ we find the desired result, namely $u$ is decreasing in $t$.



\section{Proof of Theorem \ref{THM:TwoReg}}\label{sec:ReforWeightLInf}

We start with deriving an effective equation for $\chi_{\Omega}w$, almost identically to those in \cite{GZ2017}.

Plug the decomposition of $v$ in \eqref{eq:decomVToW} into \eqref{eq:scale1} to find
\begin{align}
\partial_{\tau}w=-Lw+F(B,a)+G(\vec\beta_1, \vec\beta_2,\ \vec\beta_3,\ \alpha_1,\ \alpha_2)+N_1(v)+N_2(\eta),\label{eq:wEqn}
\end{align}
where the linear operator $L$ is defined as
\begin{align}
L:=-\Delta_y+\frac{1}{2}y\cdot \nabla_{y}-V^{-2}_{a,B}\partial_{\theta}^2-\frac{1}{2}-V^{-2}_{a,B} ,\label{eq:aBeqn}
\end{align} 
the nonlinearity $N_2(\eta)$ is defined as
\begin{align}
\begin{split}\label{eq:defN2eta}
N_2(\eta):= &-v^{-1}+V_{a,B}^{-1}-V_{a,B}^{-2}\eta+\big(v^{-2}-V^{-2}_{a,B}\big)\partial_{\theta}^2 \eta\\
=&-V_{a,B}^{-2}v^{-1} \eta^2-v^{-2}V^{-2}_{a,B}(v+V_{a,B})\eta \partial_{\theta}^2 \eta,
\end{split}
\end{align}
the function $\eta$ is defined as
\begin{align}
\eta(y,\theta,\tau):=&{\vec\beta}_1(\tau)\cdot y +{\vec\beta}_2(\tau)\cdot y cos\theta + {\vec\beta}_3(\tau)\cdot y sin\theta +\alpha_1(\tau) cos\theta +\alpha_2(\tau) sin\theta 
+w(y,\theta,\tau),\label{eq:decomW}
\end{align} 
and
the function $F(B,\alpha)$ is defined as
\begin{align}
\begin{split}\label{eq:source}
F(B,\alpha)
:=& -\frac{y^{T}(\partial_{\tau}B +B^{T}B) y}{2 \sqrt{2a} \sqrt{2+y^{T}B y}}+\frac{1}{\sqrt{2a}\sqrt{2+y^{T}By}}\big[\frac{a_{\tau}}{a}+1-2a+b_{11}+b_{22}+b_{22}\big]\\
&+\frac{(y^{T}B^{T}B y) \ (y^{T}B y)}{2 \sqrt{2a} (2+y^{T}B y)^{\frac{3}{2}}}+\frac{a_{\tau}}{(2a)^{\frac{3}{2}}} \frac{y^{T}By}{\sqrt{2+y^{T}By}},
\end{split}
\end{align}
and $G(\vec\beta_1, \vec\beta_2,\ \vec\beta_3,\ \alpha_1,\ \alpha_2)$ is defined as
\begin{align}
G(\vec\beta_1, \vec\beta_2, \vec\beta_3, \alpha_1, \alpha_2)
:=&\big[\frac{2a}{2+y^{T}By}\vec\beta_1-\frac{d}{d\tau}\vec\beta_1\big]\cdot y-\frac{d}{d\tau}\vec\beta_2 \cdot y cos\theta-\frac{d}{d\tau}\vec\beta_3\cdot y sin\theta\nonumber\\
&+\big[\frac{1}{2}\alpha_1-\frac{d}{d\tau}\alpha_1\big] cos\theta+\big[\frac{1}{2}\alpha_2-\frac{d}{d\tau}\alpha_2\big] sin\theta.\nonumber
\end{align}

Impose the cutoff function $\chi_{\Omega}$ onto \eqref{eq:wEqn} and find
\begin{align}
\partial_{\tau}(\chi_{\Omega}w)=-L(\chi_{\Omega}w)+\chi_{\Omega}\Big[F(B,a)+G(\vec\beta_1, \vec\beta_2, \vec\beta_3, \alpha_1, \alpha_2)
+N_1(v)+N_2(\eta)\Big]+\mu(w),\label{eq:tildew3}
\end{align} here the term $\mu(w)$ is defined as 
\begin{align}
\mu(w):=\frac{1}{2}\big(y\cdot\nabla_{y}\chi_{\Omega}\big)w+\big(\partial_{\tau}\chi_{\Omega}\big)w-\big(\Delta_{y}\chi_{\Omega}\big)w-2\nabla_{y}\chi_{\Omega}\cdot  \nabla_{y}w.\label{eq:Tchi3}
\end{align}

In what follows we prove \eqref{eq:orthow}-\eqref{eq:expWei} of Theorem \ref{THM:TwoReg}. Here instead of going through the lengthy process as in \cite{GZ2017}, we would like to argue that, loosely speaking, we can just take the estimates in  \cite{GZ2017}, since the estimates in the present situation are only different from those in \cite{GZ2017} by an order of $\mathcal{O}(\tau^{-20})$. 

To justify this, we recall that to prove the existence and uniqueness of the decomposition in \eqref{eq:decomVToW}, we look for $a$, $B$, $\vec\beta_k,\ k=1,2,3,$ and $\alpha_l,\ l=1,2,$ to fulfill the orthogonality conditions \eqref{eq:orthow},
\begin{align}
\Big\langle e^{-\frac{1}{8}|y|^2}\chi_{\Omega}\Big[v-\big(V_{a, B}+\vec\beta_1\cdot y+{\vec\beta}_2\cdot y cos\theta + {\vec\beta}_3\cdot y sin\theta +\alpha_1 cos\theta +\alpha_2 sin\theta\big)\Big],\ e^{-\frac{1}{8}|y|^2}g_{k}\Big\rangle=0
\end{align}  where $e^{-\frac{1}{8}|y|^2}g_{k}, \ k=1,2,\cdots, 18,$ are the 18 functions listed in \eqref{eq:orthow}. Hence $g_k=P(y) e^{im\theta}$ where $P(y)$ is a polynomial with degree less than or equal to $2$, $m=-1,\ 0,\ 1.$ We can prove the existence by a standard fixed point argument.

While in the previous paper \cite{GZ2017} the condition is
\begin{align}
\Big\langle e^{-\frac{1}{8}|y|^2}\chi_{R}\Big[v-\big(V_{a, B}+\vec\beta_1\cdot y+{\vec\beta}_2\cdot y cos\theta + {\vec\beta}_3\cdot y sin\theta +\alpha_1 cos\theta +\alpha_2 sin\theta\big)\Big],\ e^{-\frac{1}{8}|y|^2}g_{k}\Big\rangle=0.
\end{align} where $\chi_{R}(y)=\chi(\frac{y}{R})$ and $R(\tau)$ is defined in \eqref{eq:defRTau}. The two conditions are only different by the cutoff functions. Since the function $e^{-\frac{1}{8}|y|^2}$ decays rapidly, and $v$ grows modestly, we have that the differences between two sets of parameters is of an order $\tau^{-20}$, as claimed. These together with the results in \cite{GZ2017} imply the desired estimates in \eqref{eq:orthow}-\eqref{eq:betaA} and \eqref{eq:expWei}.

Similarly every equation in \eqref{eq:scalarEqn} is derived by taking inner product $\langle e^{-\frac{1}{8}|y|^2}\eqref{eq:tildew3},\ e^{-\frac{1}{8}|y|^2}g_k\rangle$. Together with the analysis above, we see that the differences between those proved in \cite{GZ2017} and the present are of order $\tau^{-20}$, hence is negligibly small.

Next we prove \eqref{eq:weighL1}-\eqref{eq:weighL3} of Theorem \ref{THM:TwoReg}.

Decompose $w$ into four parts, according to the frequencies in $\theta$, 
\begin{align}
w(y,\theta,\tau)=w_0(y,\tau)+e^{i\theta} w_{1}(y,\tau)+e^{-i\theta} w_{-1}(y,\tau)+P_{\theta, \geq 2} w(y,\theta,\tau)\label{eq:fourParts}
\end{align}
where the functions $w_m,\ m=-1,0,1$ are defined as,
\begin{align} 
w_m(y,\tau):=\frac{1}{2\pi}\langle w(y,\cdot,\tau),\ e^{im\theta} \rangle_{\theta}.
\end{align}
Here $P_{\theta, \geq 2}$ is the orthogonal projection onto the subspace orthogonal to $1, \ e^{\pm i\theta}:$
for any smooth function $f(\theta)=\displaystyle\sum_{n=-\infty}^{\infty}e^{in\theta}f_{n}$,
\begin{align*}
P_{\theta,\geq 2}f(\theta)=\sum_{|n|\geq 2}e^{in\theta}f_{n}.
\end{align*}

The reason in decomposing $w$ is that we will apply different techniques to estimate these components. Specifically, we will estimate $\chi_{\Omega}w_0$ and $\chi_{\Omega}w_{\pm 1}$ by propagator estimates; and apply the maximum principle to estimate $\chi_{\Omega}P_{\theta, \geq 2} w(y,\theta,\tau)$.

To facilitate later discussions, we define controlling functions $\mathcal{M}_{k},\ k=1,2,3,4.$ Recall that we start our rescaled MCF at time $\tau= \xi_0.$
\begin{align}\label{def:M1}
\begin{split}
\mathcal{M}_{1}(\tau):=\max_{\xi_0\leq s\leq\tau} \frac{1}{\kappa(\epsilon)\Omega^{-4}(s)+s^{-2}}\bigg(&\sum_{m=-1,0,1}\|\langle y\rangle^{-3} \chi_{\Omega(s)}w_{m}(\cdot,s)\|_{\infty}\\
&+\big\| (100+|y|^2)^{-\frac{3}{2}} \|\partial_{\theta}^3 P_{\theta,\geq 2}\chi_{\Omega(s)}w(\cdot,s)\|_{L_{\theta}^2}\big\|_{\infty}\bigg),
\end{split}
\end{align}
\begin{align}\label{def:M2}
\begin{split}
\mathcal{M}_{2}(\tau):=\max_{\xi_0\leq s\leq\tau} \frac{\Omega^{3}(s)}{\kappa(\epsilon)}\bigg(&\sum_{m=-1,0,1}\|\langle y\rangle^{-2} \nabla_y \chi_{\Omega(s)}w_{m}(\cdot,s)\|_{\infty}\\
&+\big\| (100+|y|^2)^{-1}  \|\partial_{\theta}^2 P_{\theta,\geq 2}\nabla_y\chi_{\Omega(s)}w(\cdot,s)\|_{L_{\theta}^2}\big\|_{\infty}\bigg),
\end{split}
\end{align}
\begin{align}
\begin{split}
\mathcal{M}_{3}(\tau):=\max_{\xi_0\leq s\leq\tau} \frac{\Omega^{2}(s)}{\kappa(\epsilon)}\sum_{|l|=2}\bigg(&\sum_{m=-1,0,1}\|\langle y\rangle^{-1} \nabla_y^{l} \chi_{\Omega(s)}w_m(\cdot,s)\|_{\infty}\\
 &+\big\| (100+|y|^2)^{-\frac{1}{2}}  \|\partial_{\theta} P_{\theta,\geq 1}\nabla_y^{l}\chi_{\Omega(s)}w(\cdot,s)\|_{L_{\theta}^2}\big\|_{\infty}\bigg),
 \end{split}
\end{align}
\begin{align}
\begin{split}\label{def:M4}
\mathcal{M}_{4}(\tau):=\max_{\xi_0\leq s\leq\tau} \frac{\Omega^{3}(s)}{\kappa(\epsilon)}\bigg(&\sum_{m=\pm 1}\|\langle y\rangle^{-2} \chi_{\Omega(s)}w_{m}(\cdot,s)\|_{\infty}\\
&+\big\| (100+|y|^2)^{-1}  \|\partial_{\theta}^3 P_{\theta,\geq 2}\chi_{\Omega(s)}w(\cdot,s)\|_{L_{\theta}^2}\big\|_{\infty}\bigg).
\end{split}
\end{align}
Here and in the rest of the paper the norm $\|\cdot\|_{L_{\theta}^2}$ and the inner product $\langle \cdot,\ \cdot\rangle_{\theta}$ signify the $L^2-$norm and $L^2$-inner product in the $\theta-$variable.

Our previous paper \cite{GZ2017} provides the initial conditions for some of $\mathcal{M}_k,\ k=1,2,3,4.$ Specifically, for some constant $C$ and for any $\tau$ sufficiently large, with $R(\tau)$ defined in \eqref{eq:defRTau},
\begin{align}
\begin{split}\label{eq:prev}
\|\langle y\rangle^{-3}\chi_{R}w(\cdot,\tau)\|_{\infty}\leq &C \kappa(\epsilon)R^{-4}(\tau),\\
\|\langle y\rangle^{-2}\nabla_{y}^{m}\partial_{\theta}^{n}\chi_{R}w(\cdot,\tau)\|_{\infty}\leq &C \kappa(\epsilon)R^{-3}(\tau),\ |m|+n=1,\\
\|\langle y\rangle^{-1}\nabla_{y}^{m}\partial_{\theta}^{n}\chi_{R}w(\cdot,\tau)\|_{\infty}\leq &C \kappa(\epsilon)R^{-2}(\tau),\ |m|+n=2.
\end{split}
\end{align}  

Besides these we need the following estimates, 
\begin{proposition}\label{prop:332211}
There exists some constant $C$ such that, if $\tau$ is sufficiently large, then
\begin{align}
\|\langle y\rangle^{-3}\partial_{\theta}^3 P_{\theta,\geq 2}\chi_{R}w(\cdot,\tau)\|_{L_{\theta}^{2}}\leq & C \kappa(\epsilon) R^{-4}(\tau), \label{eq:33} \\
\|\langle y\rangle^{-2}\partial_{\theta}^2 \nabla_{y}P_{\theta,\geq 2}\chi_{R}w(\cdot,\tau)\|_{L_{\theta}^{2}}\leq &C\kappa(\epsilon) R^{-3}(\tau),\label{eq:22}\\
 \sum_{|k|=2}\|\langle y\rangle^{-1}\partial_{\theta} \nabla_{y}^{k}P_{\theta,\geq 2}\chi_{R}w(\cdot,\tau)\|_{L_{\theta}^{2}}\leq & C\kappa(\epsilon) R^{-2}(\tau)\label{eq:11}.
\end{align}
\end{proposition}
This proposition will be proved in Appendix \ref{sec:improved}, based on the results proved in \cite{GZ2017}.

Note here we slightly abuse the notations: in \cite{GZ2017}, we require $e^{-\frac{1}{8}|y|}\chi_{R}w$ to be orthogonal to the functions listed in \eqref{eq:orthow}, while in the present paper we require that $e^{-\frac{1}{8}|y|}\chi_{\Omega}w$ to satisfy the same orthogonal conditions. As argued in the estimates of parameters above, the difference between $\chi_{R}w$ and $\chi_{\Omega}w$ in the overlapping part is of order $\tau^{-20}$, which is negligibly small for the present purpose.

The estimates in \eqref{eq:prev} and Proposition \ref{prop:332211} directly imply that, for some $C>0,$ if $\xi_0$ is sufficiently large, then
\begin{align}
\mathcal{M}_{k}(\xi_0)\leq C.\label{eq:IniWeighted}
\end{align}

For the time $\tau\geq \xi_0$, the functions $\mathcal{M}_{l},\ k=1,2,3,4,$ satisfy the following estimates:
\begin{proposition}\label{prop:weight}
Under the conditions in Theorem \ref{THM:TwoReg}, there exists a constant $C>0$, independent of $\delta,$ such that for any $\tau\in [\xi_0,\tau_1],$ $l=1,2,3,4,$
\begin{align}
\mathcal{M}_l(\tau)\leq C+C\delta \sum_{l=1,2}\sum_{k=1}^{4}\mathcal{M}_{k}^{l},\label{eq:estM1234}
\end{align} where $\tau_1$ is the same time to that in \ref{THM:TwoReg}.
\end{proposition}
This proposition will be proved in subsequent sections.

Now assuming Proposition \ref{prop:weight}, we are ready to prove \eqref{eq:weighL1}-\eqref{eq:weighL3}, which is part of Theorem \ref{THM:TwoReg}.

\begin{proof}
\eqref{eq:IniWeighted} and \eqref{eq:estM1234} imply that, if $\delta$ is sufficiently small, then for all $\tau\in [\xi_0,\ \tau_1]$
\begin{align}
\mathcal{M}_{k}(\tau)\lesssim 1,\ k=1,2,3,4.
\end{align} 

This together with Lemma \ref{LM:poinEmbed} below implies that, for example, 
\begin{align}
\|\langle y\rangle^{-3} \chi_{\Omega}w(\cdot,\tau)\|_{\infty}\lesssim \Big(\tau^{-2}+\kappa(\epsilon)\Omega^{-4}(\tau)\Big) \mathcal{M}_1(\tau)\lesssim \tau^{-2}+\kappa(\epsilon)\Omega^{-4}(\tau).
\end{align} This is part of \eqref{eq:weighL1}.

The other estimates in \eqref{eq:weighL1}-\eqref{eq:weighL3} will be derived similarly.

The proof is complete.
\end{proof}

In the proof above the following lemma was used.
\begin{lemma}\label{LM:poinEmbed}
There exists a constant $C$ such that for any smooth function $f: \mathbb{T}\rightarrow \mathbb{C}$, we have that, for any $l\in \mathbb{N},$
\begin{align}
\|P_{\theta,\geq 2} f\|_{L^{\infty}_{\theta}}\leq C \|\partial_{\theta}^{l}P_{\theta,\geq 2} f\|_{L^{2}_{\theta}},\ \text{and}\ \|P_{\theta,\geq 2}f\|_{L_{\theta}^2}\leq \|\partial_{\theta}^{l}P_{\theta,\geq 2} f\|_{L^{2}_{\theta}}.\label{eq:embedding}
\end{align}

\end{lemma}
\begin{proof}
It is easy to prove the second estimate, by Fourier expanding $f$ and analyzing each frequency.
We only prove the first one with $l=1.$  The others follow after applying the second one.

Apply $P_{\theta,\geq 2}$ to $f$, whose Fourier expansion takes the form
$
f(\theta)=\displaystyle\sum_{n=-\infty}^{\infty} e^{in\theta} f_{n},
$  to have
\begin{align}
P_{\theta,\geq 2} f(\theta)=\sum_{|n|\geq 2} e^{in\theta} f_{n}.
\end{align}

Compute directly to obtain, for any $\zeta>0,$
\begin{align}
|P_{\theta,\geq 2} f(\theta)|\leq \sum_{|n|\geq 2} |f_{n}| =\sum_{|n|\geq 2} \frac{1}{|n|} |n f_{n}|\leq \zeta\sum_{|n|\geq 2}\frac{1}{|n|^2}+\frac{1}{\zeta} \sum_{|n|\geq 2}  |n f_{n}|^2
\end{align} Now we choose
$
\zeta:=\|\partial_{\theta}P_{\theta,\geq 2} f\|_{L^{2}}=\sqrt{\sum_{|n|\geq 2}  |n f_{n}|^2},
$
and use that $\sum_{|n|\geq 2}\frac{1}{|n|^2}<\infty$ to find the desired result.

\end{proof}

\section{Local Smooth Extension: Proof of Lemma \ref{LM:regularity}}\label{sec:regul}

The main tools here are the local smooth extension, see \cite{Eckerbook}, and comparing the rescaled MCF to MCF, which were used in \cite{ColdingMiniUniqueness}. See also \cite{MR485012, white05, MR1196160, CIM13}.

To facilitate later discussions we define a new MCF by rescaling the old one, with a fixed constant, and shifting time. Let $\tau_1$ be the chosen time in Lemma \ref{LM:regularity}. We define
\begin{align}
t_1:=t(\tau_1)
\end{align}
recall that $\tau$ and $t$ are related (one to one) by $\tau:=-\ln (T-t)$. We rescale the MCF such that the part of the manifold being a graph parametrized by $u$ in \eqref{eq:origPara} is parametrized as
\begin{align}
\sqrt{T-t_1}\left[
\begin{array}{ccc}
z\\
q(z, \theta, s)\cos\theta\\
q(z,\theta,s) \sin \theta
\end{array}
\right],
\end{align}
and it is related to the original evolution by
\begin{align}
q(z,\theta,s)=\frac{1}{\lambda} u(\lambda z, \theta, \lambda^2 s+t_1).\label{eq:restart0}
\end{align} Here $s$ and $\lambda$ are defined as \begin{align}
s:=\frac{t-t_1}{\lambda^2}, \ \lambda:=\sqrt{T-t_1}.
\end{align} This is a well defined MCF since MCF is scaling invariant, which means that if $u(z,\theta,t)$ is a solution to \eqref{eq:MCF}, then so is $\lambda^{-1} u(\lambda z,\theta, \lambda^2 t)$ for any $\lambda>0$. 

The evolution of $q$ is related to that of $v$ by the identity, 
\begin{align}
q(z,\theta,s)=&\frac{1}{\lambda} u(\lambda z, \theta, \lambda^2 s+t_1)\nonumber\\
=&\sqrt{1-s}\ v\big(\frac{z}{\sqrt{1-s}},\theta, -\ln(1-s)+\tau_1\big).\label{eq:newRes}
\end{align} 
resulted by the identity relating $v$ to $u$ in \eqref{eq:rescale} and elementary calculations.

\eqref{eq:newRes} implies that the new MCF will blow up at time $s=1$ and at $z=0$.

Now we start proving Lemma \ref{LM:regularity}.

Similar results were proved in our previous paper \cite{GZ2017}, where the considered region is $|y|\leq R(\tau)=\mathcal{O}(\sqrt{ln \tau})$, thus the considered part of hypersurface is very close to a cylinder with radius $\sqrt{2}.$ In the present situation, the considered region is $|y|\leq \Omega(\tau)=\mathcal{O}(\tau^{\frac{1}{2}+\frac{1}{20}})$, the radius are very different on the different scales if at least one of $b_k$'s in \eqref{eq:Best} is $1$. Hence the difficulty here is to show that the estimates hold uniformly across all parts.

For that purpose, we consider the different regions separately. Let $N$ be the smallest integer such that 
\begin{align}
2^{N}\sqrt{2}\geq \sup_{|y|\leq\Omega(\tau)}\sqrt{2+
\tau_1^{-1}y^{T}\tilde{B}y},\label{eq:Nn}
\end{align} Then for $n=  1,\ 2,\cdots, \ N$, we define the following overlapping regions $\Lambda_{n}\subset \mathbb{R}^3$,
\begin{align}\label{eq:Lambdan}
\Lambda_{n}:=\Big\{y\ \big|\ |y|\leq \Omega(\tau_1),\ \text{and }\ \sqrt{2+
\tau_1^{-1}y^{T}\tilde{B}y}\in \Big[2^{n-1}\sqrt{2}, 2^{n+1}\sqrt{2}\Big]\Big\}.
\end{align}

It is easy to prove the local extension for the region $\Lambda_1$ as in \cite{GZ2017}. 

To study the evolution of $q$ for positive time $s\geq 0$, by the technique of smooth local extension, we need information of $q$ when $s\in [-1,\ 0].$ These will be provided, through the identity \eqref{eq:newRes}, by the estimates on $v$ in the region $|y|\leq \Omega(\tau)$. 

$v(y,\theta,\tau)$, restricted to the region $|y|\leq \Omega(\tau)$, can provides information for $q(z,\theta,s)$ in the region
\begin{align}
 -1\leq s\leq 0,\ \text{and},\ |z|\leq \Omega(\tau_1) ,\label{eq:growR}
\end{align} since $\frac{\Omega(\tau_1)}{\sqrt{1+|s|}}\leq \Omega(-\ln(1+|s|)+\tau_1)$ when $s\in [-1,\ 0]$ by the definition of $\Omega$ in \eqref{eq:defOmega}, and the different growth rates of $\Omega(-\ln(1+|s|)+\tau_1)$ and $\sqrt{1+|s|}$ for $s\leq 0$.

And we observe that, by setting $s=0$ in \eqref{eq:newRes},
\begin{align}
q(z,\theta,0)=v(z,\theta, \tau_1).\label{eq:setS0}
\end{align}

By the estimates for $v$ and its derivatives in \eqref{eq:PointSmall}, we have that when
\begin{align}
z\in  \Lambda_1\ \text{and}\ s\in [-1,\ 0],
\end{align} the function $q$ satisfies the estimate
\begin{align}
1\leq q\leq 10, \ |\nabla_{z}^{k}\partial_{\theta}^{l}q|\leq \Omega^{-\frac{1}{2}}(\tau_1),\  |k|+l=1,2.
\end{align}
This, together with the techniques of regularity estimate and of interpolation between the estimates of derivatives, implies that when $$s=0,\ z\in \{z| z\in \Lambda_1,\  dist(z,\partial\Lambda_1)\geq 1\},$$ we have, there exists some $\tilde\Omega(\tau_1)$ satisfying $\tilde{\Omega}(\tau_1)\rightarrow \infty$ as $\Omega(\tau_1)\rightarrow \infty,$ such that
\begin{align}
\Big|q(\cdot,0)-\sqrt{2+\tau_1^{-1}z^{T}\tilde{B}z}\Big|, \ \Big|\nabla_{z}^{k}\partial_{\theta}^{l} q(\cdot,0)\Big|\leq \frac{1}{\tilde{\Omega}(\tau_1)},\ \text{if}\ |k|+l=1, 2, 3,4,5.\label{eq:1To5}
\end{align} Here $\partial \Lambda_1$ signifies the boundary of the set $\Lambda_1.$

Let $\xi_0$ be the constant in the definition of $\Omega(\tau)$. For the positive time, $s\geq 0,$
we choose $\xi_0$, the initial time, to be large enough to make $ \frac{1}{\tilde{\Omega}(\tau_1)}\leq \frac{1}{40}\delta $, with $\delta$ being the chosen small constant in Lemma \ref{LM:regularity}. 
By the local smooth extension, there exists a small constant $\kappa=\kappa(\delta)$ such that 
\begin{align}
\text{for}\ s\leq 10\kappa\ \text{and}\ z\in \Big\{z| z\in \Lambda_1,\  dist(z,\partial\Lambda_1)\geq 10\Big\}, 
\end{align}
the following estimates hold, for $|k|+l=1, 2, 3,4,5,$
\begin{align}
\Big|\frac{1}{\sqrt{1-s}}q(\cdot,s)-\sqrt{2+\big(\tau_1-\ln(1-s)\big)^{-1}(1-s)^{-1} z^{T}\tilde{B}z}\Big|,\  \ \Big|\nabla_{z}^{k}\partial_{\theta}^{l} q(\cdot,s)\Big|\leq \frac{1}{20}\delta.\label{eq:1fifDe}
\end{align}

Now we convert the estimates to that on $v$ by the identity \eqref{eq:newRes}, for $\tau=\tau_1-\ln(1-s)$ and $y=\frac{z}{\sqrt{1-s}}$, with $0\leq s\leq 10\kappa$, and, 
\begin{align}
(1-s)y\in\tilde\Lambda_1:= \{z| z\in \Lambda_1,\  dist(z,\partial\Lambda_1)\geq 10\}
\end{align} we have
\begin{align}
\Big|v(y,\theta,\tau)-\sqrt{2+\tau^{-1}y^{T}\tilde{B}y}\Big|, \  v^{|k|-1}\Big|\nabla_{y}^{k}\partial_{\theta}^{l} v(y,\theta,\tau)\Big|\lesssim \delta.\label{eq:PosiSma}
\end{align}

Now we consider the region $\Lambda_n,\ n\geq 2$.

To make the arguments comparable to that for $ \Lambda_1$, we rescale $q$ in \eqref{eq:newRes} by defining $\lambda_n:=2^{n}$ and,
\begin{align}
\begin{split}\label{eq:scalN}
q_n(z,\theta,s_1):=&\lambda_n^{-1} q(\lambda_n z,\theta, \lambda_n^2 s_1),\\
=&\frac{\sqrt{1-s_1}}{\lambda_n} v\big(\frac{\lambda_n}{\sqrt{1-s_1}}z,\theta, -\ln(1-\lambda_n^2 s_1)+\tau_1\big).
\end{split}
\end{align}
Note that this is a new MCF.

We observe that:
\begin{itemize}
\item[(1)] When $n\geq 2$, we only need to extend the solution into the positive time $s_1\leq 10\lambda_n^{-2}\kappa$ since the purpose is to extend $q$ into positive time $10\kappa$, see the identity in \eqref{eq:scalN} .
\item[(2)]
when $s_1\in [-1,\ 0]$ and $z\in \lambda_n^{-1}\Lambda_n\subset \lambda^{-1}_{n}\Big\{y\ \big|\ |y|\leq \Omega(-\ln(1-\lambda_n^2 s_1)+\tau_1)\Big\}$, the informations for $q_n(z,\theta,s_1)$ can still be provided by $v(y,\theta,\tau),$ $\tau\in [\xi_0,\ \tau_1]$, $|y|\leq \Omega(\tau)$, through the second identity in \eqref{eq:scalN} by the following reasons. 

If the set $\Lambda_n,\ n\geq 2,$ is not empty, i.e. $\sqrt{2+
\tau_1^{-1}y^{T}\tilde{B}y}\geq 2\sqrt{2}$ for some $|y|\leq \Omega(\tau_1)$ then the definition of $\Omega(\tau)$ forces that $\tau_1-\xi_0\geq \tau_1^{\frac{1}{2}}.$ By the second identity in \eqref{eq:scalN} and that $\lambda_n$ is at most (a modest) $4\tau_1^{\frac{1}{20}}$ by \eqref{eq:Lambdan}, $s_1$ can take any value in $[-1,\ 0]$ since for the corresponding time variable $\tau$ of $v$, we have $\tau=-\ln(1-\lambda_n^2 s_1)+\tau_1\in (\xi_0,\tau_1]$.

For the $z$ variable, we observe that even though the sets $\{y\big| \ |y|\leq \Omega(\tau)\}$ are increasing in $\tau$, the rescaling relation $y=\frac{z}{\sqrt{T-t}}=e^{\frac{1}{2}\tau}z$ (when we rescale $u$ to $v$ in \eqref{eq:rescale}) renders that the corresponding set for $z$ is actually (favorably) shrinking. Recall that $q_n$ is a rescaled version of $q$, and hence is a rescaled version of $u$.
Hence $v$ can provide information for $q_n$ for $s_1\in [-1,\ 0]$ and $z\in \lambda_n^{-1}\Lambda_n\subset \lambda_{n}^{-1}\{y\ |\ |y|\leq \Omega(\tau_1)\}.$
\item[(3)] The estimates in \eqref{eq:PointSmall} imply that, at $s_1=0$, for any $z\in \lambda_n^{-1}\Lambda_n,$ since $1\leq \lambda_n< 4\tau_1^{\frac{1}{20}},$
\begin{align}
\Big| q_n(z,\theta,0)-\lambda_n^{-1}\sqrt{2+\tau_1^{-1}\lambda_n^2 z^T\tilde{B}z} \Big|\leq \Omega^{-\frac{2}{5}}(\tau_1),\ \text{and} \  |\nabla_{z}^{k}\partial_{\theta}^{l}q_n(z,\theta,0)|\leq \Omega^{-\frac{3}{4}}(\tau_1).
\end{align}
And for $s_1\in [-1,0]$, and for any $z \in \lambda_n^{-1}  \Lambda_n$,
\begin{align}
1\leq q_n(z,\theta,s_1)\leq 10, \ |\nabla_{z}^{k}\partial_{\theta}^{l}q_n(z,\theta,s_1)|\leq \Omega^{-\frac{1}{2}}(\tau_1),\  |k|+l=1,2.
\end{align}
\end{itemize}

Now we run the same arguments as those in \eqref{eq:1To5}-\eqref{eq:1fifDe}, and find that, 
\begin{align}
\text{for}\ s_1\leq 10 \lambda_n^{-2}\kappa\ll 1\ \text{and}\ z\in \Big\{z| z\in \lambda_n^{-1}\Lambda_n,\  dist(z,\partial\lambda_n^{-1}\Lambda_n)\geq 10\Big\}, 
\end{align} the following estimates hold
\begin{align}\label{eq:smallV}
|q_n(z,\theta,s_1)-q_n(z,\theta,0)|\leq \frac{1}{20}\delta
\end{align}
and hence
\begin{align}
\Big|\frac{1}{\sqrt{1-s_1}}q_n(z,\theta,s_1)-\lambda_n^{-1}\sqrt{2+\big(\tau_1-\ln(1-s_1)\big)^{-1}(1-s_1)^{-1} \lambda_n^2 z^{T}\tilde{B}z}\Big|\leq & \frac{1}{5}\delta,\label{eq:smallDel}
\end{align}
and similarly
\begin{align}
|\partial_{\theta}^{l}\nabla_{z}^{k}q_n(z,\theta,s_1)|\leq &\frac{1}{5}\delta, \ |k|+l=1,2,3,4,5.
\end{align} 

The estimates on $q_n$ imply those for $q$, and hence $v$ through the identities in \eqref{eq:scalN}.

Here we have a minor difficulty. After rescaling back, some estimates will become adversely large. For example, $|\partial_{\theta}^{k} q|, \ k=1,\cdots,5,$ will be of the order $\lambda_n \delta.$ To offset this, we apply a factor $q^{-1}$ or $\frac{1}{\sqrt{2+\tau^{-1}y^{T}\tilde{B}y}}$ to make
\begin{align}
v^{-1}|\partial_{\theta}^k v(y,\theta,\tau)|=q^{-1}|\partial_{\theta}^{k} q(z,\theta,t)|=q_n^{-1}|\partial_{\theta}^{k} q_n(x,\theta,\lambda_{n}^{-2}t)|_{z=\lambda_n^{-1} x}\lesssim \delta
\end{align}
and
\begin{align}
 \frac{1}{\sqrt{2+\tau^{-1}y^{T}\tilde{B}y}}|\partial_{\theta}^{k} v|\lesssim \delta
\end{align}
since $\lambda_{n}=2^{n}$ and by \eqref{eq:Nn}, \eqref{eq:smallV} and \eqref{eq:smallDel}, $\frac{\sqrt{2+\tau^{-1}y^{T}\tilde{B}y}}{2^{n}}, \ \frac{v}{\sqrt{2+\tau^{-1}y^{T}\tilde{B}y}}\in [\frac{1}{4},\ 4]$ in the region $\Lambda_n$.

And to some of the functions, for example $\nabla_{z}q(z,\theta,t),$ we do not need such help, since itself is ``scaling invariant" in the sense that 
 $$\nabla_{z}q(z,\theta,t)=\nabla_{x} q_n( x,\theta,\lambda_{n}^{-2}t)\Big|_{z=\lambda_n^{-1} x},$$ and similarly for $\nabla_{z}^{k}\partial_{\theta}^{m}q,\ |k|\geq 2$, we have $$q^{|k|-1}|\nabla_{z}^{k}q(z,\theta,t)|= \Big( q_n( x,\theta,\lambda_{n}^{-2}t)\Big)^{|k|-1}\Big|\nabla_{x}^{k} q_n( x,\theta,\lambda_{n}^{-2}t)\Big|_{z=\lambda_n^{-1} x}.$$

Consequently for $\tau=\tau_1-\ln(1-s)$ with $0\leq s\leq 10\kappa,$ and
\begin{align}\label{eq:defTLam}
(1-s)y\in\tilde{\Lambda}_n:= \{z| z\in \Lambda_n,\  dist(z,\partial\Lambda_n)\geq 10 \lambda_n \}
\end{align}
$v$ satisfies the desired \eqref{eq:SmoExte}, namely for $m=1,\cdots, 5,$ and $|k|+l=1,\cdots,5$, and $|k|\geq 1$,
\begin{align}
|\frac{v(\cdot,\tau)}{\sqrt{2+\tau^{-1}y^{T}\tilde{B}y}}-1|,\ v^{-1}\Big|\partial_{\theta}^m v\Big|,\ v^{|k|-1}\Big|\nabla_{y}^{k}\partial_{\theta}^{l} v(\cdot,\tau)\Big|\lesssim \delta.
\end{align}

Now we need to prove that the union of the sets $\tilde\Lambda_n$, defined in \eqref{eq:defTLam}, includes the desired set \eqref{eq:SmoExte}. This is easy, since in making sure there will not be a ``hole" inside, we make the overlapping of different sets $\Lambda_n$ large enough. Recall that $\lambda_n$ is at most $4\tau_1^{\frac{1}{20}}$, see \eqref{eq:Lambdan}, while the overlapping of different $\Lambda_n$ is of order $\tau_1.$

The proof is complete.



\section{Estimates for $\mathcal{M}_1$ and $\mathcal{M}_4$, Proof of part of \eqref{eq:estM1234}}\label{sec:estM1}
We reformulate the results to Proposition \ref{prop:M14}, specifically estimates in \eqref{eq:estM14} below.

To simplify the notations, in the rest of paper we use $P(M)$ to denote
\begin{align}
P(M):=1+\sum_{k=1}^{4}\mathcal{M}_{k}+\sum_{k=1}^{4}\mathcal{M}_{k}^{2}.\label{eq:defPM}
\end{align}
Choose $\xi_0$, which is the initial time for our rescaled MCF, to be sufficiently large such that
\begin{align}
 \Big(\kappa(\epsilon)+1\Big)\Omega^{-\frac{1}{10}}(\tau)\leq \delta,\ \text{for}\ \tau\geq \xi_0.\label{eq:assum}
\end{align} 
Recall that $\Omega$ and $\kappa(\epsilon)$ are defined in \eqref{eq:defOmega} and \eqref{eq:defKappa} respectively.

In the rest of this section we prove the following estimates for the different components in $\mathcal{M}_1$ and $\mathcal{M}_4$, 
\begin{proposition}\label{prop:M14}
\begin{align}
\|\langle y\rangle^{-3} w_0(\cdot,\tau)\|_{\infty}\lesssim \Big(\kappa(\epsilon) \Omega^{-4}(\tau)  +\tau^{-2}\Big)\big(1+\delta P(M)\big),\label{eq:30w0}
\end{align}
\begin{align}
\begin{split}\label{eq:3020wpm}
\|\langle y\rangle^{-3} w_{\pm}(\cdot,\tau)\|_{\infty}\lesssim & \Big(\kappa(\epsilon)\Omega^{-4}  +\tau^{-2}\Big)\big(1+\delta P(M)\big),\\
\|\langle y\rangle^{-2} w_{\pm}(\cdot,\tau)\|_{\infty}\lesssim & \delta \kappa(\epsilon)\Omega^{-3}  P(M),
\end{split}
\end{align}
and 
\begin{align}
\begin{split}\label{eq:3020Wtheta}
\Big\| (100+|y|^2)^{-\frac{3}{2}} \| P_{\theta,\geq 2}\partial_{\theta}^3 \chi_{\Omega}w(\cdot,\tau)\|_{L_{\theta}^2}\Big\|_{\infty}\lesssim &\delta \Big(\kappa(\epsilon)\Omega^{-4} +\tau^{-2}\Big)P(M),\\
\Big\| (100+|y|^2)^{-1} \| P_{\theta,\geq 2}\partial_{\theta}^3 \chi_{\Omega}w(\cdot,\tau)\|_{L_{\theta}^2}\Big\|_{\infty}\lesssim &\delta \kappa(\epsilon)\Omega^{-3} P(M).
\end{split}
\end{align}
Consequently, by the definitions of $\mathcal{M}_1$ and $\mathcal{M}_4$ in \eqref{def:M1} and \eqref{def:M4},
\begin{align}\label{eq:estM14}
\mathcal{M}_1\lesssim 1+\delta P(\mathcal{M}),\ \text{and}\ 
\mathcal{M}_4\lesssim \delta \kappa(\epsilon)P(\mathcal{M}).
\end{align}
\end{proposition}

The proposition will be proved in subsequent subsections. \eqref{eq:estM14} is directly implied by the estimate before it.

The following results will be used frequently in controlling various terms in nonlinearity, see e.g. \eqref{eq:nonl} below. For convenience we state them here.
\begin{lemma}\label{LM:appliM1234}
For $l_1=0,1,2,$ and $l_2=0,1,2,3,$
\begin{align}
\|\langle y\rangle^{-3} \partial_{\theta}^{l_1}\chi_{\Omega}w(\cdot,\tau)\|_{\infty}, \ \big\|\langle y\rangle^{-3} \|\partial_{\theta}^{l_2}\chi_{\Omega}w(\cdot,\tau)\|_{L_{\theta}^{2}}\big\|_{\infty}\lesssim \big(\tau^{-2}+\kappa(\epsilon) \Omega^{-4} \big) \mathcal{M}_1(\tau),\label{eq:embeddM1}
\end{align}
for $n_1=1,2,$ and $n_2=1,2,3,$
\begin{align}
\|\langle y\rangle^{-2} \partial_{\theta}^{n_1}\chi_{\Omega}w(\cdot,\tau)\|_{\infty},\ \big\|\langle y\rangle^{-2} \|\partial_{\theta}^{n_2}\chi_{\Omega}w(\cdot,\tau)\|_{L_{\theta}^2}\big\|_{\infty}\lesssim \kappa(\epsilon) \Omega^{-3}  \mathcal{M}_4(\tau),\label{eq:embeddM4}
\end{align}
for $m_1=0,1$, and $m_2=0,1,2,$
\begin{align}
\|\langle y\rangle^{-2} \partial_{\theta}^{m_1}\nabla_{y}\chi_{\Omega}w(\cdot,\tau)\|_{\infty},\ \big\|\langle y\rangle^{-2} \|\partial_{\theta}^{m_2}\nabla_{y}\chi_{\Omega}w(\cdot,\tau)\|_{L_{\theta}^2}\big\|_{\infty}\leq \kappa(\epsilon) \Omega^{-3}  \mathcal{M}_2(\tau),\label{eq:embeddM2}
\end{align}
and for $|k|=2,$ and $d=0,1,$
\begin{align}
\|\langle y\rangle^{-1}\nabla_{y}^k\chi_{\Omega}w(\cdot,\tau)\|_{\infty},\ \big\|\langle y\rangle^{-1}\|\nabla_{y}^k\partial_{\theta}^{d}\chi_{\Omega}w(\cdot,\tau)\|_{L_{\theta}^2}\big\|_{\infty}\lesssim \kappa(\epsilon) \Omega^{-2}  \mathcal{M}_3(\tau).\label{eq:embeddM3}
\end{align}
\end{lemma}
\begin{proof}
To prove the first estimate in \eqref{eq:embeddM1} 
we decompose $w$ as in \eqref{eq:fourParts}, and apply Lemma \ref{LM:poinEmbed}, to find, when $l=0,1,2,$
\begin{align}
\begin{split}
\|\langle y\rangle^{-3} \partial_{\theta}^{l}\chi_{\Omega}w\|_{\infty}
\lesssim &  \sum_{m=-1,0,1}\|\langle y\rangle^{-3} \chi_{\Omega}w_{m}\|_{\infty}+\Big\|(100+|y|^2)^{-\frac{3}{2}} \|P_{\theta,\geq 2}\partial_{\theta}^3\chi_{\Omega}w\|_{L_{\theta}^2}\Big\|_{\infty},
\end{split}
\end{align}
which, together with the definitions of $\mathcal{M}_{1}$, implies the desired estimate.

Similarly, for the second in \eqref{eq:embeddM1} 
\begin{align}
\|\chi_{\Omega}\partial_{\theta}^{l_2}w\|_{L_{\theta}^2}\lesssim \sum_{m=-1,0,1}|w_m|+\|P_{\theta,\geq 2}\partial_{\theta}^3\chi_{\Omega}w\|_{L_{\theta}^2}.
\end{align}
then we consider the weighted $L^{\infty}-$norm to obtain the desired result.

The proof of the others are similar, except that we need to use, for example, in the proof of \eqref{eq:embeddM4}, that $\partial_{\theta}w_0=0$ since $w_0$ is $\theta-$independent. Hence we skip the details here.

\end{proof}

In the rest of this section we prove Proposition \ref{prop:M14}. Since similar techniques will be used frequently in the rest of the paper, we present the details here.

\subsection{Proof of \eqref{eq:30w0}}

From \eqref{eq:tildew3} we derive an equation for $\chi_{\Omega}w_0$ as 
\begin{align}\label{eq:eqnW0}
\begin{split}
\partial_{\tau}\chi_{\Omega}w_0=-H \chi_{\Omega}w_0+   \chi_{\Omega}\Sigma 
&+\frac{1}{2\pi}\chi_{\Omega}\langle N_{1}(v)+N_2(\eta),\ 1\rangle_{\theta}+\mu(w_0),
\end{split}
\end{align} where the function $\Sigma$ is defined as
\begin{align*}
\Sigma:=&\frac{1}{2\pi}\langle F(B,a)+G(\vec\beta_1, \vec\beta_2, \vec\beta_3, \alpha_1, \alpha_2),\ 1\rangle_{\theta}
=F(B,a)+ [\frac{2a}{2+y^{T}By}\vec\beta_1-\frac{d}{d\tau}\vec\beta_1] \cdot y,
\end{align*} by using that $F(B,a)$ is independent of $\theta;$ and the $\theta-$dependent part of $G$ is orthogonal to 1, hence does not contribute;
and the linear operator $H$ is defined as 
\begin{align}
\begin{split}\label{eq:defML}
H:=&-\Delta_y+\frac{1}{2}y\cdot \nabla_{y}-\frac{1}{2}-a-\tau^{-\frac{1}{2}}+V_1 ,\\
V_1:=&\Big[\frac{a y^{T}By}{2+y^{T}B y}+\tau^{-\frac{1}{2}}\Big]\chi_{2\Omega},
\end{split}
\end{align} resulted from the observations $$
\langle L \chi_{\Omega}w, \ 1\rangle_{\theta}=H \chi_{\Omega}w_{0},\ \text{and}\ \chi_{2\Omega}\chi_{\Omega}=\chi_{\Omega}.
$$
Here we use $\langle V_{a,B}^{-2}\partial_{\theta}^2 w,\ 1\rangle_{\theta}=0$ since $V_{a,B}$ is independent of $\theta$, the function $\chi_{D}$ is defined as
\begin{align}
\chi_{D}(|y|):=\chi(\frac{|y|}{D}),\ \text{for any} \ D>0.
\end{align}

In the definition of $V_1$ in \eqref{eq:defML}, we impose the cutoff function $\chi_{2\Omega}$ to make sure that $V_1$ is well defined for all $y\in \mathbb{R}^3$ since, by \eqref{eq:Best}, the denominator $2+y^{T}B y\geq 1$ when $|y|\leq 2(1+\epsilon)\Omega.$

Now we study $\mu(w_0)$ by dividing it into two parts:
\begin{align}
\mu(w_0)=\frac{1}{2}(y\cdot\nabla_{y}\chi_{\Omega})w_0+\Gamma(w_0)
\end{align} with $\Gamma(w_0)$ defined as
\begin{align*}
\Gamma(w_0):=(\partial_{\tau}\chi_{\Omega})w_0-(\Delta_{y}\chi_{\Omega})w_0-2\nabla_{y}\chi_{\Omega}\cdot  \nabla_{y}w_0.
\end{align*}

The term $\Gamma(w_0)$ is small, by
the definition $\chi_{\Omega}(|y|)=\chi(\frac{|y|}{\Omega})$ and the estimates in \eqref{eq:SmCon}.

It is more involved to control the second term, resulted by two observations, as pointed out in \cite{GZ2017}: (1) the $L^{\infty}$-norm of the function $\frac{1}{2}y\cdot\nabla_{y}\chi_{\Omega}$ will not converge to $0$ as $\tau\rightarrow\infty$, instead, the definition $\chi_{\Omega}(|y|)=\chi(\frac{|y|}{\Omega})$ implies
\begin{align}
\sup_{y}\Big|\frac{1}{2}y\cdot\nabla_{y}\chi_{\Omega}(y)\Big|=\frac{1}{2}\sup_{x\geq 0}\Big|\chi^{'}(x)\Big|,
\end{align}  (2) the mapping $\chi_{\Omega}w\rightarrow \frac{1}{2}(y\cdot \nabla_{y}\chi_{\Omega})w=\frac{1}{2}\frac{y\cdot \nabla_{y}\chi_{\Omega}}{\chi_{\Omega}}\chi_{\Omega}w$ is unbounded since $|\frac{y\cdot \nabla_{y}\chi_{\Omega}}{\chi_{\Omega}}|\rightarrow \infty$ as $|y|\rightarrow (1+\epsilon)\Omega$.

To overcome this difficulty, the key observation is that $\frac{1}{2}y\cdot\nabla_{y}\chi_{\Omega}$ has a favorable non-positive sign, by the requirement $\chi(z)=\chi(|z|)$ being decreasing in $|z|$ (see \eqref{eq:defChi3}).
Our strategy is to absorb ``most" of it into the linear operator. For that purpose we define a new non-negative smooth cutoff function $\tilde\chi_{\Omega}(y)$ such that
\begin{align}
\tilde\chi_{\Omega}(y)=\left[
\begin{array}{ll}
1 ,\ \text{if}\ |y|\leq \Omega(1+\epsilon- \Omega^{-\frac{1}{4}}),\\
0,\ \text{if}\ |y|\geq \Omega(1+\epsilon-2 \Omega^{-\frac{1}{4}})
\end{array}
\right.
\end{align} and require it satisfies the estimate
\begin{align}
|\nabla_{y}^{k}\tilde{\chi}_{\Omega}(y)|\lesssim \Omega^{-\frac{3}{4}|k|},\ |k|=1, \ 2.
\end{align} Such a function is easy to construct, hence we skip the details. 

Then we decompose $\frac{1}{2} (y\cdot\nabla_y  \chi_{\Omega}) w_0 $ into two parts
\begin{align}
\frac{1}{2} (y\cdot\nabla_y \chi_{\Omega}) w_0=\frac{1}{2} \frac{(y\cdot\nabla_y \chi_{\Omega}) \tilde\chi_{\Omega}}{\chi_{\Omega}}\chi_{\Omega} w_0+\frac{1}{2} (y\cdot\nabla_y \chi_{\Omega}) (1-\tilde\chi_{\Omega})w_0.\label{eq:unboundtwo}
\end{align} 

The following three observations will be used in later development: 
\begin{itemize}
\item[(A)]
The first part in \eqref{eq:unboundtwo} is a bounded (but not uniformly bounded) multiplication operator since, for some $c(\epsilon)>0,$
\begin{align}
\Big|\frac{y\cdot\nabla_{y} \chi_{\Omega}\  \tilde\chi_{\Omega}}{\chi_{\Omega}}\Big|\leq c(\epsilon) \Omega^{\frac{1}{4}}.\label{eq:NewPoten}
\end{align}
\item[(B)]
If $\Omega$ is sufficiently large, then by that $|\frac{d}{d|z|}\chi(|z|)|\rightarrow 0$ rapidly as $|z|\rightarrow 1+\epsilon$, see \eqref{eq:properties}, we have, since $\nabla_{y}^{l}\chi_{\Omega}(y)=\Omega^{-|l|}\ \nabla_{z}^{l}\chi(z)\Big|_{z=\frac{y}{\Omega}},$
\begin{align}
\Big|\nabla_{y}^{k}\big(y\cdot\nabla_{y} \chi_{\Omega}) (1-\tilde\chi_{\Omega})\big)\Big|\leq \Omega^{-5},\ |k|=0,1,2 .\label{eq:TwoCut}
\end{align}  
\item[(C)]
If $\Omega$ is sufficiently large, then by the properties of $\chi$ in \eqref{eq:properties}, we have that,
\begin{align}
\Big|\nabla_{y}^{k}\big[\frac{y\cdot\nabla_{y} \chi_{\Omega}\  \tilde\chi_{\Omega}}{\chi_{\Omega}}\big]\Big|\leq  \Omega^{-\frac{1}{4}},\ |k|=1,2.\label{eq:deriveSmooth}
\end{align} 
\end{itemize}

Returning to the equation for $\chi_{\Omega}w_0$ in \eqref{eq:eqnW0}, we absorb the first part in \eqref{eq:unboundtwo} into the linear operator and leave the second to the remainder, thus
\begin{align}
\begin{split}
\partial_{\tau}(\chi_{\Omega}w_0)=-H_2(\chi_{\Omega} w_0)+\chi_{\Omega} \Sigma
+\frac{1}{2\pi}\chi_{\Omega}\langle N_{1}(v)+N_2(\eta),\ 1\rangle_{\theta}+
\Lambda(w_0)
,\label{eq:weightLinfw}
\end{split}
\end{align} with the linear operator $H_2$ defined as
\begin{align*}
H_2:=H-\frac{1}{2}\frac{\tilde\chi_{\Omega}\ y\cdot \nabla_{y} \chi_{\Omega}  }{\chi_{\Omega}}=H+\frac{1}{2}|\frac{\tilde\chi_{\Omega}\ y\cdot \nabla_{y} \chi_{\Omega}  }{\chi_{\Omega}}|,
\end{align*} and $\Lambda$ is a linear operator defined as
\begin{align}
\begin{split}\label{eq:defPsiw}
\Lambda(w_0):=
&(\frac{1}{2}y\cdot \nabla_{y} \chi_{\Omega})\ (1-\tilde\chi_{\Omega}) w_0+(\partial_{\tau}\chi_{\Omega})w_0-(\Delta_{y}\chi_{\Omega})w_0-2\nabla_{y}\chi_{\Omega}\cdot  \nabla_{y}w_0.
\end{split}
\end{align}

Observing the operator $e^{-\frac{1}{8}|y|^2}H_2 e^{\frac{1}{8}|y|^2}$, mapping $L^2$ space into itself, is self-adjoint, we transform the equation accordingly
\begin{align}
\begin{split}\label{eq:selfadw}
\partial_{\tau} (e^{-\frac{1}{8}|y|^2}\chi_{\Omega}w_0)=-\mathcal{L}(e^{-\frac{1}{8}|y|^2}\chi_{\Omega}w_0)
+&e^{-\frac{1}{8}|y|^2}\Big[\chi_{\Omega}\bigg(\Sigma+\frac{1}{2\pi}\langle N_{1}(v)+N_2(\eta),\ 1\rangle_{\theta}\bigg)+
\Lambda(w_0)
\Big],
\end{split}
\end{align}
with the linear operator $\mathcal{L}$ defined as
\begin{align*}
\mathcal{L}:=e^{-\frac{1}{8}|y|^2} H_2 e^{\frac{1}{8}|y|^2}=-\Delta_{y}+\frac{1}{16}|y|^2-\frac{3}{4}-a-\frac{1}{2}-\tau^{-\frac{1}{2}}+V_1+\frac{1}{2}\Big|\frac{\tilde\chi_{\Omega}\ y\cdot \nabla_{y} \chi_{\Omega}  }{\chi_{\Omega}}\Big|.
\end{align*}

The orthogonality conditions imposed on $e^{-\frac{1}{8}|y|^2}\chi_{\Omega}w$ in \eqref{eq:orthow} imply those for $e^{-\frac{1}{8}|y|^2}\chi_{\Omega}w_0:$
\begin{align}
\begin{split}
e^{-\frac{1}{8}|y|^2}\chi_{\Omega}w_0\perp  e^{-\frac{1}{8} |y|^2},\ & e^{-\frac{1}{8} |y|^2} y_k,\ e^{-\frac{1}{8} |y|^2}(\frac{1}{2}y_k^2-1),\  k=1,2,3,\\
& e^{-\frac{1}{8} |y|^2} y_m y_n, \ m\not=n, \ \ m,n=1,2,3.
\end{split}
\end{align}

Define the orthogonal projection onto the $L^2$ subspace orthogonal to these 13 functions by $P_{13}$, which makes
\begin{align*}
P_{13}e^{-\frac{1}{8}|y|^2}\chi_{\Omega}w_0=e^{-\frac{1}{8}|y|^2}\chi_{\Omega}w_0.
\end{align*}

Apply the operator $P_{13}$ on \eqref{eq:selfadw}, and then apply Duhamel's principle to have
\begin{align}
\begin{split}\label{eq:3w}
e^{-\frac{1}{8}|y|^2}\chi_{\Omega}w_0=&U_1(\tau, \xi_0) e^{-\frac{1}{8}|y|^2}\chi_{\Omega}w_{0}(\xi_0)+
\int_{\xi_0}^{\tau}U_1(\tau,s) P_{13}e^{-\frac{1}{8}|y|^2}
\Lambda(w_0)(s)\ ds\\
&+\int_{\xi_0}^{\tau}U_1(\tau,s) P_{13}e^{-\frac{1}{8}|y|^2}\chi_{\Omega}\Big(\Sigma(s)+\frac{1}{2\pi}\langle N_1+N_2, \ 1\rangle_{\theta}(s) \Big)\ ds,
\end{split}
\end{align} where $U_1(\sigma_1,\sigma_2)$ is the propagator generated by the linear operator $-P_{13}\mathcal{L}P_{13}$ from $\sigma_2$ to $\sigma_1$, with $\sigma_1\geq \sigma_2.$

The propagator satisfies the following estimate:
\begin{lemma}\label{LM:propagator}
There exists a constant $C$, such that for any function $g$ and for any times $\sigma_1\geq \sigma_2\geq \xi_0$, we have that
\begin{align}
\|\langle y\rangle^{-3} e^{\frac{1}{8}|y|^2} U_1(\sigma_1,\sigma_2) P_{13}g\|_{\infty}\leq C e^{-\frac{2}{5}(\sigma_1-\sigma_2)} \|\langle y\rangle^{-3}e^{\frac{1}{8}|y|^2} g\|_{\infty}.
\end{align}

\end{lemma}
\begin{proof}
Recall that
\begin{align}
\mathcal{L}=-\Delta_y+\frac{1}{16}|y|^2-\frac{3}{4}-a-\frac{1}{2}-\tau^{-\frac{1}{2}} +V_1+V_2
\end{align} with the multipliers $V_{1}$ and $V_2$ are defined as
\begin{align*}
V_1:=\Big[\frac{a y^{T}B y}{2+y^{T}By}+\tau^{-\frac{1}{2}}\Big]\chi_{2\Omega},\ V_2:=\frac{1}{2}\Big|\frac{\tilde\chi_{\Omega}\ y\cdot \nabla_{y} \chi_{\Omega}  }{\chi_{\Omega}}\Big|,
\end{align*}

In the previous papers \cite{BrKu, DGSW, GS2008, GaKnSi, GaKn20142, MultiDHeat}, the identical results were proved for the propagator generated by the linear operator $\mathcal{L}_1$, defined as
\begin{align*}
\mathcal{L}_1:=-\Delta_y+\frac{1}{16}|y|^2-\frac{3}{4}-a-\frac{1}{2} +\frac{ay^{T}B_1y}{2+y^{T}B_1 y},
\end{align*} with $B_1$ is positive definite: $B_{1}=\tau^{-1}Id+\mathcal{O}(\tau^{-2})$. Then in \cite{GZ2017}, we prove the same results for the linear operator
\begin{align*}
\mathcal{L}_2:=-\Delta_y+\frac{1}{16}|y|^2-\frac{3}{4}-1 +V_2.
\end{align*}

By the same arguments in \cite{GZ2017}, one can prove the same result for the present operator $\mathcal{L}$, resulted by that both $V_1$ and $V_2$ are favorably nonnegative:
the estimates on $B$ and $a$ in \eqref{eq:Best} and \eqref{eq:10parame} imply that $V_1$ is non-negative; the function $V_2$ is non-negative by definition.
Moreover we have that $\|(1+|y|)^{-5}V_k(\tau)\|_{\infty},\ |\nabla_y V_k|\rightarrow 0$ as $\tau\rightarrow \infty,\ k=1,2$. 

Thus here we choose to skip the proof.

\end{proof}

Returning to the equation for $e^{-\frac{1}{8}|y|^2}\chi_{\Omega}w_0$ in \eqref{eq:3w}, we apply the propagator estimate to find
\begin{align}
\begin{split}\label{eq:3wEnd}
\|\langle y\rangle^{-3}\chi_{R}w_0(\cdot,\tau)\|_{\infty}\lesssim &e^{-\frac{2}{5}(\tau-\xi_0)}\|\langle y\rangle^{-3}\chi_{R}w_0(\cdot,\xi_0)\|_{\infty}+\int_{\xi_0}^{\tau} e^{-\frac{2}{5}(\tau-s)} \|\langle y\rangle^{-3} \Lambda(w_0)(s)\|_{\infty}\ ds\\
&+
\int_{\xi_0}^{\tau} e^{-\frac{2}{5}(\tau-s)} \|\langle y\rangle^{-3}(\Sigma+ \frac{1}{2\pi}\langle N_1+N_2,\ 1\rangle_{\theta})(s)\|_{\infty}\ ds.
\end{split}
\end{align}

The terms on the right hand side satisfy the following estimates:
\begin{proposition}\label{Prop:weight3}
\begin{align}
\|\langle y\rangle^{-3}\Lambda(w_0)\|_{\infty} \lesssim & \delta\kappa(\epsilon) \Omega^{-4},\label{eq:est3PsiW0}\\
\|\langle y\rangle^{-3}\chi_{R}\Sigma\|_{\infty}\lesssim &\tau^{-2},\label{eq:est3FG}\\
\|\langle y\rangle^{-3}\chi_{R}N_1\|_{\infty}\lesssim & \delta \big(\tau^{-2} +\kappa(\epsilon) \Omega^{-4}\big) P(M), \label{eq:est3N1}\\
\|\langle y\rangle^{-3}\chi_{R}\langle N_2,\ 1\rangle_{\theta}\|_{\infty}\lesssim &\delta \big(\tau^{-2} +\kappa(\epsilon) \Omega^{-4}\big)\mathcal{M}_1.\label{eq:est3N2}
\end{align}

\end{proposition}
The proposition will be proved in Subsection \ref{subsec:weight3}. Here $\kappa(\epsilon)$ and $P(M)$ are defined in \eqref{eq:defKappa} and \eqref{eq:defOmega} respectively.

We continue to study \eqref{eq:3wEnd}. Apply the estimates in Proposition \ref{Prop:weight3} and the estimate $$\|\langle y\rangle^{-3}\chi_{\Omega}w(\cdot,\xi_0)\|_{\infty}\lesssim \kappa(\epsilon)\Omega^{-4}(\xi_0)$$ from \eqref{eq:IniWeighted} to obtain the desired result,
\begin{align}
\begin{split}\label{eq:integ}
\|\langle y\rangle^{-3}\chi_{\Omega}w_0(\cdot,\tau)\|_{\infty}
\lesssim &e^{-\frac{2}{5}(\tau-\xi_0)}\Omega^{-4}(\xi_0) +\big(\tau^{-2}+\kappa(\epsilon) \Omega^{-4} \big)\big(1+\delta P(M)\big)\\
\lesssim & \big(\tau^{-2}+\kappa(\epsilon) \Omega^{-4} \big)\big(1+\delta P(M)\big),
\end{split}
\end{align}
here we use the following facts: (1) $\mathcal{M}_k,\ k=1,2,3,4,$ are increasing functions, (2) for any $k>0,$ there exists a constant $C_k$ such that
\begin{align}
\int_{\xi_0}^{\tau}  e^{-\frac{2}{5}(\tau-s)} \Omega^{-k}(s)\ ds \leq C_{k} \Omega^{-k}(\tau),\ \ \text{and}\ 
\ \int_{\xi_0}^{\tau}  e^{-\frac{2}{5}(\tau-s)} s^{-2}\ ds \lesssim  \tau^{-2}.\label{eq:Lhos1}
\end{align}

To prove the first estimates in \eqref{eq:Lhos1}, we find a function equivalent to the function $\Omega^{-k}$, namely there exist constants $C_k$ such that, for $\tau\geq \xi_0,$
\begin{align}
\frac{1}{C_k}\leq \frac{\Omega^{-k}(\tau)}{ \min\{\Omega^{-k}(\xi_0), \big(1+\tau-\xi_0\big)^{-\frac{11 }{20}k}\}}\leq C_k.\label{eq:comparable}
\end{align}
Compute directly to obtain
\begin{align}
\int_{\xi_0}^{\tau} e^{-\frac{2}{5}(\tau-s)} \Omega^{-k}(\xi_0)ds\lesssim \Omega^{-k}(\xi_0),\label{eq:mini1}
\end{align} and apply L$'$Hospital's rule to obtain, for some $a_k>0,$
\begin{align}
\int_{\xi_0}^{\tau} e^{-\frac{2}{5}(\tau-s)} \Big(2+s-\xi_0\Big)^{-\frac{11}{20}k} ds\leq a_k\Big(2+\tau-\xi_0\Big)^{-\frac{11}{20}k}.\label{eq:mini2}
\end{align}

We take the minimum of these two estimates \eqref{eq:mini1} and \eqref{eq:mini2}. This, together with \eqref{eq:comparable}, implies the desired first estimate in \eqref{eq:Lhos1}. 

The second estimate in \eqref{eq:Lhos1} will be proved similarly, hence we skip the details here.

\begin{flushright}
$\square$
\end{flushright}


\subsubsection{Proof of Proposition \ref{Prop:weight3}}\label{subsec:weight3}

\begin{proof}
As said earlier, here we estimate the different terms in detail, since the techniques will be used repeated in the rest of the paper.

To prove \eqref{eq:est3PsiW0}, we recall the definition of $\Lambda(w_0)$ in \eqref{eq:defPsiw}.
For the first term, we use \eqref{eq:TwoCut} and the fact that $|\nabla_{y}\chi_{\Omega}|$ is supported by the set $|y|\in [\Omega,\ (1+\epsilon)\Omega]$, and $|w_0|\lesssim \delta \sqrt{1+\tau^{-1}|y|^2}$ in \eqref{eq:SmCon}, to find
\begin{align}
\langle y\rangle^{-3} |1-\tilde\chi_{\Omega}|\ |y\cdot \nabla_{y}\chi_{\Omega}|\ |w_0|\lesssim \kappa(\epsilon)\Omega^{-7} 1_{|y|\leq (1+\epsilon)\Omega} |w_0|\lesssim \delta \kappa(\epsilon)\Omega^{-6}.
\end{align}
where, recall the definitions of $\Omega$ and $\kappa(\epsilon)$ in \eqref{eq:defOmega} and \eqref{eq:defKappa}. For the others, besides applying the techniques above, the derivatives of $\chi_{\Omega}$ contribute some decay estimates, for example $$|\nabla_{y}\chi_{\Omega}|\leq \Omega^{-1} \chi'(\frac{|y|}{\Omega}).$$
This, together with that it is supported by the set $|y|\in [\Omega,\ (1+\epsilon)\Omega]$, and that $|\nabla_{y}w_0|\leq \delta$, $|w_0|\lesssim \delta \sqrt{1+\tau^{-1}|y|^2}$ by \eqref{eq:SmCon}, implies the desired result,
\begin{align}
\begin{split}
\|\langle y\rangle^{-3}\Lambda(w_0)\|_{\infty} \lesssim & \delta\sqrt{1+\tau^{-1}\Omega^2}\Big[ \|\langle y\rangle^{-3}\partial_{\tau}\chi_{\Omega}
\|_{\infty}+\|\langle y\rangle^{-3}\Delta_{y}\chi_{\Omega}\|_{\infty}\Big]+ \delta \kappa(\epsilon)\Omega^{-4}\\
\lesssim &\delta \kappa(\epsilon)\Omega^{-4}.
\end{split}
\end{align}

It is easy to prove \eqref{eq:est3FG}, by the estimates on the scalar functions in \eqref{eq:Best}-\eqref{eq:scalarEqn}.

To prove \eqref{eq:est3N1}, we observe that
\begin{align}\label{eq:estN1}
|N_1(v)|&\lesssim \delta |\nabla_y v|^2+v^{-1}|\partial_{\theta}v|^2.
\end{align}
Decompose $v$, and use the estimates in \eqref{eq:Best}-\eqref{eq:betaA}
to find
\begin{align}
\begin{split}\label{eq:twoV}
|\nabla_y v|^2&\lesssim  \tau^{-2}|y|^2V_{a,B}^{-2} +(1+|y|^2)\tau^{-4}+ |\nabla_y w|^2,\\
&\\
 v^{-1}|\partial_{\theta}v|^2 &\lesssim (1+|y|^2)\tau^{-4}+v^{-1}|\partial_{\theta}w|^2.
\end{split}
\end{align}

We claim that
\begin{align}
\|\langle y\rangle^{-3}\chi_{\Omega} |\nabla_{y}w|^2\|_{\infty},\ \|\langle y\rangle^{-3}\chi_{\Omega} v^{-1}|\partial_{\theta}w|^2\|_{\infty}\lesssim &\delta \kappa(\epsilon) \Omega^{-4} P(M).\label{eq:y3nablaw2}
\end{align}

Suppose these hold, then these together with \eqref{eq:estN1} imply the desired estimates \eqref{eq:est3N1}.

Now we prove \eqref{eq:y3nablaw2}. Compare to proving the first one, it is easier to prove the second, since the two operators $\partial_{\theta}$ and $\chi_{\Omega}$ commute. 
Hence we skip proving the latter here.

To prove the first one, compute directly to obtain
\begin{align}
\|\langle y\rangle^{-3}\chi_{\Omega} |\nabla_{y}w|^2\|_{\infty}\lesssim \|\langle y\rangle^{-2}\chi_{\Omega} \nabla_{y}w\|_{\infty}\|\langle y\rangle^{-1} 1_{\leq (1+\epsilon)\Omega}\nabla_{y}w\|_{\infty},\label{eq:twofacto}
\end{align} where $1_{\leq (1+\epsilon)\Omega}$ is the Heaviside function taking value $1$ for $|y|\leq (1+\epsilon)\Omega$, and $0$ otherwise.

To control the first factor, we change the order of $\nabla_{y}$ and $\chi_{\Omega}$ to find
\begin{align}
\chi_{\Omega} \nabla_{y} w =\nabla_{y}(\chi_{\Omega}w)- w \nabla_{y}\chi_{\Omega}.\label{eq:order10}
\end{align} 
Apply Lemma \ref{LM:appliM1234}
 to control the first term
\begin{align}
\|\langle y\rangle^{-2} \nabla_{y}(\chi_{\Omega}w)\|_{\infty}\lesssim \kappa(\epsilon) \Omega^{-3}\mathcal{M}_2.\label{eq:applyM2}
\end{align} 

It is more involved to estimate the second term in \eqref{eq:order10}. Since similar arguments will be used repeatedly in the rest of the paper, we make a detailed presentation here. The difficulty is that, due to technical reasons, we need a factor of cutoff function $\chi_{\Omega}$ in $\chi_{\Omega}w$ to prove it decays rapidly; without this, we can not prove that $w$ decays. In the present situation, for the term $w\nabla_{y}\chi_{\Omega}$, we do not have a factor $\chi_{\Omega}$ since $\sup_{z}\{|\frac{\nabla_{z}\chi(z)}{ \chi(z)}|\}=\infty$. Here we try to have some fractional power of it. 

The detailed estimate is the following:
\begin{align}\label{eq:fractional}
\begin{split}
\langle y\rangle^{-2}|w \nabla_{y}\chi_{\Omega}|\leq &\Omega^{-3} |\chi_{\Omega}w|^{\frac{3}{4}}   |w|^{\frac{1}{4}}
\sup_{z}\{|\frac{\nabla_{z}\chi(z)}{ \chi^{\frac{3}{4}}(z)}|\}\\
\lesssim &\kappa(\epsilon) \delta^{\frac{1}{4}} (1+\tau^{-1}\Omega^2)^{\frac{1}{8}}\Omega^{-3}\Big[\Omega^{3}\Big(\kappa(\epsilon)\Omega^{-4}+\tau^{-2}\Big)\mathcal{M}_1 \Big]^{\frac{3}{4}}\\
\lesssim &\Omega^{-\frac{1}{2}}\big(\kappa(\epsilon)\Omega^{-4}+\tau^{-2}\big)^{\frac{3}{4}}\mathcal{M}_1^{\frac{3}{4}}
\end{split}
\end{align}
based on the following facts,
\begin{itemize}
\item[(1)] in the first step we use that $\nabla_{y}\chi_{\Omega}=\Omega^{-1} (\nabla_{z}\chi)|_{z=\frac{y}{\Omega}}$ and it is supported by the set $|y|\in [\Omega, \ (1+\epsilon)\Omega]$, thus in this region we have $\langle y\rangle^{-2}\leq \Omega^{-2},$
\item[(2)] in the second step we use that $|\chi^{-\frac{3}{4}}\ \nabla_{z}\chi |\leq \kappa(\epsilon)$ in \eqref{eq:defKappa}; and $|w|\lesssim \delta V_{a,B}\lesssim \delta \sqrt{1+\tau^{-1}\Omega^2}$ by \eqref{eq:SmCon}; and in the considered region $$|\chi_{\Omega}w|\lesssim \Omega^3 \|\langle y\rangle^{-3}\chi_{\Omega}w\|_{\infty}\lesssim \Omega^{3}\Big(\kappa(\epsilon)\Omega^{-4}+\tau^{-2}\Big)\mathcal{M}_1 .
$$ 
\item[(3)] in the last step we use that $\kappa(\epsilon) \delta^{\frac{1}{4}} (1+\tau^{-1}\Omega^2)^{\frac{1}{8}}\Omega^{-\frac{1}{4}}\leq 1$ by the condition in \eqref{eq:assum} and the definition of $\Omega$ in \eqref{eq:defOmega}. 
\end{itemize}

\eqref{eq:order10}-\eqref{eq:fractional} imply that 
the first factor in in \eqref{eq:twofacto} satisfies the estimate
\begin{align}\label{eq:nablanablaw}
\|\langle y\rangle^{-2}\chi_{\Omega} \nabla_{y}w\|_{\infty}
\lesssim  \kappa(\epsilon)\Omega^{-3}\mathcal{M}_2+ \Omega^{-3-\frac{1}{4}}\mathcal{M}_1^{\frac{3}{4}}.
\end{align}

Now we estimate the second factor in \eqref{eq:twofacto}, which is $\|\langle y\rangle^{-1} 1_{\leq (1+\epsilon)\Omega}\nabla_{y}w\|_{\infty}$. Insert the identity $1=\chi_{\Omega}+1-\chi_{\Omega}$ before $\nabla_y w$ to find
\begin{align}
\begin{split}\label{eq:Hea1yNw}
\|\langle y\rangle^{-1} 1_{\leq (1+\epsilon)\Omega}\nabla_{y}w\|_{\infty}\leq& \|\langle y\rangle^{-1} \chi_{\Omega}\nabla_{y}w\|_{\infty}+\|\langle y\rangle^{-1} 1_{\leq (1+\epsilon)\Omega}(1-\chi_{\Omega})\nabla_{y}w\|_{\infty}\\
\lesssim &\Omega \|\langle y\rangle^{-2} \chi_{\Omega}\nabla_{y}w\|_{\infty}+\delta\Omega^{-1}\\
\lesssim &\kappa(\epsilon)\Omega^{-2}\mathcal{M}_2+ \Omega^{-2-\frac{1}{4}}\mathcal{M}_1^{\frac{3}{4}}+\delta\Omega^{-1},
\end{split}
\end{align} where, in the second step we use that the cutoff function $\chi_{\Omega}$ is supported on the set $|y|\leq (1+\epsilon)\Omega$, and use $|\nabla_{y}w|\lesssim \delta$ by \eqref{eq:SmCon}, and that $\langle y\rangle^{-1} 1_{\leq (1+\epsilon)\Omega}(1-\chi_{\Omega})\leq \Omega^{-1}$, and in the last step we take the estimate from \eqref{eq:nablanablaw}.

This and \eqref{eq:nablanablaw} imply the desired \eqref{eq:y3nablaw2}.

To prove \eqref{eq:est3N2} we observe some cancellations. By the definition of $N_{2}(\eta)$ in \eqref{eq:defN2eta},
\begin{align}
\langle N_2(\eta),\ 1\rangle_{\theta}=&-\langle V_{a,B}^{-2}v^{-1}\eta^2,\ 1\rangle_{\theta}+\langle (v^{-2}-V_{a,B}^{-2})\partial_{\theta}^2 \eta,\ 1\rangle_{\theta}\nonumber\\
=&-\langle V_{a,B}^{-2}v^{-1}\eta^2,\ 1\rangle_{\theta}+2\langle v^{-3}(\partial_{\theta}\eta)^2,\ 1\rangle_{\theta},\label{eq
:RewN1Int1}
\end{align}
where we use that $\langle V_{a,B}^{-2}\partial_{\theta}^2 \eta,\ 1\rangle_{\theta}=0$ since $V_{a,B}^{-2}$ is independent of $\theta$, and observe $\langle v^{-2}\partial_{\theta}^2 \eta,\ 1\rangle_{\theta}=2\langle v^{-3}(\partial_{\theta}\eta)^2,\ 1\rangle_{\theta}$ after integrating by parts in $\theta$ and using that $\partial_{\theta}v=\partial_{\theta}\eta$. 

Apply the estimates in \eqref{eq:SmCon}, decompose $\eta$ as in \eqref{eq:decomW}, and apply \eqref{eq:embeddM1} again to have the desired result,
\begin{align}\label{eq:nonl}
\|\langle y\rangle^{-3} \chi_{\Omega}\langle N_2(\eta),\ 1\rangle_{\theta} \|_{\infty}\lesssim &\delta\sum_{l=0,1} \| \langle y\rangle^{-3} \chi_{\Omega}\partial_{\theta}^{l} \eta \|_{\infty}
\lesssim  \delta \Big[\tau^{-2}+ \kappa(\epsilon)\Omega^{-4}\Big](1+\mathcal{M}_1).
\end{align}

\end{proof}

\subsection{Proof of \eqref{eq:3020wpm}}
Since $w_{-1}=\overline{w_1}$, we only need to estimate $w_1.$

We derive an equation for $w_{1}$ from \eqref{eq:tildew3}, 
\begin{align}
\partial_{\tau}\chi_{\Omega}w_1
=-\big(-\Delta_{y}+\frac{1}{2}y\cdot \nabla_{y}-\frac{1}{2}\big) \chi_{\Omega}w_{1}+ \frac{1}{2\pi}\chi_{\Omega}\Big\langle G+N_{1}(v)+N_{2}(\eta), \ e^{i\theta}\Big\rangle_{\theta}+\mu(w_1),
\end{align} where the linear operator is derived from the operator $L$ in \eqref{eq:tildew3} by
\begin{align}
\frac{1}{2\pi}\langle L w, \ e^{i\theta}\rangle_{\theta}= (-\Delta_{y}+\frac{1}{2}y\cdot \nabla_{y}-\frac{1}{2}) w_{1},
\end{align}
resulted by a cancellation $\langle -V_{a,B}^{-2} \partial_{\theta}^2 w-V_{a,B}^{-2}w,\ e^{i\theta}\rangle_{\theta}=0$ since $V_{a,B}$ is independent of $\theta$, and also we use that $\langle F(B,a),\ e^{i\theta}\rangle_{\theta}=0$ since $F(B,a)$ is independent of $\theta.$

As discussed in \eqref{eq:unboundtwo}, to control the difficult term $\frac{1}{2} (y\cdot \nabla_{y}\chi_{\Omega} )w_1$ in $\mu(w_1)$, we decompose it into two parts
\begin{align}
\frac{1}{2} (y\cdot \nabla_{y}\chi_{\Omega} )w_1= \frac{1}{2} (y\cdot \nabla_{y}\chi_{\Omega} )\tilde\chi_{\Omega} w_1+\frac{1}{2} (y\cdot \nabla_{y}\chi_{\Omega} )(1-\tilde\chi_{\Omega}) w_1
\end{align} and move the first part into the linear operator and the second to the remainder.
This makes
\begin{align}\label{eq:eqnw1}
\partial_{\tau}\chi_{\Omega}w_1
=&-H_{1} w_{1}+ \frac{1}{2\pi}\chi_{\Omega}\langle G+N_{1}(v)+N_{2}(\eta), \ e^{i\theta}\rangle_{\theta}+\Lambda(w_1),
\end{align}
where the linear operator $H_{1}$ is defined as
\begin{align*}
H_{1}:=-\Delta_{y}+\frac{1}{2}y\cdot \nabla_{y}-\frac{1}{2}+\frac{1}{2}\Big|\frac{\tilde\chi_{\Omega}\ y\cdot \nabla_{y} \chi_{\Omega}  }{\chi_{\Omega}}\Big|
\end{align*} and $\Lambda(w_1)$ is defined in the same way to that $\Lambda(w_0)$ in \eqref{eq:defPsiw}.

The orthogonality conditions imposed on $\chi_{\Omega}w$ in \eqref{eq:orthow} imply that
\begin{align}
e^{-\frac{1}{8}|y|^2} \chi_{\Omega}w_1\perp e^{-\frac{1}{8}|y|^2},\ y_k e^{-\frac{1}{8}|y|^2},\ k=1,2,3.
\end{align}
We denote, by $P_4$, the orthogonal projection onto the subspace orthogonal to these four functions, which makes
\begin{align*}
P_4 e^{-\frac{1}{8}|y|^2} \chi_{\Omega}w_1=e^{-\frac{1}{8}|y|^2} \chi_{\Omega}w_1.
\end{align*}

On \eqref{eq:eqnw1} we apply $e^{-\frac{1}{8}|y|^2}$, and then $P_4$, and then Duhamel's principle to have
\begin{align}
\begin{split}\label{eq:durP4w1}
e^{-\frac{1}{8}|y|^2} \chi_{\Omega}w_1(\tau)=&U_2(\tau,\xi_0) e^{-\frac{1}{8}|y|^2} \chi_{\Omega}w_1(\xi_0)\\ 
&+\int_{\xi_0}^{\tau} U_{2}(\tau,\sigma)P_4\Big[ \frac{1}{2\pi}\chi_{\Omega}\langle G+N_{1}(v)+N_{2}(\eta), \ e^{i\theta}\rangle_{\theta}+\Lambda(w_1)\Big](\sigma)\ d\sigma.
\end{split}
\end{align} where $U_2(\tau,\sigma)$ is the propagator generated by $-P_4e^{-\frac{1}{8}|y|^2}H_{1}e^{\frac{1}{8}|y|^2}P_4$ from $\sigma$ to $\tau.$

The propagator satisfies the following estimates, as discussed in the proof of Lemma \ref{LM:propagator}, its proof is very similar to the proved cases, thus we choose to skip it.
\begin{lemma} For $l=2,3,$ and any function $g,$
\begin{align}\label{eq:propa}
\|\langle y\rangle^{-l} e^{\frac{1}{8}|y|^2}U_{2}(\tau,\sigma)P_{4}g\|_{\infty}\lesssim e^{-\frac{2}{5}(\tau-\sigma)} \|\langle y\rangle^{-l} e^{\frac{1}{8}|y|^2}g\|_{\infty}.
\end{align}
\end{lemma}

Now we estimate the terms on the right hand side of \eqref{eq:durP4w1}, recall that $P(M)$ is defined in \eqref{eq:defPM},
\begin{proposition}\label{prop:0yW1}
For $l=2,3,$
\begin{align}
\|\langle y\rangle^{-l} \langle N_1(v), \ e^{i\theta}\rangle_{\theta}\|_{\infty}\lesssim &\delta \tau^{-2}+\delta \kappa(\epsilon) \Omega^{-l-1} P(M),\label{eq:N1Theta1}\\
\|\langle y\rangle^{-l}\langle N_{2}(\eta), \ e^{i\theta}\rangle_{\theta}\|_{\infty}\lesssim &\tau^{-4}+\delta \kappa(\epsilon) \Omega^{-l-1}P(M),\label{eq:N2Theta1}\\
\|\langle y\rangle^{-l}\Lambda(w_1)\|_{\infty} \lesssim & \tilde\delta\kappa(\epsilon) \Omega^{-l-1},\label{eq:est3PsiW1}\\
\|\langle y\rangle^{-l} \langle G, \ e^{i\theta}\rangle_{\theta}\|_{\infty}\lesssim & \tau^{-2}.\label{eq:3GTheta}
\end{align}

\end{proposition}
The proposition will be proved in subsubsection \ref{subsub:0yW1}.

Suppose the proposition holds, then we prove the desired results for $\mathcal{M}_1$ and $\mathcal{M}_4$ in \eqref{eq:21w} as in \eqref{eq:integ}. Here we choose to skip the details.

\subsubsection{Proof of Proposition \ref{prop:0yW1}}\label{subsub:0yW1}
\begin{proof}
In what follows we only consider the case $l=3.$ The proof for $l=2$ is considerably easier since the need decay estimate is slower, hence easier to prove. We skip that part.

For \eqref{eq:N1Theta1}, since
$
\|\langle y\rangle^{-3} \langle N_1, \ e^{i\theta}\rangle_{\theta}\|_{\infty}\lesssim \|\langle y\rangle^{-3}N_1 \|_{\infty},
$ and the latter was estimated in \eqref{eq:est3N1} satisfactorily, we take that as the desired results.

The proof of \eqref{eq:est3PsiW1} is identical to that of \eqref{eq:est3PsiW0}, hence we skip the details here.

For \eqref{eq:N2Theta1},
it is important to observe some cancellations,
\begin{align}
\begin{split}\label{eq:N2Theta}
\langle N_{2}(\eta), \ e^{i\theta}\rangle_{\theta}=&\langle -v^{-1}+V_{a,B}^{-1}-V_{a,B}^{-2}\eta+\big(v^{-2}-V_{a,B}^{-2}\big)\partial_{\theta}^2 \eta, \ e^{i\theta}\rangle_{\theta}\\
=&\langle -v^{-1}+v^{-2} \partial_{\theta}^2 \eta, \ e^{i\theta}\rangle_{\theta}
=-2\langle v^{-3}(\partial_{\theta}\eta)^2, e^{i\theta}\rangle_{\theta},
\end{split}
\end{align} where we use that $\langle -V_{a,B}^{-2}\eta-V_{a,B}^{-2}\partial_{\theta}^2 \eta, \ e^{i\theta}\rangle_{\theta}=0$ and $\langle V_{a,B}^{-1}, \ e^{i\theta}\rangle_{\theta}=0$ by that $V_{a,B}$ is independent of $\theta$; we integrate by parts in $\theta$ to find $\langle v^{-1}, \ e^{i\theta}\rangle_{\theta}=-\langle \partial_{\theta}^2 v^{-1}, \ e^{i\theta}\rangle_{\theta}$, then use that $\partial_{\theta}v=\partial_{\theta}\eta$.

Decompose $\eta$ as in \eqref{eq:decomW} and apply Young's inequality to obtain
\begin{align}
\begin{split}
|\langle v^{-3}(\partial_{\theta}\eta)^2,\ 1\rangle_{\theta}\chi_{\Omega}| \lesssim (1+|y|^2)\tau^{-4}+  \chi_{\Omega}\langle v^{-3}(\partial_{\theta} w)^2,\ 1\rangle_{\theta}.
\end{split}
\end{align}
Then apply the results in \eqref{eq:y3nablaw2} on the last term to have the desired results.

For \eqref{eq:3GTheta}, the desired estimate follows from the estimates in \eqref{eq:Best}-\eqref{eq:scalarEqn}.

\end{proof}

\subsection{Proof of \eqref{eq:3020Wtheta}}\label{subsec:3020Wtheta}

In what follows we prove the estimate for $\Big\|(100+|y|^2)^{-\frac{3}{2}}\|\partial_{\theta}^3 P_{\theta, \geq 2} \chi_{\Omega}w\|_{L_{\theta}^2}\Big\|_{\infty}$. It is easier to estimate $\Big\|(100+|y|^2)^{-1}\|\partial_{\theta}^3 P_{\theta, \geq 2} \chi_{\Omega}w\|_{L_{\theta}^2}\Big\|_{\infty}$ since its desired decay estimate is considerably slower. Hence we skip this part.

The main tool here is the maximum principle.
To make it applicable we start with deriving an equation for $\partial_{\theta}^3 P_{\theta, \geq 2} \chi_{\Omega}w$ from the equation for $v$ in \eqref{eq:scale1}. We do not start from the equation for $w$ since \eqref{eq:scale1} makes it easier to observe some positivity and cancellations.

To simplify the notation we define a new function $\Phi_3$ as
\begin{align}
\Phi_3:=(100+|y|^2)^{-3}\langle P_{\theta,\geq 2}\partial_{\theta}^3\chi_{\Omega} v,\ \partial_{\theta}^3\chi_{\Omega} v\rangle_{\theta}=(100+|y|^2)^{-3}\langle P_{\theta,\geq 2}\partial_{\theta}^3\chi_{\Omega} w,\ \partial_{\theta}^3\chi_{\Omega} w\rangle_{\theta}, 
\end{align} 
where we use that, by the decomposition of $v$ in \eqref{eq:decomVToW},
\begin{align}
P_{\theta, \geq 2} w=P_{\theta, \geq 2} v.\label{eq:P2wv}
\end{align}

$\Phi_3$ satisfies the equation
\begin{align}
\begin{split}\label{eq:Phi3eqn}
\partial_{\tau}\Phi_3=-(L_3+V_3) \Phi_{3}-2(100+|y|^2)^{-3}\|P_{\theta,\geq 2}\partial_{\theta}^3\nabla_{y}\chi_{\Omega}v\|_{L^2_{\theta}}^2 +2\sum_{k=1}^3\Psi_{3k},
\end{split}
\end{align}
where the linear operator $L_3+V_3$ is related to $-\Delta+\frac{1}{2}y\cdot \nabla_{y}-1$ by $$L_3+V_3:=(100+|y|^2)^{-3} (-\Delta+\frac{1}{2}y\cdot \nabla_{y}-1) (100+|y|^2)^{3},$$ and $L_3$ is a differential operator, and $V_3$ is a multiplier, defined as
\begin{align}
\begin{split}
L_3:=& -\Delta+\frac{1}{2}y\cdot \nabla_{y}-2(100+|y|^2)^{-3} \big(\nabla_y (100+|y|^2)^{3}\big)\cdot \nabla_{y},\\
V_{3}:=&-1+\frac{3|y|^2}{100+|y|^2}-\frac{18}{100+|y|^2}-\frac{24|y|^2}{(100+|y|^2)^2}.
\end{split}
\end{align} The functions $\Psi_{3k},\ k=1,2,3,$ are defined as
\begin{align}
\begin{split}
\Psi_{31}:=&(100+|y|^2)^{-3}\langle P_{\theta,\geq 2}\partial_{\theta}^3\chi_{\Omega} v,\ \partial_{\theta}^3\chi_{\Omega} \big(v^{-2}\partial_{\theta}^2 v-v^{-1}\big)
\rangle_{\theta},\\
\Psi_{32}:=&(100+|y|^2)^{-3}\langle P_{\theta,\geq 2}\partial_{\theta}^3\chi_{\Omega} v,\ \partial_{\theta}^3\chi_{\Omega} N_{1}(v)\rangle_{\theta},\\
\Psi_{33}:=&(100+|y|^2)^{-3}\langle P_{\theta,\geq 2}\partial_{\theta}^3\chi_{\Omega} v,\ \mu(P_{\theta,\geq 2}\partial_{\theta}^3 v)\rangle_{\theta}.
\end{split}
\end{align}
Here the $\mu-$term in $\Psi_{33}$ is defined in the same fashion to that in \eqref{eq:Tchi3}.

It is important to observe that $\Psi_{31}$ contains some negative terms and they will be used to cancel some difficult terms in $\Psi_{32}$ and $\Psi_{33}$. For the terms in $\Psi_{32}$ and $\Psi_{33}$ there are only two possibilities: (1) they can be cancelled by the negative terms in $\Psi_{31}$ and the second term in \eqref{eq:Phi3eqn}, (2) they decay rapidly.

The result is the following, recall that $P(M)$ is defined in \eqref{eq:defPM},
\begin{proposition}\label{prop:03w}
There exists some constant $C>0$ such that
\begin{align}
\Psi_{31}\leq &-(\frac{18}{25}-C\delta ) (100+|y|^2)^{-3}V_{a,B}^{-2}\| P_{\theta,\geq 2}\partial_{\theta}^4\chi_{\Omega} v\|^2_{L_{\theta}^2} +C\delta^2 \Big[ \tau^{-2}+\kappa(\epsilon) \Omega^{-4} \Big]^2 P^2(M),\label{eq:03wNega}\\
\Psi_{33}\leq &\frac{1}{100} \Phi_3+C\delta^2 \kappa^2(\epsilon)\Omega^{-8}P^2(M),\label{eq:03Lambda}\\
\Psi_{32}\leq &\frac{1}{100}(100+|y|^2)^{-3}\Big[V_{a,B}^{-2}\| P_{\theta,\geq 2}\partial_{\theta}^4\chi_{\Omega} v\|_{L_{\theta}^2}^2+\| P_{\theta,\geq 2}\partial_{\theta}^3\nabla_{y}\chi_{\Omega} v\|_{L_{\theta}^2}^2\Big]\nonumber\\
& +C\delta^2 \Big[ \tau^{-2}+\kappa(\epsilon) \Omega^{-4} \Big]^2P^2(M). \label{eq:03wN1}
\end{align}

\end{proposition}

The proposition will be proved in subsubsection \ref{subsub:03w}.

Assuming the proposition holds, we use the facts that $\|P_{\theta,\geq 2}\partial_{\theta}f\|_{L_{\theta}^2}^2\geq 4\|P_{\theta,\geq 2}f\|_{L_{\theta}^2}^2$ for any smooth function $f$, and that $\delta$ is sufficiently small, to find that, for some $C>0,$
\begin{align*}
2\sum_{k=1,2,3}\Psi_{3k}\leq &-\frac{28}{5} (100+|y|^2)^{-3}V_{a,B}^{-2}\| P_{\theta,\geq 2}\partial_{\theta}^3\chi_{\Omega} v\|_{L_{\theta}^2}+\frac{1}{50}(100+|y|^2)^{-3}\| P_{\theta,\geq 2}\partial_{\theta}^3\nabla_{y}\chi_{\Omega} v\|_{L_{\theta}^2}^2\\
&+\frac{1}{50} \Phi_3+C\delta^2 \Big[\tau^{-2}+\kappa(\epsilon)\Omega^{-4}\Big]^2 P^2(M).
\end{align*}

Return to the equation for $\Phi_{3}$ in \eqref{eq:Phi3eqn}, and find that it takes a new form
\begin{align}
\partial_{\tau}\Phi_3\leq -(L_3+V_3) \Phi_{3}-\Upsilon \Phi_3+C\delta^2 [\tau^{-2}+\kappa(\epsilon)\Omega^{-4}]^2 P^2(M),
\end{align} where $\Upsilon$ is a function defined as
\begin{align*}
\Upsilon(y):=\frac{28}{5}V_{a,B}^{-2}(y)-\frac{1}{50}+V_3(y).
\end{align*}
It is critically important that $\Upsilon$ is strictly positive. This is indeed true, since
in the region $|y|\leq \tau^{\frac{1}{4}}$, we have that $V_{a,B}\approx \frac{1}{\sqrt{2}}$ and $V_{3}\geq -\frac{3}{2}$, and when $|y|\geq \tau^{\frac{1}{4}}$, $V_3\approx 2$ and $V_{a,B}\geq 0$, consequently
\begin{align}
\Upsilon\geq 1.
\end{align}
Thus, for some constant $C>0,$
\begin{align}
\partial_{\tau}\Phi_3\leq -(L_3+1) \Phi_3+C\delta^2 \Big[\tau^{-2}+\kappa(\epsilon)\Omega^{-4}
\Big]^2 P^2(M).\label{eq:deriPhi3}
\end{align}

Before applying the maximum principle, we need to check the boundary condition. This is indeed easy since cutoff function $\chi_{\Omega}$ in the definition $\Phi_3$ makes $\Phi_{3}(y,\tau)=0$ if $|y|\geq (1+\epsilon)\Omega.$ 

A standard application of the maximum principle yields, 
\begin{align}
\begin{split}
\Phi_3(\tau)\lesssim & e^{-(\tau-\xi_0)}\Phi_3(\xi_0)+ \Big[\tau^{-2}+\kappa(\epsilon)\Omega^{-4}
\Big]^2 P^2(M)\\
\lesssim &\Big(1+\delta^2 P^2(M)\Big)\Big[\tau^{-2}+\kappa(\epsilon)\Omega^{-4}
\Big]^2
\end{split}
\end{align} 
where in the last step we use that $\Phi_3(\xi_0)\lesssim \kappa^2(\epsilon)\Omega^{-8}(\xi_0)$ by \eqref{eq:33}, and recall the definitions of $R(\tau)$ and $\Omega(\tau)$ in \eqref{eq:defRTau} and \eqref{eq:defOmega} respectively, and recall that we choose $\xi_0\gg \tau_0$.

This implies the desired estimate after taking a square root on both sides.

\subsubsection{Proof of Proposition \ref{prop:03w}}\label{subsub:03w}

\begin{proof}
We start proving \eqref{eq:03wNega} by separating the negative part of $\Psi_{31}$ from the rest,
\begin{align}
\begin{split}\label{eq:Psi0123}
\Psi_{31}(v)=&(100+|y|^2)^{-3}\Big[-\langle P_{\theta,\geq 2}\partial_{\theta}^4\chi_{\Omega} v,\ \partial_{\theta}^2\chi_{\Omega} v^{-2}\partial_{\theta}^2 v\rangle_{\theta}-\langle P_{\theta,\geq 2}\partial_{\theta}^3\chi_{\Omega} v,\ \partial_{\theta}^3\chi_{\Omega} v^{-1}\rangle_{\theta}
\Big]\\
=&-(100+|y|^2)^{-3}\Big[D_{1}+D_{2}+D_{3}\Big],
\end{split}
\end{align}
where in the first step and in the first term we integrate by parts in $\theta$, and the terms $D_l,\ l=1,2,3,$ are defined as
\begin{align*}
D_{1}:=&\langle P_{\theta,\geq 2}\partial_{\theta}^4\chi_{\Omega} v,\  v^{-2}P_{\theta,\geq 2}\partial_{\theta}^4 \chi_{\Omega}v\rangle_{\theta}-\langle P_{\theta,\geq 2}\partial_{\theta}^3\chi_{\Omega} v,\ v^{-2}P_{\theta,\geq 2}\partial_{\theta}^3\chi_{\Omega} v\rangle_{\theta},\\
D_{2}:=&\langle P_{\theta,\geq 2}\partial_{\theta}^4\chi_{\Omega} v,\  v^{-2}(1-P_{\theta,\geq 2})\partial_{\theta}^4 \chi_{\Omega}v\rangle_{\theta}+\langle P_{\theta,\geq 2}\partial_{\theta}^3\chi_{\Omega} v,\ v^{-2}(1-P_{\theta,\geq 2})\partial_{\theta}^3\chi_{\Omega} v\rangle_{\theta},\\
D_{3}:=&\langle P_{\theta,\geq 2}\partial_{\theta}^4\chi_{\Omega} v,\ \chi_{\Omega} \big[\partial_{\theta}^2 (v^{-2}\partial_{\theta}^2 v)-v^{-2}\partial_{\theta}^4v\big]\rangle_{\theta}-\langle P_{\theta,\geq 2}\partial_{\theta}^3\chi_{\Omega} v,\ \chi_{\Omega}(\partial_{\theta}^3 v^{-1}+v^{-2}\partial_{\theta}^3 v)\rangle_{\theta}.
\end{align*}

$D_1$ is important since it contains positive terms. Apply $|\frac{V_{a,B}}{v}-1|\lesssim \delta$ in \eqref{eq:SmCon}, and use that for any smooth function $f$, $\|P_{\theta,\geq 2}\partial_{\theta}^{4}f\|^2_{L_{\theta}^2}\geq 4 \|P_{\theta,\geq 2}\partial_{\theta}^{3}f\|^2_{L_{\theta}^2}$ to find, for some $C>0,$
\begin{align}
D_{1}\geq &(\frac{3}{4}-C\delta ) V_{a,B}^{-2}\Big\langle P_{\theta,\geq 2}\partial_{\theta}^4\chi_{\Omega} v,\ P_{\theta,\geq 2}\partial_{\theta}^4\chi_{\Omega} v\Big\rangle_{\theta}.\label{eq:negativeTerm}
\end{align}

For $D_2$, since the operator $(1-P_{\theta,\geq 2})\partial_{\theta}^{4}$ removes all the frequencies except $e^{\pm i\theta},$
\begin{align*}
(1-P_{\theta,\geq 2})\partial_{\theta}^4 \chi_{\Omega}v=\frac{1}{2\pi}\chi_{\Omega}\Big[e^{i\theta}\langle v, e^{i\theta}\rangle_{\theta}+e^{-i\theta}\langle v, e^{-i\theta}\rangle_{\theta}\Big]
\end{align*} thus
\begin{align}
\langle P_{\theta,\geq 2}\partial_{\theta}^4 \chi_{\Omega} v,\ v^{-2} (1-P_{\theta,\geq 2})\partial_{\theta}^4 \chi_{\Omega}v\rangle_{\theta}=\frac{1}{2\pi} \Big[K\chi_{\Omega} \langle v, \ e^{i\theta}\rangle_{\theta}+\chi_{\Omega}\overline{K \langle v, \ e^{-i\theta}\rangle_{\theta}}\Big].\label{eq:deltaK}
\end{align}
Here the term $K$ is defined as, 
\begin{align*}
K:=\langle P_{\theta,\geq 2}\partial_{\theta}^4 \chi_{\Omega} v,\ v^{-2} e^{i\theta}\rangle_{\theta}=\langle P_{\theta,\geq 2}\partial_{\theta}^4 \chi_{\Omega} v,\ P_{\theta,\geq 2}v^{-2} e^{i\theta}\rangle_{\theta}.
\end{align*} 
Since $v=V_{a,B}+\eta$, $V_{a,B}$ is independent of $\theta$, and $\frac{|\eta|}{V_{a,B}}\lesssim \delta$ in \eqref{eq:SmCon}, we have
\begin{align}\label{eq:estKdelta}
|P_{\theta,\geq 2}v^{-2} e^{i\theta}|\lesssim \delta V_{a,B}^{-2},\  \ \text{and hence}\ \ \ 
|K| \lesssim & \delta V_{a,B}^{-2} \|P_{\theta,\geq 2}\partial_{\theta}^4 \chi_{\Omega} v\|_{L_{\theta}^2}.
\end{align} 

Returning to \eqref{eq:deltaK}, we decompose $v$ and apply \eqref{eq:embeddM1} to find that
\begin{align*}
(100+|y|^2)^{-\frac{3}{2}}\chi_{\Omega}|\langle v, \ e^{i\theta}\rangle|\lesssim \tau^{-2}+(100+|y|^2)^{-\frac{3}{2}} \|\chi_{\Omega}w\|_{L_{\theta}^2}\lesssim (\tau^{-2}+\kappa(\epsilon)\Omega^{-4})(1+\mathcal{M}_1).
\end{align*}
These together with Young's inequality make, for some $C>0,$
\begin{align}
\begin{split}
&(100+|y|^2)^{-3}\Big|\langle P_{\theta,\geq 2}\partial_{\theta}^4 \chi_{\Omega} v,\ v^{-2} (1-P_{\theta,\geq 2})\partial_{\theta}^4 v\rangle_{\theta}\Big|\\
\leq &\frac{1}{100} V_{a,B}^{-2}(100+|y|^2)^{-3} \|P_{\theta,\geq 2}\partial_{\theta}^4 \chi_{\Omega} v\|_{L_{\theta}^2}^2+C\delta^2(\tau^{-2}+\kappa(\epsilon)\Omega^{-4})^2(1+\mathcal{M}_1)^2.
\end{split}
\end{align}

We estimate the second term in $D_2$ similarly and find
\begin{align}
 |(100+|y|^2)^{-3}D_{2}|\leq \frac{1}{50} (100+|y|^2)^{-3}V_{a,B}^{-2}& \|P_{\theta,\geq 2}\partial_{\theta}^4 \chi_{\Omega} v\|_{L_{\theta}^2}^2+C\delta^2 (1+\mathcal{M}_1^2) (\tau^{-2}+\kappa(\epsilon)\Omega^{-4})^2.\label{eq:Wei3D2}
\end{align}

To estimate $D_{3}$ we find parts of the integrands satisfy the estimate
\begin{align*}
v|\partial_{\theta}^2 (v^{-2}\partial_{\theta}^2 v)-v^{-2}\partial_{\theta}^4 v|,\ v|\partial_{\theta}^3 v^{-1} +v^{-2}\partial_{\theta}^3 v|\lesssim &\delta \sum_{k=1,2}|\partial_{\theta}^k v|\\
\lesssim &\delta\Big[ (1+|y|)\tau^{-2}+  \sum_{k=1,2} |\partial_{\theta}^k w|\Big].
\end{align*}

This, together with Lemma \ref{LM:poinEmbed} and Young's inequality,  implies, for some $C>0,$
\begin{align}
|(100+|y|^2)^{-3}D_{3}|\leq \frac{1}{100} (100+|y|^2)^{-3}V_{a,B}^{-2} \|P_{\theta,\geq 2}\partial_{\theta}^4 \chi_{\Omega} v\|_{L_{\theta}^2}^2+C\delta^2 \Big[ \tau^{-2}+\kappa(\epsilon) \Omega^{-4} \Big]^2\mathcal{M}_1^2.
\end{align}

This, together with \eqref{eq:Psi0123}, \eqref{eq:negativeTerm} and \eqref{eq:Wei3D2}, implies the desired estimate \eqref{eq:03wNega}.

Now we prove \eqref{eq:03wN1}. We divide $\Psi_{32}$ into two parts
\begin{align}
\Psi_{32}=&(100+|y|^2)^{-3}\Big[\langle P_{\theta,\geq 2}\partial_{\theta}^3\chi_{\Omega} v, \partial_{\theta}^3\chi_{\Omega} N_{11}\rangle_{\theta}+\langle P_{\theta,\geq 2}\partial_{\theta}^3\chi_{\Omega} v,\ \partial_{\theta}^3\chi_{\Omega} N_{12}\rangle_{\theta}\Big]\nonumber\\
:=&(100+|y|^2)^{-3}\Big[W_1+W_2\Big], \label{eq:psi32w12}
\end{align} where $W_l,\ l=1,2,$ are naturally defined, and $N_{1l},\ l=1,2,$ are two parts of $N_1$:
\begin{align}
N_{1}=N_{11}+N_{12},\label{eq:N112}
\end{align}
where $N_{11}$ is defined as
\begin{align*}
N_{11}:=-\frac{\sum_{k=1}^{3}(\partial_{y_k}v)^{2} \partial^{2}_{y_k}v}{1+|\nabla_y v|^2+(\frac{\partial_{\theta}v}{v})^2}-\sum_{i\not= j}\frac{\partial_{y_i} v \partial_{y_j} v}{1+|\nabla_y v|^2+(\frac{\partial_{\theta}v}{v})^2}\partial_{y_i}\partial_{y_j}v,
\end{align*}
and 
$N_{12}$ is different from $N_{11}$ by that each term has a factor $v^{-l}, \ l=2,3,4,$
\begin{align*}
N_{12}:=&-v^{-4} \frac{(\partial_{\theta}v)^{2} \partial^{2}_{\theta}v}{1+|\nabla_y v|^2+(\frac{\partial_{\theta}v}{v})^2}
+v^{-2}\frac{2\partial_{\theta}v}{1+|\nabla_y v|^2+(\frac{\partial_{\theta}v}{v})^2}\sum_{l=1}^{3}\partial_{y_l}v\partial_{y_l}\partial_{\theta}v\\
&+v^{-3}\frac{(\partial_{\theta}v)^2}{1+|\nabla_y v|^2+(\frac{\partial_{\theta}v}{v})^2}.
\end{align*}

We start with estimating $W_1.$

Observe that all the terms in $N_{11}$ are of the form $\frac{\partial_{y_k}v\partial_{y_l}v}{1+|\partial_{y}v|^2+v^{-2}(\partial_{\theta}v)^2}  
 \ \partial_{y_k}\partial_{y_l}v$, $ k,l=1,\cdots,3.$ We apply the estimates $|\partial_{\theta}^{m}\nabla_{y}^{k}v|,\ v^{-1}|\partial_{\theta}^{n}v|\lesssim \delta$ when $m+|k|\leq 5$ and $|k|\geq 1$, $n=1,2,$ in \eqref{eq:SmCon} and compute directly to obtain
\begin{align}\label{eq:distribu3Thet}
\begin{split}
\chi_{\Omega}|\partial_{\theta}^3 N_{11}|\lesssim & \chi_{\Omega} \sum_{m=0}^{3}\sum_{k,l=1,2,3} \Big|\partial_{\theta}^m \big(\partial_{y_k}v\partial_{y_l}v \partial_{y_k}\partial_{y_l}v \big)
\Big|
\lesssim \delta \chi_{\Omega}\sum_{m+n=0,1,2,3} |\partial_{\theta}^m\nabla_{y}v||\partial_{\theta}^n\nabla_{y}v|.
\end{split}
\end{align}

Now we estimate $\chi_{\Omega}|\nabla_y v||\partial_{\theta}^3\nabla_y v|$, which is the only term containing a fourth order derivative. Compute directly to find
\begin{align}
D:=&(100+| y|^2)^{-\frac{3}{2}} \chi_{\Omega}\| \nabla_y v\ \partial_{\theta}^3\nabla_y v \|_{L_{\theta}^2}
\lesssim  (100+| y|^2)^{-\frac{3}{2}}    \|\nabla_{y}v\|_{L^{\infty}_{\theta}}\|\chi_{\Omega} \partial_{\theta}^3 \nabla_{y}v\|_{L_{\theta}^2}
\end{align}
For the factor $\|\chi_{\Omega} \partial_{\theta}^3 \nabla_{y}v\|_{L_{\theta}^2}$ we decompose $v$ and change the order $\chi_{\Omega}$ and $\nabla_y$, to find,
\begin{align*}
\begin{split}
D\lesssim &\delta (100+| y|^2)^{-\frac{3}{2}}
\|P_{\theta,\geq 2}\partial_{\theta}^3\nabla_y\chi_{\Omega} w\|_{L_{\theta}^2}+\|\langle y\rangle^{-1}1_{\leq (1+\epsilon)\Omega}\nabla_{y}v\|_{\infty} \sum_{m=\pm 1} \|\langle y\rangle^{-2} \nabla_y\chi_{\Omega} w_{m}\|_{\infty}\\
&\hskip 2cm+ \delta (100+| y|^2)^{-\frac{3}{2}} |\nabla_{y}\chi_{\Omega}| \|\partial_{\theta}^3 w\|_{L_{\theta}^2}+\delta\tau^{-2}\\
\lesssim &\delta (100+| y|^2)^{-\frac{3}{2}}
\|P_{\theta,\geq 2}\partial_{\theta}^3\nabla_y\chi_{\Omega} v\|_{L_{\theta}^2} +( \tau^{-2}+\kappa(\epsilon)\Omega^{-4})P(M),
\end{split}
\end{align*}
where, we use that $|\nabla_{y}v|\lesssim \delta$ in \eqref{eq:SmCon}, and we control $\|\langle y\rangle^{-1}1_{\leq (1+\epsilon)\Omega}\nabla_{y}v\|_{\infty}$ as in \eqref{eq:Hea1yNw} after decomposing $v$; and we 
observe that, for $m=\pm 1$, by the definition of $\mathcal{M}_4$,
$$\|\langle y\rangle^{-2} \nabla_y\chi_{\Omega} w_{m}\|_{\infty}\lesssim \|\langle y\rangle^{-2} \nabla_y\chi_{\Omega} \partial_{\theta}w\|_{\infty}\lesssim \kappa(\epsilon)\Omega^{-3}\mathcal{M}_4,$$
and moreover in the last step we use $P_{\theta,\geq 2}w=P_{\theta,\geq 2}v,$ and we argue as in \eqref{eq:fractional} to find
\begin{align}
(100+| y|^2)^{-\frac{3}{2}} |\nabla_{y}\chi_{\Omega}| \|\partial_{\theta}^3 w\|_{L_{\theta}^2}\lesssim \kappa(\epsilon) \Omega^{-\frac{9}{2}}\mathcal{M}_4^{\frac{3}{4}},
\end{align} 

Compare to estimating $D$, it is easier to control the other one in \eqref{eq:distribu3Thet} resulted by the lower number of derivatives. Compute directly to have, for $m+n=0,1,2,3$ and $m,n\leq 2,$
\begin{align}\label{eq:lowOrder}
(100+| y|^2)^{-\frac{3}{2}}\Big\|\chi_{\Omega} \partial_{\theta}^m\nabla_{y}v\ \partial_{\theta}^n\nabla_{y}v \Big\|_{L_{\theta}^2}
\lesssim \big(\tau^{-2}+\kappa(\epsilon)\Omega^{-4}\big)P(M).
\end{align}

Collect the estimates above to obtain
\begin{align}
\begin{split}\label{eq:Theta2N11}
(100+|y|^2)^{-\frac{3}{2}} \chi_{\Omega}\|\partial_{\theta}^3 N_{11}\|_{L^{2}_{\theta}}
\lesssim \delta (100+|y|^2)^{-\frac{3}{2}} \|\partial_{\theta}^3 \nabla_{y}P_{\theta,\geq 2}\chi_{\Omega}w\|_{L^{2}_{\theta}}+\delta \Big(\tau^{-2}+ \kappa(\epsilon)\Omega^{-4}\Big) P(M).
\end{split}
\end{align} 

Returning to the definition of $W_1$ in \eqref{eq:psi32w12}, and applying Young's inequality, we find
\begin{align}\label{eq:estW1}
(100+|y|^2)^{-3}|W_1|\leq& \frac{1}{100}(100+|y|^2)^{-3}\|\partial_{\theta}^3 \nabla_{y}P_{\theta,\geq 2}\chi_{\Omega}w\|_{L^{2}_{\theta}}^2
 +C\delta^2 \Big(\tau^{-2}+\kappa(\epsilon)\Omega^{-4}\Big)^2 P^2(M).
\end{align}

Now we estimate $W_2$, which, after integrating by parts in $\theta$, becomes
\begin{align}
W_2=-\langle v^{-1}P_{\theta,\geq 2}\partial_{\theta}^4\chi_{\Omega} v,\ v\partial_{\theta}^2 \chi_{\Omega} N_{12}\rangle_{\theta}\label{eq:neww2}
\end{align}
Compute directly and use \eqref{eq:SmCon} to find
\begin{align}\label{eq:distribu3Thet2}
\begin{split}
|v\partial_{\theta}^2 N_{12}|\lesssim & \sum_{m=0,1,2}\Big[v^{-1}|\partial_{\theta}^{m}\big(\nabla_{y}v \partial_{\theta}v \nabla_{y}\partial_{\theta}v\big)|+v^{-3}|\partial_{\theta}^{m}\big( (\partial_{\theta}v)^2 \partial_{\theta}^2v\big)|+v^{-2}|\partial_{\theta}^{m}(\partial_{\theta}v)^2|\Big]\\
\lesssim &\delta  \sum_{m=1,2,3}|\partial_{\theta}^{m} v|.
\end{split}
\end{align}
Decompose $v$, and consider it in the space $\|\cdot \|_{L_{\theta}^2}$ and apply Lemma \ref{LM:appliM1234} to find
\begin{align*}
\begin{split}
\chi_{\Omega} \langle y\rangle^{-3}\|v\partial_{\theta}^2 N_{12}\|_{L_{\theta}^2}
\lesssim \delta \big(\tau^{-2}+\kappa(\epsilon)\Omega^{-4} \Big) \big(1+\mathcal{M}_1\big).
\end{split}
\end{align*}

Return to \eqref{eq:neww2} and find, for some $C>0,$
\begin{align}
(100+|y|^2)^{-3}|W_2|\leq \frac{1}{100}(100+|y|^2)^{-3}V_{a,B}^{-2}\|P_{\theta,\geq 2}\partial_{\theta}^4\chi_{\Omega} v\|_{L_{\theta}^2}^2+C\delta^2 (\tau^{-2}+\kappa(\epsilon)\Omega^{-4})^2 P^2(M).
\end{align}

This, together with the estimate for $W_1$ in \eqref{eq:estW1} and the identity in \eqref{eq:psi32w12}, implies the desired \eqref{eq:03wN1}.

Now we prove \eqref{eq:03Lambda}. The key observation is that the difficult term has a favorable sign, specifically, the fact $y\cdot \nabla_{y}\chi_{\Omega}\leq 0$ makes
$$\Big\langle P_{\theta,\geq 2}\partial_{\theta}^3\chi_{\Omega}v, \ (y\cdot \nabla_{y}\chi_{\Omega}) \partial_{\theta}^3v\Big\rangle\leq 0.$$
Compute directly to find
\begin{align}
\begin{split}
\Psi_{33}\leq& (100+|y|^2)^{-3}\langle P_{\theta,\geq 2}\partial_{\theta}^3\chi_{\Omega} v,\ P_{\theta\geq 2}\partial_{\theta}^3\Big((\partial_{\tau}\chi_{\Omega})v-(\Delta_{y}\chi_{\Omega})v-2\nabla_{y}\chi_{\Omega}\cdot  \nabla_{y}v\Big)\rangle_{\theta}\\
\leq & \frac{1}{100}(100+|y|^2)^{-3}\| P_{\theta,\geq 2}\partial_{\theta}^3\chi_{\Omega} v\|_{L_{\theta}^2}^2+\delta^2\kappa^2(\epsilon)\Omega^{-8} ,
\end{split}
\end{align}
where the decay estimates are from the derivatives of $\chi_{\Omega}$ and that they are supported by the set $|y|\in [\Omega,\ (1+\epsilon)\Omega]$, and we use that
$|P_{\theta,\geq 2}w|=|P_{\theta,\geq 2}\eta|\lesssim \max_{\theta}|\eta(\theta)|\lesssim \delta \sqrt{1+\tau^{-1}\Omega^2}$, and $|\nabla_{y}v|\lesssim \delta$, see \eqref{eq:SmCon}.
\end{proof}


\section{Estimate for $\mathcal{M}_2$, Proof of part of \eqref{eq:estM1234} }\label{sec:estM2}
The following results obviously imply the desired estimate for $\mathcal{M}_2$.
\begin{proposition} 
\begin{align}
\begin{split}
\sum_{m=\pm 1,\ 0}\|\langle y\rangle^{-2}\nabla_y\chi_{\Omega}w_{m}(\cdot,\tau)\|_{\infty}, \ 
\big\|   \langle y\rangle^{-2} \|\nabla_y P_{\theta,\geq 2}\partial_{\theta}^2\chi_{\Omega}w(\cdot,\tau)\|_{L^2_{\theta}}     \big\|_{\infty}
\lesssim \delta \kappa(\epsilon)\Omega^{-3} P(M).\label{eq:21w}
\end{split}
\end{align}
\end{proposition}

The proposition will be proved in subsequent subsections.

Compare to the proof of Proposition \ref{prop:M14}, the main difficulty here is that we need to estimate higher order derivatives of $w$. On the other hand the needed decay estimates here are considerably slower, thus make it easier to prove. 

Hence when we do not need new technical tools, we will skip the details.

\subsection{Proof of the estimate for $\|\langle y\rangle^{-2}\nabla_y\chi_{\Omega}w_{0}(\cdot,\tau)\|_{\infty}$ in \eqref{eq:21w}}
Derive an equation for $\nabla_y \chi_{\Omega}w_0$ by taking a $\nabla_y$ on the equation for $\chi_{\Omega}w_0$ in \eqref{eq:weightLinfw},
\begin{align}
\begin{split}\label{eq:yw0}
\partial_{\tau}(\nabla_y\chi_{\Omega}w_0)=-(H_2+\frac{1}{2})&(\nabla_y\chi_{\Omega} w_0)+\nabla_y\chi_{\Omega}\Big(\Sigma+\frac{1}{2\pi}\langle N_{1}(v)+N_2(\eta),\ 1\rangle_{\theta}\Big)+
\Lambda_1(w_0),
\end{split}
\end{align}
where, the linear operator $H_2$ of \eqref{eq:weightLinfw} becomes $H_2+\frac{1}{2}$ here by the commutation relation: for any function $g$ and $l=1,2,3,$
\begin{align}
\partial_{y_l} \frac{1}{2} y\cdot \nabla_{y}g=(\frac{1}{2} y\cdot \nabla_{y}+\frac{1}{2})\partial_{y_l}g,\label{eq:betterDecay}
\end{align}
and the term $\Lambda_1(w_0)$ is defined as
\begin{align}
\Lambda_1(w_0):=&\nabla_y \Lambda(w_0)
+ \frac{1}{2}(\nabla_y\frac{\tilde\chi_{\Omega}\ y\cdot\nabla_{y} \chi_{\Omega}  }{\chi_{\Omega}}) \chi_{\Omega}w_0
+ (\nabla_{y}V_{a,B}^{-2})\chi_{\Omega}w_0\nonumber\\
=&\nabla_{y} \Big( (\partial_{\tau}\chi_{\Omega})w_0-(\Delta_{y}\chi_{\Omega})w_0-2\nabla_{y}\chi_{\Omega}\cdot  \nabla_{y}w_0\Big)+\frac{1}{2}(y\cdot \nabla_y \chi_{\Omega})(1-\tilde\chi_{\Omega})\nabla_{y}w_0\label{eq:Lbda11}\\
&
+\frac{1}{2}\big(\nabla_{y}(y\cdot \nabla_y\chi_{\Omega})\big)w_0
-\frac{1}{2}\frac{(y\cdot \nabla_y \chi_{\Omega}) \tilde\chi_{\Omega}\nabla_y\chi_{\Omega}}{\chi_{\Omega}}w_0+ (\nabla_{y}V_{a,B}^{-2})\chi_{\Omega}w_0,\label{eq:Lbda21}
\end{align} where the expression is simplified after observing $\nabla_{y}(1-\tilde\chi_{\Omega})+\nabla_{y}\tilde\chi_{\Omega}=0$.

The orthogonality conditions imposed on $\chi_{\Omega}w$ in \eqref{eq:orthow} imply that
\begin{align}
e^{-\frac{1}{8}|y|^2}\nabla_{y}\chi_{\Omega}w_0\perp e^{-\frac{1}{8}|y|^2}, \ y_{k}e^{-\frac{1}{8}|y|^2},\ k=1,2,3.
\end{align}

Denote the orthogonal projection onto the subspace orthogonal to these 4 functions by $P_{4}$, which makes
\begin{align}
P_{4}e^{-\frac{1}{8}|y|^2}\nabla_{y}\chi_{\Omega}w_0=e^{-\frac{1}{8}|y|^2}\nabla_{y}\chi_{\Omega}w_0.
\end{align}

Returning to \eqref{eq:yw0}, we apply $e^{-\frac{1}{8}|y|^2}$, $P_{4}$, and then Duhamel's principle to obtain
\begin{align}
\begin{split}\label{eq:yDw0}
e^{-\frac{1}{8}|y|^2}&\nabla_y\chi_{\Omega}w_0=U_3(\tau, \xi_0) e^{-\frac{1}{2}(\tau-\xi_0)} e^{-\frac{1}{8}|y|^2}\nabla_{y}\chi_{\Omega}w_{0}(\xi_0)\\
&+\int_{\xi_0}^{\tau}U_3(\tau,s)e^{-\frac{1}{2}(\tau-s)} P_{4}e^{-\frac{1}{8}|y|^2}\nabla_{y}\Big(\chi_{\Omega}\big(\Sigma+\frac{1}{2\pi}\langle N_1+N_2, \ 1\rangle_{\theta} \big)+\Lambda_1(w_0)\Big)(s) ds,
\end{split}
\end{align} where $U_3(\tau,\sigma)$ is the propagator generated by $-P_{4}e^{-\frac{1}{8}|y|^2}(H_2+\frac{1}{2})e^{\frac{1}{8}|y|^2}P_4$ from $\sigma$ to $\tau.$

We have the following estimate for the propagator. As discussed in the proof of Lemma \ref{LM:propagator}, its proof is very similar to the proved cases, thus we choose to skip the proof.
\begin{lemma}\label{LM:propagator3}
For any function $g$ and $\sigma_1\geq\sigma_2\geq \xi_0,$
\begin{align}
\|\langle y\rangle^{-2} e^{\frac{1}{8}|y|^2} U_3(\sigma_1,\sigma_2) P_{4}g\|_{\infty}\lesssim e^{-\frac{2}{5}(\sigma_1-\sigma_2)} \|\langle y\rangle^{-2}e^{\frac{1}{8}|y|^2} g\|_{\infty}.
\end{align}
\end{lemma}

Next we estimate the terms on the right hand side. Recall that $P(M)$ is defined in \eqref{eq:defPM}.
\begin{proposition}\label{Prop:2weiDy}
\begin{align}
\|\langle y\rangle^{-2}\nabla_{y} \chi_{\Omega}\Sigma\|_{\infty}\lesssim &\tau^{-2},\label{eq:2weiFG}\\
\|\langle y\rangle^{-2}\nabla_{y} \chi_{\Omega}N_1\|_{\infty}\lesssim &\delta \kappa(\epsilon) \Omega^{-3}  P(M),\label{eq:yN1P1}\\
\|\langle y\rangle^{-2}\nabla_{y} \chi_{\Omega}\langle N_2, \ 1\rangle_{\theta}\|_{\infty}\lesssim & \delta\kappa(\epsilon)\Omega^{-3}(1 +
 \mathcal{M}_4), \label{eq:yN2P1}\\
\|\langle y\rangle^{-2}\Lambda_1 (w_0)\|_{\infty}\lesssim & \delta \kappa(\epsilon)\Omega^{-3}(1+\mathcal{M}_1).\label{eq:est2WeiPsi}
\end{align}
\end{proposition}
The proposition will be proved in subsection \ref{subsec:2weiDy}.

Suppose the proposition holds, then we prove the desired result for $\mathcal{M}_2$ in \eqref{eq:21w} as in \eqref{eq:integ}. Here we choose to skip the details.

\subsubsection{Proof of Proposition \ref{Prop:2weiDy}}\label{subsec:2weiDy}
\begin{proof}
It is easy to prove \eqref{eq:2weiFG} by the estimates on the scalar functions in \eqref{eq:Best}-\eqref{eq:scalarEqn}. 

Now we prove \eqref{eq:yN1P1}.
Reason as in \eqref{eq:N112}, \eqref{eq:distribu3Thet} and \eqref{eq:distribu3Thet2}, to find that
\begin{align}
|\nabla_{y}\chi_{\Omega}N_1|\lesssim \delta \chi_{\Omega}\Big[\sum_{|k|=1,2}|\nabla_{y}^{k}v|^2+ \sum_{|k|=0,1}|\nabla_{y}^{k}\partial_{\theta}v|\Big]+\delta^2 |\nabla_{y}\chi_{\Omega}|.\label{yN1T}
\end{align} 

For the first term, apply the same techniques used in \eqref{eq:twoV} and \eqref{eq:y3nablaw2} to find that, 
\begin{align}
\sum_{|k|=1,2}\|\langle y\rangle^{-2}\chi_{\Omega}|\nabla_{y}^{k}v|^2\|_{\infty}\lesssim \tau^{-2}+&\delta \kappa(\epsilon)\Omega^{-3} 
P(M).\label{eq:firT}
\end{align}

For the second term, we discuss separately the cases $|k|=0$ and $|k|=1$. When $|k|=0,$ we decompose $v$, and apply Lemma \ref{LM:appliM1234} to find that
\begin{align}\label{eq:2yThetaEta}
\|\langle y\rangle^{-2}\chi_{\Omega}\partial_{\theta}\eta\|_{\infty}\lesssim  \tau^{-2}+\|\langle y\rangle^{-2}\partial_{\theta}\chi_{\Omega}w\|_{\infty}\lesssim \tau^{-2}+\kappa(\epsilon) \Omega^{-3} \mathcal{M}_4.
\end{align}

For the case $|k|=1$ we decompose $v$, change the order of $\nabla_{y}$ and $\chi_{\Omega}$, to find that 
\begin{align*}
\langle y\rangle^{-2}\chi_{\Omega}|\partial_{\theta}\nabla_{y}v|\lesssim \tau^{-2}+ \langle y\rangle^{-2}|\partial_{\theta}\nabla_{y}\chi_{\Omega}w|+\langle y\rangle^{-2} |\nabla_{y}\chi_{\Omega}| \ |\partial_{\theta}w|.
\end{align*}
We control the second term by $\lesssim \kappa(\epsilon)\Omega^{-3}\mathcal{M}_2$ by Lemma \ref{LM:appliM1234}.
For the third term, we apply the same techniques as in proving \eqref{eq:fractional} to obtain,
\begin{align}
\langle y\rangle^{-2}|\partial_{\theta}w \nabla_{y}\chi_{\Omega}|\leq \Omega^{-3} |\chi_{\Omega}\partial_{\theta}w|^{\frac{3}{4}} \  |\partial_{\theta}w|^{\frac{1}{4}}\ 
\sup_{z}\{|\nabla_{z}\chi(z) |\chi^{-\frac{3}{4}}(z)\}\lesssim \Omega^{-3-\frac{1}{2}}\mathcal{M}_4^{\frac{3}{4}}.\label{eq:fractional2}
\end{align}

Consequently
\begin{align}\label{eq:theYW}
\|\langle y\rangle^{-2}\chi_{\Omega}\nabla_{y}\partial_{\theta}v\|_{\infty}\lesssim \tau^{-2}+\kappa(\epsilon) \Omega^{-3}  (\delta+\mathcal{M}_2+\Omega^{-\frac{1}{2}}\mathcal{M}_4^{\frac{3}{4}}).
\end{align}

For the third term in \eqref{yN1T}
, we use that $|\nabla_{y}\chi_{\Omega}|=\Omega^{-1}\chi'(\frac{|y|}{\Omega})$ and is supported by the set $|y|\in [\Omega,\ (1+\epsilon)\Omega]$ to find
\begin{align}
\langle y\rangle^{-2}|\nabla_{y}\chi_{\Omega}|\lesssim  \kappa(\epsilon) \Omega^{-3}.
\end{align} 

This, together with \eqref{eq:firT}, \eqref{eq:theYW} and \eqref{yN1T}, implies the desired \eqref{eq:yN1P1}.

Now we prove \eqref{eq:yN2P1}. Rewrite the expression $\langle N_2,\ 1\rangle_{\theta}$ as in \eqref{eq
:RewN1Int1}, 
apply the estimates in \eqref{eq:SmCon} and decompose $\eta$ as in \eqref{eq:decomW} to have
\begin{align*}
|\nabla_{y}\chi_{\Omega}\langle N_2,\ 1\rangle_{\theta}|\lesssim &\delta \chi_{\Omega}\Big\|v^{-1}|\eta| (|\nabla_{y}V_{a,B}|+|\nabla_{y}\eta|) + |\partial_{\theta}\eta|\Big\|_{L_{\theta}^2}+\delta |\nabla_{y}\chi_{\Omega}|\\
\lesssim & \delta \chi_{\Omega} \Big\|\tau^{-\frac{1}{2}}|w| + |\partial_{\theta}w|+|\nabla_{y}w| \Big\|_{L_{\theta}^2}+\delta\tau^{-2} (1+|y|)+\delta |\nabla_{y}\chi_{\Omega}|.
\end{align*}
This together with the techniques in proving \eqref{eq:nablanablaw} implies that
\begin{align}
\langle y\rangle^{-2}|\nabla_{y}\chi_{\Omega}\langle N_2,\ 1\rangle_{\theta}|
\lesssim & \delta \kappa(\epsilon)\Omega^{-3} P(M).
\end{align}

Now we prove \eqref{eq:est2WeiPsi}, which contains two parts, \eqref{eq:Lbda11} and \eqref{eq:Lbda21}.
The terms in \eqref{eq:Lbda11} will be treated as those in \eqref{eq:est3PsiW0}, hence we skip the details.

The terms in \eqref{eq:Lbda21} are new. Use that $|\nabla_{y}V_{a,B}|\lesssim \tau^{-\frac{1}{2}}$ and the condition in \eqref{eq:assum} to obtain,
\begin{align}
\|\langle y\rangle^{-2}(\nabla_{y}V_{a,B})\chi_{\Omega}w_0\|_{\infty}\lesssim \tau^{-\frac{1}{2}}\Omega\|\langle y\rangle^{-3}\chi_{\Omega}w_0\|_{\infty}\lesssim  \delta \kappa(\epsilon)\Omega^{-3}\mathcal{M}_1,
\end{align}
and for the other two terms, by the same strategies as those in proving \eqref{eq:fractional},
\begin{align}
\begin{split}\label{eq:chiTrick}
\Big|\langle y\rangle^{-2}\big(\nabla_y( y\cdot\nabla_{y} \chi_{\Omega}  ) \big) w_0\Big|\lesssim &\Omega^{-3}|\chi_{\Omega}w_0|^{\frac{3}{4}}|w_0|^{\frac{1}{4}} \sup_{z}\Big|\frac{\nabla_{z}(z\cdot \nabla_{z}\chi(z))}{\chi^{\frac{3}{4}}(z)}\Big|
\lesssim \Omega^{-3-\frac{1}{2}}\mathcal{M}_1^{\frac{3}{4}},\\
\Big|\langle y\rangle^{-2}\frac{y\cdot \nabla_y \chi_{\Omega}\ \tilde\chi_{\Omega}\ \nabla_y\chi_{\Omega}}{\chi_{\Omega}} w_0\Big|\lesssim &\Omega^{-3}|\chi_{\Omega}w_0|^{\frac{1}{2}}|w_0|^{\frac{1}{2}} \sup_{z}\Big| \frac{z\cdot \nabla_{z}\chi(z) \nabla_{z}\chi}{\chi^{\frac{3}{2}}(z)}\Big|
\lesssim \Omega^{-3-\frac{1}{4}}\mathcal{M}_1^{\frac{1}{2}}.
\end{split}
\end{align}

\end{proof}
\subsection{Proof of the estimate for $\|\langle y\rangle^{-2}\nabla_y\chi_{\Omega}w_{\pm 1}(\cdot,\tau)\|_{\infty}$ in \eqref{eq:21w}}
Since $w_{-}$ and $w_{+}$ are complex conjugate to each other, we only need to estimate one of them, and we choose $w_1$. 

To derive an equation for $\nabla_{y}\chi_{\Omega} w_1$, we take a derivative $\nabla_{y}$ on \eqref{eq:eqnw1} and use commutation relation in \eqref{eq:betterDecay} to find, 
\begin{align}
\begin{split}\label{eq:yw1}
\partial_{\tau}\nabla_{y}\chi_{\Omega} w_1=
-\big(H_1+\frac{1}{2}\big) &\nabla_y\chi_{\Omega}w_{1}
+ \frac{1}{2\pi}\nabla_y\chi_{\Omega}\langle G+N_{1}(v)+N_{2}(\eta), \ e^{i\theta}\rangle_{\theta}+\Lambda_2(w_1),
\end{split}
\end{align}
where $\Lambda_2(w_1)$ is defined similarly to $\Lambda_1(w_0)$ in \eqref{eq:Lbda11} and \eqref{eq:Lbda21},
\begin{align}
\begin{split}
\Lambda_2(w_1):=&\nabla_{y} \Big( (\partial_{\tau}\chi_{\Omega})w_1-(\Delta_{y}\chi_{\Omega})w_1-2\nabla_{y}\chi_{\Omega}\cdot  \nabla_{y}w_1\Big)+\frac{1}{2}(y\cdot \nabla_y \chi_{\Omega})(1-\tilde\chi_{\Omega})\nabla_{y}w_1\\
&
+\frac{1}{2}\big(\nabla_{y}(y\cdot \nabla_y\chi_{\Omega})\big)w_1
-\frac{1}{2}\frac{(y\cdot \nabla_y \chi_{\Omega}) \tilde\chi_{\Omega}\nabla_y\chi_{\Omega}}{\chi_{\Omega}}w_1.
\end{split}
\end{align}

The terms satisfy the following estimates:
\begin{proposition}\label{prop:yw1}
\begin{align}
\|\langle y\rangle^{-2}\Lambda_2(w_1)\|_{\infty}\lesssim &\delta \kappa(\epsilon) \Omega^{-3} (1+\mathcal{M}_4) , \label{eq:yPsiP1}\\
\|\langle y\rangle^{-2}\nabla_y\chi_{\Omega}\langle G, \ e^{i\theta}\rangle_{\theta}\|_{\infty}\lesssim &\tau^{-2},\label{eq:yGP1}\\
\|\langle y\rangle^{-2}\nabla_{y} \chi_{\Omega}N_1\|_{\infty}\lesssim &\delta \kappa(\epsilon) \Omega^{-3}  P(M),\label{eq:yN1P12}\\
\|\langle y\rangle^{-2} \nabla_{y}\chi_{\Omega}\langle N_2, \ e^{i\theta}\rangle_{\theta}\|_{\infty}\lesssim &\delta \tau^{-2} +\delta \kappa(\epsilon)\Omega^{-3} (1+\mathcal{M}_4).\label{eq:N2EiTh}
\end{align}
\end{proposition}
\begin{proof}
The proofs of \eqref{eq:yPsiP1} and \eqref{eq:yN1P12} are almost identical to those of \eqref{eq:est2WeiPsi} and \eqref{eq:yN1P1}, we skip the details here.
The proof of \eqref{eq:yGP1} is easy by the estimates in \eqref{eq:Best}-\eqref{eq:scalarEqn}, we also skip the details.

For \eqref{eq:N2EiTh}, after observing cancellations in $\langle N_2, \ e^{i\theta}\rangle_{\theta}$ in \eqref{eq:N2Theta}, we find
\begin{align}
\begin{split}
\langle y\rangle^{-2}|\nabla_{y}\chi_{\Omega}\langle N_2,\ e^{i\theta}\rangle_{\theta}|\lesssim &\delta  \langle y\rangle^{-2}\Big[|\nabla_{y}\chi_{\Omega}|+\chi_{\Omega}\Big]\|v^{-1}\partial_{\theta} \eta\|_{\theta}\\
\lesssim &\delta \langle y\rangle^{-2}|\nabla_{y}\chi_{\Omega}|+\delta \tau^{-2}+\delta \chi_{\Omega}\|\partial_{\theta} w\|_{L_{\theta}^2}\\
\lesssim & \delta \kappa(\epsilon)\Omega^{-3}+\delta\tau^{-2}+\delta \chi_{\Omega}\|\partial_{\theta} w\|_{L_{\theta}^2}
\end{split}
\end{align}
This, together with the estimates in \eqref{eq:embeddM4} and the techniques in proving \eqref{eq:nablanablaw}, implies \eqref{eq:N2EiTh}.

\end{proof}

Return to the equation \eqref{eq:yw1}. The orthogonality conditions of $\chi_{\Omega}w_1$ imply,
\begin{align}
e^{-\frac{1}{8}|y|^2}\nabla_{y}\chi_{\Omega} w_1\perp e^{-\frac{1}{8}|y|^2}.
\end{align}
Denote by $P_{1}$ the orthogonal projection onto the subspace orthogonal to $e^{-\frac{1}{8}|y|^2}$. This makes
\begin{align*}
P_{1} e^{-\frac{1}{8}|y|^2}\nabla_{y}\chi_{\Omega} w_1=e^{-\frac{1}{8}|y|^2}\nabla_{y}\chi_{\Omega} w_1.
\end{align*}

Then we have that
\begin{align}
\begin{split}
e^{-\frac{1}{8}|y|^2}\nabla_{y}&\chi_{\Omega} w_1(\tau)=U_4(\tau,\xi_0)e^{-\frac{1}{8}|y|^2}\nabla_{y}\chi_{\Omega} w_1(\xi_0)\\
&+\int_{\xi_0}^{\tau} U_4(\tau,\sigma)P_1 \Big(\frac{1}{2\pi}\nabla_y\chi_{\Omega}\big\langle G+N_{1}(v)+N_{2}(\eta), \ e^{i\theta}\big\rangle_{\theta}+\Lambda_2(w_1)\big)(\sigma)\ d\sigma.
\end{split}
\end{align}
Here $U_{4}(\tau,\sigma)$ is the propagator generated by the linear operator $-P_1 e^{-\frac{1}{8}|y|^2} (H_1+\frac{1}{2})e^{\frac{1}{8}|y|^2}P_1$.

We have the following estimate for the propagator. As discussed in the proof of Lemma \ref{LM:propagator}, its proof is very similar to the previously proved ones, thus we choose to skip the proof.
\begin{lemma} For any function $g$, and $\tau\geq \sigma,$
\begin{align}\label{eq:u4}
\|\langle y\rangle^{-2}e^{\frac{1}{8}|y|^2}U_4(\tau,\sigma)P_{1} g\|_{L^{\infty}}\lesssim e^{-\frac{2}{5}(\tau-\sigma)}\|\langle y\rangle^{-2}e^{\frac{1}{8}|y|^2} g\|_{\infty}.
\end{align}
\end{lemma}

What is left is to apply \eqref{eq:u4} and Proposition \ref{prop:yw1} to obtain the desired estimate for $\|\langle y\rangle^{-2}\nabla_y\chi_{\Omega}w_{\pm 1}(\cdot,\tau)\|_{\infty}$ in \eqref{eq:21w}. The procedure is similar to \eqref{eq:integ}, we skip the details here.

\subsection{Proof of the estimate for $\big\|   \langle y\rangle^{-2} \|\nabla_y P_{\theta,\geq 2}\partial_{\theta}^2\chi_{\Omega}w(\cdot,\tau)\|_{L^2_{\theta}}     \big\|_{\infty}$ in \eqref{eq:21w}}

Here we follow the steps in subsection \ref{subsec:3020Wtheta}.

By the equation for $v$ in \eqref{eq:scale1} and the commutation relation in \eqref{eq:betterDecay}, we find that the function $$\Phi_2:=(100+|y|^2)^{-2} \|P_{\theta,\geq 2}\nabla_{y}\partial_{\theta}^2\chi_{\Omega} v\|_{L^{2}_{\theta}}^2$$ satisfy the equation
\begin{align}
\partial_{\tau}\Phi_2=-(L_{2}+V_2) \Phi_2-2(100+|y|^2)^{-2} \sum_{k=1,2,3}\|P_{\theta,\geq 2}\nabla_{y}\partial_{y_k}\partial_{\theta}^2\chi_{\Omega} v\|_{L^{2}_{\theta}}^2+2\sum_{k=1}^3\Psi_{2k},\label{eq:phi1}
\end{align} where the linear operator $L_2+V_2$ is related to $-\Delta+\frac{1}{2}y\cdot \nabla_{y}$ by the identity
\begin{align}
\begin{split}
L_2+V_2:=&(100+|y|^2)^{-2} \Big(-\Delta+\frac{1}{2}y\cdot \nabla_{y}\Big)(100+|y|^2)^{2},
\end{split}
\end{align} and 
the linear operators $L_2$ and $V_2$ are defined as,
\begin{align*}
\begin{split}
L_2:=&-\Delta+\frac{1}{2}y\cdot \nabla_{y}-2 (100+y^2)^{-2} \big(\nabla_{y} (100+|y|^2)^{2}\big)\cdot \nabla_y,\\
V_2:=&\frac{2|y|^2}{100+|y|^2}-\frac{12}{100+|y|^2}-\frac{8|y|^2}{(100+|y|^2)^2},
\end{split}
\end{align*}
and the functions $\Psi_{2k},\ k=1,2,3,$ are defined as
\begin{align*}
\Psi_{21}:=&(100+|y|^2)^{-2} \big\langle P_{\theta,\geq 2}\nabla_{y}\partial_{\theta}^2\chi_{\Omega} v,\ \nabla_{y}\partial_{\theta}^2\chi_{\Omega} \big(v^{-2}\partial_{\theta}^2 v-v^{-1}\big)\big\rangle_{\theta},\\
\Psi_{22}:=& (100+|y|^2)^{-2} \big\langle P_{\theta,\geq 2}\nabla_{y}\partial_{\theta}^2\chi_{\Omega} v,\ \nabla_{y}\partial_{\theta}^2\chi_{\Omega} N_{1}(v)\big\rangle_{\theta},\\
\Psi_{23}:=&(100+|y|^2)^{-2} \big\langle P_{\theta,\geq 2}\nabla_{y}\partial_{\theta}^2\chi_{\Omega} v,\ \nabla_{y}\mu(P_{\theta,\geq 2}\partial_{\theta}^2 v)\big\rangle_{\theta}.
\end{align*} Here $\mu$ is defined in the same fashion as that in \eqref{eq:Tchi3}.

For the terms on the right hand side we have
\begin{proposition}\label{prop:y2T1w}
There exists a constant $C>0$ such that 
\begin{align}
\Psi_{21}
\leq & -(\frac{18}{25}-C\delta)V_{a,B}^{-2} (100+|y|^2)^{-2}\|P_{\theta,\geq 2}\nabla_{y}\partial_{\theta}^3\chi_{\Omega} v\|_{L^{2}_{\theta}}^2+C\delta^2 \kappa^2(\epsilon)\Omega^{-6}P^2(M),\label{eq:1y2Nonlin}\\
\Psi_{23}\leq &\frac{1}{100}\Big[\Phi_{2}+V_{a,B}^{-2} (100+|y|^2)^{-2}\|P_{\theta,\geq 2}\nabla_{y}\partial_{\theta}^3\chi_{\Omega} v\|_{L^{2}_{\theta}}^2\Big]+C\delta^2 \kappa^2(\epsilon)\Omega^{-6}P^2(M),\label{eq:1y2Lambda}\\
\Psi_{22} \leq &\frac{1}{50}V_{a,B}^{-2} (100+|y|^2)^{-2}\|P_{\theta,\geq 2}\nabla_{y}\partial_{\theta}^3\chi_{\Omega} v\|_{L^{2}_{\theta}}^2+C\delta^2 \kappa^2(\epsilon)\Omega^{-6}P^2(M).\label{eq:1y2ThetN1}
\end{align}

\end{proposition}
The proposition will be proved in subsubsection \ref{subsub:prop:y2T1w}.

We continue to study \eqref{eq:phi1}.
The estimates above imply that, for some $C_1>0,$
\begin{align}
\partial_{\tau}\Phi_2\leq -(L_2+1)\Phi_2+C_{1} \delta^2 \kappa^2(\epsilon)\Omega^{-6}P^2(M).
\end{align}
Apply the maximum principle to find, for some $C_2>0,$
\begin{align}
\Phi_2(\tau)\leq e^{-\frac{1}{2}(\tau-\xi_0)}\Phi_2(\xi_0)+C_{2} \delta^2 \kappa^2(\epsilon)\Omega^{-6}P^2(M),
\end{align} then obtain the desired result after using that $\Phi_2(\xi_0)\lesssim \kappa^2(\epsilon)\Omega^{-6}(\xi_0)$ implied by \eqref{eq:22}, and taking square roots on both sides.

\subsubsection{Proof of Proposition \ref{prop:y2T1w}}\label{subsub:prop:y2T1w}
\begin{proof}

To prove \eqref{eq:1y2ThetN1}, we decompose $N_{1}$ as $N_1=N_{11}+N_{12}$ as in \eqref{eq:N112}, and this makes
\begin{align} 
\Psi_{22}:=& (100+|y|^2)^{-2} \big\langle P_{\theta,\geq 2}\nabla_{y}\partial_{\theta}^2\chi_{\Omega} v,\ \nabla_{y}\partial_{\theta}^2\chi_{\Omega} \Big(N_{11}(v)+N_{12}(v)\Big)\big\rangle_{\theta}\nonumber\\
=& D_1+D_2\label{eq:Psi22}
\end{align} with $D_1$ and $D_2$ naturally defined.

To estimate $D_1$, we argue as in \eqref{eq:distribu3Thet} to obtain
\begin{align}
\begin{split}\label{eq:manyT}
|\nabla_{y}\partial_{\theta}^2 \chi_{\Omega}N_{11}|\lesssim &\delta\chi_{\Omega}\Big[ |\nabla_{y}v|^2+|\nabla_{y}v|\sum_{|k|=2}(|\nabla_{y}^{k}v|+|\nabla_{y}^{k}\partial_{\theta}v|)+ \sum_{l=1,2} |\nabla_{y}\partial_{\theta}^l v|\Big]+\delta |\nabla_{y}\chi_{\Omega}|\\
\lesssim &\delta\chi_{\Omega}\Big[ \sum_{|k|=1,2}|\nabla_{y}^{k}v|^2+|\nabla_{y}v|\sum_{|k|=2}|\nabla_{y}^{k}\partial_{\theta}v|+ \sum_{l=1,2} |\nabla_{y}\partial_{\theta}^l v|\Big]+\delta |\nabla_{y}\chi_{\Omega}|.
\end{split}
\end{align}


For the last term, since $|\nabla_{y}\chi_{\Omega}|=\Omega^{-1}|\chi^{'}(\frac{|y|}{\Omega})|$ and it is supported by the set $|y|\in [\Omega,(1+\epsilon)\Omega],$
\begin{align}
(100+|y|^2)^{-1} |\nabla_{y}\chi_{\Omega}|\lesssim \kappa \Omega^{-3}.
\end{align}

Now we analyze the other terms.

For the terms containing only the first and second order derivatives of $v$, we estimate as in \eqref{eq:twoV}-\eqref{eq:y3nablaw2} and \eqref{eq:fractional2} to find
\begin{align}
\chi_{\Omega} \langle y\rangle^{-2}\big[ \sum_{|k|=1,2}|\nabla_{y}^{k}v|^2+  |\nabla_{y}\partial_{\theta} v|\big]\lesssim \tau^{-2} +\kappa(\epsilon)\Omega^{-3}P(M).
\end{align}

For the term $|\nabla_{y}v||\nabla_{y}^k \partial_{\theta}v|$, $|k|=2$, 
\begin{align}
\begin{split}
|\nabla_{y}v||\nabla_{y}^k \partial_{\theta}v|\lesssim &|\nabla_{y}v|\Big[\tau^{-2}+\sum_{m=\pm}|\nabla_{y}^{k}w_{m}|+|\nabla_{y}\partial_{\theta}P_{\theta,\geq 2}w|\Big]\\
\lesssim &\delta \Big[\tau^{-2}+|\nabla_{y}\partial_{\theta}P_{\theta,\geq 2}w|\Big]+\Big[\tau^{-1}(1+|y|)+|\nabla_{y}w|\Big]\sum_{m=\pm}|\nabla_{y}^{k}w_{m}|,
\end{split}
\end{align} where, in the second step, the $\delta-$factor in the first term is from $|\nabla_y v|\lesssim \delta$, for the second term we need more contribution from $\nabla_{y}v$, and thus decompose $v$ to find the present form.

From here apply $|\nabla_{y}^{k}w_{\pm 1}|\lesssim \sup_{\theta}|\nabla_{y}^{k}w(\cdot,\theta,\tau)|\lesssim \delta$ by \eqref{eq:SmCon}, and change the order of $\nabla_{y}$ and $\chi_{\Omega}$, and use estimate \eqref{eq:nablanablaw} and also the techniques proving it, to have 
\begin{align}\label{eq:naVnaKTheV}
\begin{split}
(100+|y|^2)^{-1}\chi_{\Omega}&\|\nabla_{y}v\ \nabla_{y}^k \partial_{\theta}v\|_{L_{\theta}^2}\\
\lesssim &\delta\tau^{-2}+\delta (100+|y|^2)^{-1}\chi_{\Omega}\|\nabla_{y}^k \partial_{\theta}P_{\theta,\geq 2}w\|_{L_{\theta}^2}\\
&\hskip 2cm +\tau^{-1} \sum_{m=\pm 1}\|\langle y\rangle^{-1}\chi_{\Omega}\nabla_{y}^{k}w_{m}\|_{\infty} +\delta \|\langle y\rangle^{-2}\chi_{\Omega}\nabla_{y}w\|_{\infty} \\
\lesssim &
\delta \kappa(\epsilon)\Omega^{-3}P(M).
\end{split}
\end{align} 
For the only remaining term $|\nabla_{y}\partial_{\theta}^2v|$ in \eqref{eq:manyT}, we apply similar techniques to find
\begin{align}
(100+|y|^2)^{-1}\chi_{\Omega}\|\nabla_{y}\partial_{\theta}^2v\|_{L_{\theta}^2}
\lesssim & \tau^{-2}+\kappa(\epsilon)\Omega^{-3}P(M).
\end{align}

These estimates above, together with Young's inequality, make
\begin{align}\label{eq:estD1}
|D_1| \leq \frac{1}{100} (100+|y|^2)^{-2}\|P_{\theta,\geq 2}\nabla_{y}\partial_{\theta}^2\chi_{\Omega} v\|_{L_{\theta}^2}^2+C\delta^2 \kappa^2(\epsilon) \Omega^{-6} P^2(M).
\end{align}

To estimate $D_2$, we integrate by parts in $\theta$ to have
\begin{align*}
D_2=&-(100+|y|^2)^{-2} \big\langle v^{-1}P_{\theta,\geq 2}\nabla_{y}\partial_{\theta}^3\chi_{\Omega} v,\ v\nabla_{y}\partial_{\theta}\chi_{\Omega} N_{12}(v)\big\rangle_{\theta}.
\end{align*}
Reason as in \eqref{eq:distribu3Thet2} to find that
\begin{align*}
v|\nabla_{y}\partial_{\theta} N_{12}|\lesssim \delta v^{-1}\big[ \sum_{l=1,2}|\partial_{\theta}^l v|+ |\nabla_{y}\partial_{\theta}v|\big].
\end{align*}
Change the order of $\nabla_{y}$ and $\chi_{\Omega}$, decompose $v$, and apply \eqref{eq:SmCon} and Lemma \ref{LM:appliM1234}, to find
\begin{align}
\|(100+|y|^2)^{-1}v \nabla_{y}\partial_{\theta}\chi_{\Omega}N_{12}\|_{\infty}\lesssim \delta \kappa(\epsilon)\Omega^{-3}  \big[1+\mathcal{M}_2+\mathcal{M}_4\big].
\end{align}

This makes, for some $C>0,$
\begin{align}
D_2\leq \frac{1}{100} (100+|y|^2)^{-2}\|v^{-1}P_{\theta,\geq 2}\nabla_{y}\partial_{\theta}^3\chi_{\Omega} v\|_{L_{\theta}^2}^2+C\delta^2 \kappa^2(\epsilon)\Omega^{-6}  P^2(M).
\end{align}

This, together with \eqref{eq:estD1} and \eqref{eq:Psi22}, implies the desired estimate \eqref{eq:1y2ThetN1}.

Next we prove \eqref{eq:1y2Nonlin}, by following the steps of subsubsection \ref{subsub:03w}. 

We rewrite the first term by integrating by parts and then decompose,
\begin{align}
\Psi_{21}
=&-(100+|y|^2)^{-2} \sum_{k=1}^3 E_k,
\end{align}
where $E_k,\ k=1,2,3,$ are defined as
\begin{align*}
E_1:=\big\langle P_{\theta,\geq 2}\nabla_{y}\partial_{\theta}^3 \chi_{\Omega} v,\  &v^{-2}P_{\theta,\geq 2}\nabla_{y}\partial_{\theta}^3\chi_{\Omega} v\big\rangle_{\theta}-\big\langle P_{\theta,\geq 2}\nabla_{y}\partial_{\theta}^2\chi_{\Omega} v,\ v^{-2}P_{\theta,\geq 2}\nabla_{y}\partial_{\theta}^2\chi_{\Omega} v\big\rangle_{\theta},\\
E_2:=\big\langle P_{\theta,\geq 2}\nabla_{y}\partial_{\theta}^3 \chi_{\Omega} v,\  &v^{-2}(1-P_{\theta,\geq 2})\nabla_{y}\partial_{\theta}^3\chi_{\Omega} v\big\rangle_{\theta}\\
&-\big\langle P_{\theta,\geq 2}\nabla_{y}\partial_{\theta}^2\chi_{\Omega} v,\ v^{-2}(1-P_{\theta,\geq 2})\nabla_{y}\partial_{\theta}^2\chi_{\Omega} v\big\rangle_{\theta},\\
E_3:=\big\langle P_{\theta,\geq 2}\nabla_{y}\partial_{\theta}^3 \chi_{\Omega} v,\ &\big[\nabla_{y}\partial_{\theta}\chi_{\Omega} v^{-2}\partial_{\theta}^2 v- v^{-2}\nabla_{y}\partial_{\theta}^3\chi_{\Omega} v\big]\big\rangle_{\theta}\\
&+\big\langle P_{\theta,\geq 2}\nabla_{y}\partial_{\theta}^2\chi_{\Omega} v,\ \big[\nabla_{y}\partial_{\theta}^2\chi_{\Omega} v^{-1}+v^{-2}\nabla_{y}\partial_{\theta}^2\chi_{\Omega} v\big]\big\rangle_{\theta}.
\end{align*}

By arguing as in \eqref{eq:negativeTerm}, we have that, for some $C>0,$
\begin{align}
E_1\geq (\frac{3}{4}-C\delta) V_{a,B}^{-2} \|P_{\theta,\geq 2}\nabla_{y}\partial_{\theta}^3 \chi_{\Omega} v\|^2_{L_{\theta}^2}.
\end{align}

To estimate $E_2$, we use $$(1-P_{\theta,\geq 2})\nabla_{y}\partial_{\theta}^2 \chi_{\Omega}v=-\frac{1}{2\pi}\sum_{m=\pm 1} e^{im\theta}\nabla_{y} \chi_{\Omega}\langle v, e^{im\theta}\rangle_{\theta}$$ and then follow the steps in \eqref{eq:Wei3D2}, and apply the estimate \eqref{eq:embeddM2} to obtain
\begin{align}
(100+|y|^2)^{-2}|E_2|
&\leq \frac{1}{100} V_{a,B}^{-2}(100+|y|^2)^{-2}\|P_{\theta,\geq 2}\nabla_{y}\partial_{\theta}^3 \chi_{\Omega} v\|^2_{L_{\theta}^2}+C\delta^2 \kappa^2(\epsilon)\Omega^{-6}P^2(M).
\end{align} 

For $E_3$, compute directly and use the definition of $\Omega$ in \eqref{eq:assum} to find
\begin{align}
\begin{split}
(100+|y|^2)^{-1}|E_3|\lesssim& \delta V_{a,B}^{-1}\|P_{\theta,\geq 2}\nabla_{y}\partial_{\theta}^3 \chi_{\Omega} v\|_{L^2_{\theta}}\ (100+|y|^2)^{-1} \sum_{l=2,3} \|\chi_{\Omega}\partial_{\theta}^l v\|_{L^2_{\theta}}\\
\lesssim &\delta \kappa(\epsilon)\Omega^{-3} (1+\mathcal{M}_4)V_{a,B}^{-1}\|P_{\theta,\geq 2}\nabla_{y}\partial_{\theta}^3 \chi_{\Omega} v\|_{L^2_{\theta}}.
\end{split}
\end{align}
Then apply Young's inequality to find, for some $C>0,$
\begin{align}
(100+|y|^2)^{-2} |E_3|
\leq \frac{1}{100} V_{a,B}^{-2}(100+|y|^2)^{-2}\|P_{\theta,\geq 2}\nabla_{y}\partial_{\theta}^3 \chi_{\Omega} v\|^2_{L^2_{\theta}} +C\delta^2  \kappa^2(\epsilon)\Omega^{-6} P^2(M).
\end{align}

Collect the estimate above to have the desired \eqref{eq:1y2Nonlin}.

Now we prove \eqref{eq:1y2Lambda}. 
It is slightly more involved than proving \eqref{eq:03Lambda} since the operators $\chi_{\Omega}$ and $\nabla_{y}$ do not commute. We decompose $\Psi_{23}$ into two parts,
\begin{align}
\Psi_{23}=(100+|y|^2)^{-2}(U_1+U_2),\label{eq:decP13}
\end{align}
with $U_1$ and $U_2$ defined as
\begin{align*}
U_1:&=\frac{1}{2}\langle P_{\theta,\geq 2}\nabla_{y}\partial_{\theta}^2\chi_{\Omega}v,\ \nabla_{y}\big((y\cdot \nabla_y\chi_{\Omega})P_{\theta,\geq 2}\partial_{\theta}^2 v\big)\rangle_{\theta},\\
U_2:&=\big\langle P_{\theta,\geq 2}\nabla_{y}\partial_{\theta}^2\chi_{\Omega} v,\ \nabla_{y}\Big[  \big(\partial_{\tau}\chi_{\Omega}\big)v-\big(\Delta_{y}\chi_{\Omega}\big)v-2\nabla_{y}\chi_{\Omega}\cdot  \nabla_{y}v  \Big]\big\rangle.
\end{align*}

It is easy to estimate $U_2$ by the decay estimates of the derivatives of $\chi_{\Omega}$ and that they are supported by the set $|y|\in [\Omega,\ (1+\epsilon)\Omega]$. Use the estimates in \eqref{eq:SmCon}, and apply Young's inequality to find, for some $C>0,$
\begin{align}
(100+|y|^2)^{-2}|U_2|\leq \frac{1}{100}(100+|y|^2)^{-2}\|P_{\theta,\geq 2}\nabla_{y}\partial_{\theta}^2\chi_{\Omega}v\|_{L_{\theta}^2}^2+C\delta^2\kappa^2(\epsilon) \Omega^{-6}.\label{eq:estU200}
\end{align}

For $U_1$, it contains a difficult term, but it has a favorable nonpositive sign, specifically $$\langle \chi_{\Omega}P_{\theta,\geq 2}\partial_{\theta}^2\nabla_y v,\ (y\cdot \nabla_y\chi_{\Omega})P_{\theta,\geq 2}\partial_{\theta}^2 \nabla_y v\rangle_{\theta}\leq 0.$$ Hence, for some $C>0$
\begin{align}
U_1\leq &\frac{1}{2}\langle P_{\theta,\geq 2}\nabla_{y}\partial_{\theta}^2\chi_{\Omega}v,\ \big(\nabla_{y}(y\cdot \nabla_y\chi_{\Omega})\big)P_{\theta,\geq 2}\partial_{\theta}^2 v\rangle_{\theta}\nonumber\\
&\hskip 2cm +\frac{1}{2}\langle (\nabla_{y}\chi_{\Omega})\ P_{\theta,\geq 2}\partial_{\theta}^2v,\ (y\cdot \nabla_y\chi_{\Omega})P_{\theta,\geq 2}\partial_{\theta}^2 \nabla_y v\rangle_{\theta}\nonumber\\
\leq &\frac{1}{100} \|P_{\theta,\geq 2}\nabla_{y}\partial_{\theta}^2\chi_{\Omega}v\|_{L_{\theta}^2}^2\nonumber\\
&+
C  \Big[\big| \nabla_{y}(y\cdot \nabla_y\chi_{\Omega})\big|^2 \|\partial_{\theta}^2 v\|_{L_{\theta}^2}^2 + \big| \nabla_y\chi_{\Omega}\big|\  \big|(y\cdot \nabla_y\chi_{\Omega})\big|\|\partial_{\theta}^2 v\|_{L_{\theta}^2} \|\partial_{\theta}^2\nabla_y v\|_{L_{\theta}^2} \Big].\label{eq:extraT}
\end{align} 

Now we follow the steps in the proof of \eqref{eq:fractional}.
For the first term in \eqref{eq:extraT}, decompose $v$ and compute directly to find,
\begin{align}
&(100+|y|^2)^{-2}\big| \nabla_{y}(y\cdot \nabla_y\chi_{\Omega})\big|^2 \|\partial_{\theta}^2 v\|_{L_{\theta}^2}^2\nonumber\\
\lesssim &\Omega^{-3} \|\langle y\rangle^{-2}\chi_{\Omega}\partial_{\theta}^2 w\|_{L_{\theta}^2}^{\frac{3}{2}}\ \|1_{|y|\leq (1+\epsilon)\Omega}\partial_{\theta}^2 w\|_{L_{\theta}^2}^{\frac{1}{2}}\ \sup_{z}\Big|\frac{\nabla_{z} (z\cdot \nabla_{z}\chi)}{\chi^{\frac{3}{4}}}\Big|^2+\kappa^2(\epsilon)\Omega^{-6}\tau^{-4}\nonumber\\
\lesssim &\kappa^{\frac{7}{2}}(\epsilon)\Omega^{-\frac{15}{2}} (1+\tau^{-1}\Omega^2)^{\frac{1}{4}}(\mathcal{M}_4^{\frac{3}{2}}+1)+\kappa^2(\epsilon)\Omega^{-6}\tau^{-4},\nonumber
\end{align}
and similarly for the second term,
\begin{align}
\begin{split}
&(100+|y|^2)^{-2}\big| \nabla_y\chi_{\Omega}\big|\  \big|(y\cdot \nabla_y\chi_{\Omega})\big|\|\partial_{\theta}^2 v\|_{L_{\theta}^2} \|\partial_{\theta}^2\nabla_y v\|_{L_{\theta}^2}\\
\lesssim &\Big(\kappa^{\frac{7}{2}}(\epsilon)\Omega^{-\frac{13}{2}} (1+\tau^{-1}\Omega^2) +\kappa^2(\epsilon)\Omega^{-6}\tau^{-4}\Big)P(M).
\end{split}
\end{align} 

Collect the estimates above and use the condition in \eqref{eq:assum} to find that
\begin{align}
(100+|y|^2)^{-2}U_1\leq \frac{1}{100}  \|P_{\theta,\geq 2}\nabla_{y}\partial_{\theta}^2\chi_{\Omega}v\|_{L_{\theta}^2}^2+C\delta^2 \kappa^2(\epsilon)\Omega^{-6} P^2(M).
\end{align}

This, together with \eqref{eq:estU200} and \eqref{eq:decP13}, implies the desired estimate \eqref{eq:1y2Lambda}.
\end{proof}


\section{Estimate for $\mathcal{M}_3$, Proof of part of \eqref{eq:estM1234}}\label{sec:M3}
We reformulate the estimate into the following results, 
\begin{proposition}
For any $|k|=2$ we have
\begin{align}
\|\langle y\rangle^{-1}\nabla_{y}^{k}\chi_{\Omega}w_0\|_{\infty},\ \big\| \langle y\rangle^{-1}\|\nabla_{y}^{k}\partial_{\theta}\chi_{\Omega}w\|_{L_{\theta}^2}\big\|_{\infty}\lesssim \delta \kappa(\epsilon)\Omega^{-2}P(M).\label{eq:12w0}
\end{align}

\end{proposition}

The proposition will be proved in subsequent subsections.

\subsection{Proof of the first estimate in \eqref{eq:12w0}}
Similar to deriving the equation for $\nabla_y \chi_{\Omega}w_0$ in \eqref{eq:yw0}, 
\begin{align}
\begin{split}\label{eq:yychiw}
\partial_{\tau}\nabla_{y}^{k} \chi_{\Omega}w_0=&-(H_2+1)\nabla_{y}^{k} \chi_{\Omega}w_0+\nabla_{y}^{k} \chi_{\Omega}\Big(\Sigma+\frac{1}{2\pi}\langle N_1+N_2, \ 1\rangle_{\theta}\Big)+ \Lambda_3(w_0)
\end{split}
\end{align} where the term $\Lambda_3(w_0)$ is defined as: 
\begin{align}
\begin{split}
\Lambda_3(w_0):= &\nabla_{y}^{k} \Big( (\partial_{\tau}\chi_{\Omega})w_0-(\Delta_{y}\chi_{\Omega})w_0-2\nabla_{y}\chi_{\Omega}\cdot  \nabla_{y}w_0\Big)+\frac{1}{2}\nabla_{y}^{k_2}\big((y\cdot \nabla_y \chi_{\Omega})(1-\tilde\chi_{\Omega})\nabla_{y}^{k_1}w_0\big)\\
&
+\frac{1}{2}\nabla_{y}^{k_2}\Big[\big(\nabla_{y}^{k_1}(y\cdot \nabla_y\chi_{\Omega})\big)w_0
-\frac{1}{2}\frac{(y\cdot \nabla_y \chi_{\Omega}) \tilde\chi_{\Omega}\nabla_y^{k_1}\chi_{\Omega}}{\chi_{\Omega}}w_0+ (\nabla_{y}^{k_1}V_{a,B}^{-2})\chi_{\Omega}w_0\Big]\\
&+\frac{1}{2}\big(\nabla_{y}^{k_2} \frac{\tilde\chi_{\Omega} \ y\cdot \nabla_y\chi_{\Omega}}{\chi_{\Omega}}\big)\nabla_{y}^{k_1}(\chi_{\Omega}w_0)+(\nabla_{y}^{k_2} V_{a,B})\nabla_{y}^{k_1}(\chi_{\Omega}w_0),
\end{split}
\end{align} where, to ease the notations, all the terms except the first one are the sum of all the possibilities of $k=k_1+k_2$ with $|k_1|=|k_2|=1,$ and we use the commutation relation in \eqref{eq:betterDecay} again to find that the linear operator becomes $H_2+1$ here.

The orthogonality conditions imposed on $e^{-\frac{1}{8}|y|^2} \chi_{\Omega}w$ imply that, recall that $|k|=2$,
\begin{align}
e^{-\frac{1}{8}|y|^2}\nabla_{y}^{k} \chi_{\Omega}w_0\perp e^{-\frac{1}{8}|y|^2}.
\end{align}
We denote by $P_1$ the orthogonal projection onto the subspace orthogonal to $e^{-\frac{1}{8}|y|^2}$.

Apply $e^{-\frac{1}{8}|y|^2}$ and then $P_1$ on \eqref{eq:yychiw}, and then apply Duhamel's principle to find
\begin{align}
\begin{split}\label{eq:yyDur}
e^{-\frac{1}{8}|y|^2}\nabla_{y}^{k} &\chi_{\Omega}w_0(\cdot,\tau)=U_5(\tau,\tau_0)e^{-\frac{1}{8}|y|^2}\nabla_{y}^{k} \chi_{\Omega}w_0(\cdot,\tau_0)\\
&+\int_{\xi_0}^{\tau} U_5(\tau,\sigma) P_1e^{-\frac{1}{8}|y|^2}\Big[\nabla_{y}^{k} \chi_{\Omega}\big(\Sigma+\frac{1}{2\pi}\langle N_1+N_2, \ 1\rangle_{\theta}\big)+ \Lambda_3(w_0)\Big](\sigma)\ d\sigma,
\end{split}
\end{align}
where $U_5(\tau,\sigma)$ is the propagator generated by $-P_1e^{-\frac{1}{8}|y|^2}(H_2+1)e^{\frac{1}{8}|y|^2}P_1$ from $\sigma$ to $\tau.$

The propagator satisfies the following estimate, recall that $|k|=2,$
\begin{lemma}\label{LM:U5}
For any function $g$, and $\tau\geq \sigma\geq \xi_0,$
\begin{align}\label{eq:estU4}
\|\langle y\rangle^{-1}e^{\frac{1}{8}|y|^2}U_5(\tau,\sigma)P_1g\|_{\infty}\lesssim e^{-\frac{2}{5}(\tau-\sigma)} \|\langle y\rangle^{-1}e^{\frac{1}{8}|y|^2}g\|_{\infty}.
\end{align}
\end{lemma}
As discussed in the proof of Lemma \ref{LM:propagator}, its proof is very similar to the ones considered in the known results. Hence we skip the details.

The terms in \eqref{eq:yyDur} satisfy the following estimates:
\begin{proposition}\label{prop:2yw0}
\begin{align}
\|\langle y\rangle^{-1}\nabla_{y}^{k} \chi_{\Omega}\Sigma\|_{\infty}\lesssim &\tau^{-2},\label{eq:2yScal}\\
\|\langle y\rangle^{-1}\nabla_{y}^{k} \chi_{\Omega}\langle N_1,\ 1\rangle_{\theta}\|_{\infty}\lesssim &\delta \kappa(\epsilon)\Omega^{-2}P(M),\label{eq:2yN11}\\
\|\langle y\rangle^{-1}\nabla_{y}^{k} \chi_{\Omega}\langle N_2, \ 1\rangle_{\theta}\|_{\infty}\lesssim &\delta \kappa(\epsilon)\Omega^{-2}P(M),\label{eq:2yN21}\\
\|\langle y\rangle^{-1} \Lambda_3(w_0)\|_{\infty}\lesssim &\delta\kappa(\epsilon)\Omega^{-2}\big(1+\mathcal{M}_4+\mathcal{M}_2\big).\label{eq:weigh1Phiw0}
\end{align}
\end{proposition} 
The proposition will be proved in subsubsection \ref{subsub:2yw0}.

By the estimates above,  and going through a procedure similar to that in \eqref{eq:integ}, we obtain the desired result \eqref{eq:12w0}. We skip the details here.

\subsubsection{Proof of Proposition \ref{prop:2yw0}}\label{subsub:2yw0}
The proof of \eqref{eq:2yScal} is easy, by the estimates in \eqref{eq:Best}-\eqref{eq:scalarEqn}.

Now we prove \eqref{eq:2yN11}.
By arguing as in \eqref{eq:distribu3Thet} and \eqref{eq:distribu3Thet2} and using \eqref{eq:SmCon}, and then decompose $v$, we have, for any $|k|=2,$
\begin{align}\label{eq:2yN1}
\begin{split}
|\nabla_{y}^{k}\chi_{\Omega}N_{1}|\lesssim &\delta\chi_{\Omega}\Big[ \sum_{|l|=1,2}|\nabla_{y}^l v|^2+ \sum_{|l|=0,1}|\nabla_{y}^l \partial_{\theta}v|\Big]+\delta\Big[ |\nabla_{y}^{k}\chi_{\Omega}|+|\nabla_{y}\chi_{\Omega}|\Big]\\
\lesssim &\delta\chi_{\Omega} \Big[\sum_{|l|=1,2}|\nabla_{y}^{l}V_{a,B}|^2+\tau^{-2}(1+|y|)+ \sum_{|l|=1,2}|\nabla_{y}^l w|+\sum_{|l|=0,1}|\nabla_{y}^{l} \partial_{\theta}w|\Big]\\
&+\delta\Big[ |\nabla_{y}^{k}\chi_{\Omega}|+|\nabla_{y}\chi_{\Omega}|\Big].
\end{split}
\end{align} 

Apply the same techniques as in proving \eqref{eq:nablanablaw} to obtain the desired estimate,
\begin{align}
\|\langle y\rangle^{-1}\nabla_{y}^{k}\chi_{\Omega}N_{1}\|_{\infty}\lesssim \delta \tau^{-\frac{3}{2}}+\delta \kappa(\epsilon)\Omega^{-2}P(M).
\end{align}

Now we prove \eqref{eq:2yN21}. 
Cancellations observed in \eqref{eq
:RewN1Int1} make
\begin{align*}
\nabla_{y}^{k}\chi_{\Omega}\langle N_2(\eta),\ 1\rangle_{\theta}=&-\nabla_{y}^{k}\chi_{\Omega}\langle V_{a,B}^{-2}v^{-1}\eta^2,\ 1\rangle_{\theta}+2\nabla_{y}^{k}\chi_{\Omega}\langle v^{-3}(\partial_{\theta}\eta)^2,\ 1\rangle_{\theta}\nonumber\\
=&W_1+W_2
\end{align*} with the terms $W_1$ and $W_2$ naturally defined.

For $W_2$, we decompose $\eta$ as in \eqref{eq:decomW} and use \eqref{eq:SmCon} to find that, for $|k|=2,$
\begin{align}
\begin{split}\label{eq:estW111}
\|\langle y\rangle^{-1}W_2\|_{\infty}
\lesssim &\delta\Omega \Big[\|\langle y\rangle^{-2}\partial_{\theta}\nabla_y\chi_{\Omega}w\|_{\infty}+\|\langle y\rangle^{-2}\partial_{\theta}\chi_{\Omega}w\|_{\infty}\Big]+\delta \kappa(\epsilon) \Omega^{-2}\\
\lesssim & \delta \kappa(\epsilon)\Omega^{-2}(1+\mathcal{M}_2+\mathcal{M}_4).
\end{split}
\end{align}

For $W_1$ we use that $\nabla_{y}v=\nabla_{y}V_{a,B}+\nabla_{y}\eta$ and that $\frac{|\eta|}{V_{a,B}},\ \frac{|\eta|}{v}\lesssim \delta$ in \eqref{eq:SmCon},
\begin{align}
&\Big|\nabla_{y}^{k} \chi_{\Omega}V_{a,B}^{-2}v^{-1}\eta^2\Big|\nonumber\\
\lesssim  & V_{a,B}^{-2} \chi_{\Omega}|\eta|\Big[|\nabla_{y}^{k} V_{a,B}|+|\nabla_{y}^{k} \eta|\Big]+
V_{a,B}^{-2}\chi_{\Omega}\Big[|\nabla_y \eta|^2+|\nabla_y V_{a,B}|^2 \Big]+\delta \sum_{|l|=1,2}|\nabla_{y}^{l}\chi_{\Omega}|\nonumber\\
\lesssim &\chi_{\Omega}\Big[\tau^{-1}|w| +\delta |\nabla_{y}^{k} w|+\delta |\nabla_{y}w|\Big]+\tau^{-2}(1+|y|^2)\chi_{\Omega}+\delta \sum_{|l|=1,2}|\nabla_{y}^{l}\chi_{\Omega}|,\label{eq:fourTerms}
\end{align} where we used that $|\nabla_{y}^{k} V_{a,B}|\lesssim \tau^{-1}
$ for $|k|=2,$ and $
|\nabla_y V_{a,B}|\lesssim \tau^{-1}|y|.$

For the first term on the right hand side, 
\begin{align}
\tau^{-1}\|\langle y\rangle^{-1} \chi_{\Omega}w\|_{\infty}\lesssim \tau^{-1}\Omega^{2} \|\langle y\rangle^{-3} \chi_{\Omega}w\|_{\infty}\lesssim  \delta \kappa(\epsilon) \Omega^{-2}\mathcal{M}_1,
\end{align}
where we use that $\tau^{-3}\Omega^2\leq \Omega^{-\frac{5}{2}}$ by the definition of $\Omega$ in \eqref{eq:defOmega}. For the second and third terms we use the techniques in proving \eqref{eq:nablanablaw}, and control the last two terms by direct computation. We find
\begin{align}
\|\langle y\rangle^{-1}W_1\|_{\infty}\lesssim \delta\kappa(\epsilon) \Omega^{-2}P(M).
\end{align}

This together with \eqref{eq:estW111} implies the desired \eqref{eq:2yN21}.

For \eqref{eq:weigh1Phiw0}, the proof is considerably easier than that of \eqref{eq:est2WeiPsi}, because the needed decay estimate is significantly slower, and because that $|\nabla^{l}_{y}\chi_{\Omega}|$ decays faster as $|l|$ increases.

Hence here we choose to skip the proof.

\subsection{Proof of the second estimate in \eqref{eq:12w0} }

As in deriving \eqref{eq:phi1}, we find $\Phi_1$, defined as, for $|k|=2,$
\begin{align}
\Phi_1:=(100+|y|^2)^{-1} \|\nabla_{y}^{k} \partial_{\theta}\chi_{\Omega} v\|_{L_{\theta}^2}^2,
\end{align} satisfies the equation
\begin{align}\label{eq:phi3}
\partial_{\tau}\Phi_1=-(L_{1}+V_1)\Phi_1-2(100+|y|^2)^{-1} \|\partial_{\theta}\nabla_{y}^{k} \nabla_{y}\chi_{\Omega} v\|_{L_{\theta}^2}^2+2\sum_{k=1}^{3}\Psi_{1k},
\end{align} where the linear operator $L_{1}+V_1$ is defined as
\begin{align*}
L_{1}+V_1:=&(100+|y|^2)^{-1} \Big(-\Delta+\frac{1}{2}y\cdot\nabla_{y}+1\Big)(100+|y|^2),
\end{align*} and $L_1$ is a differential operator, $V_1$ is a multiplier defined as,
\begin{align}
\begin{split}
L_{1}:=&-\Delta+\frac{1}{2}y\cdot\nabla_{y}-\frac{2y}{100+|y|^2}\cdot \nabla_{y},\\
V_1:=&1+\frac{|y|^2}{100+|y|^2}-\frac{6}{100+|y|^2},
\end{split}
\end{align} And the terms $\Psi_{1m},\ m=1,2,3,$ are defined as
\begin{align*}
\Psi_{11}:=&(100+|y|^2)^{-1}\Big\langle \nabla_{y}^{k}\partial_{\theta}\chi_{\Omega}v,\ \nabla_{y}^{k}\partial_{\theta}\big(\chi_{\Omega}(v^{-2}\partial_{\theta}^2 v-v^{-1})\Big\rangle_{\theta},\\
\Psi_{12}:=&(100+|y|^2)^{-1}\Big\langle \nabla_{y}^{k}\partial_{\theta}\chi_{\Omega}v,\ \nabla_{y}^{k}\partial_{\theta}\chi_{\Omega} N_{1}(v)\Big\rangle_{\theta},\\
\Psi_{13}:=&(100+|y|^2)^{-1}\Big\langle \nabla_{y}^{k}\partial_{\theta}\chi_{\Omega}v,\ \nabla_{y}^{k}\partial_{\theta}\mu(v)\Big\rangle_{\theta},
\end{align*} where $\mu$ is defined as in \eqref{eq:Tchi3}.
They satisfy the following estimates:
\begin{proposition}\label{prop:2yThew}
There exists some constant $C$ such that
\begin{align}
\Psi_{11}(v)\leq & -(\frac{18}{25}-C\delta)(100+|y|^2)^{-1} V_{a,B}^{-2} \|P_{\theta,\geq 2}\nabla_{y}^{k} \partial_{\theta}^2 \chi_{\Omega} v\|_{L_{\theta}^2}^2+C\delta \Phi_1+C\delta^2 \kappa^2(\epsilon) \Omega^{-4} P^2(M),\label{eq:estPsi2}\\
|\Psi_{12}|\leq&\frac{1}{100}\Big[(100+|y|^2)^{-1}V_{a,B}^{-2} \|\nabla_{y}^{k} \partial_{\theta}^2 \chi_{\Omega} v\|_{L_{\theta}^2}^2+\Phi_1\Big]+C\delta^2 \kappa^2(\epsilon) \Omega^{-4} P^2(M),\label{eq:estPsi22}\\
\Psi_{13}\leq &\frac{1}{100}(100+|y|^2)^{-1} \|\nabla_{y}^{k} \partial_{\theta} \chi_{\Omega} v\|_{L_{\theta}^2}^2+C\delta^2 \kappa^2(\epsilon) \Omega^{-4} P^2(M).\label{eq:estP23}
\end{align}

\end{proposition}
The proposition will be proved in subsubsection \ref{subsub:2yThew}.

We continue to study \eqref{eq:phi3}. A key observation is that
$
V_1\geq \frac{9}{10}.$
This together with the results in Proposition \ref{prop:2yThew} implies, for some $C_1>0,$
\begin{align}
\partial_{\tau}\Phi_1\leq -(L_{1}+\frac{1}{2})\Phi_1+C_1\delta^2 \kappa^2(\epsilon) \Omega^{-4} P^2(M).
\end{align}

Apply the maximum principle, using that $\Phi_1(y)=0$ if $|y|\geq (1+\epsilon)\Omega(\tau),$ to find that, for some $C_2,\ C_3>0,$
\begin{align}
\Phi_1\leq &e^{-\frac{1}{2}(\tau-\xi_0)}\Phi_{1}(\xi_0)+C_2\delta^2 \kappa^2(\epsilon) \Omega^{-4} P^2(M)
\leq  C_3\delta^2 \kappa^2(\epsilon) \Omega^{-4} P^2(M),\label{eq:FPhi3}
\end{align} where in the second step we use $\Phi_1(\xi_0
)\lesssim \kappa^2(\epsilon)\Omega^{-4}(\xi_0)$ implied by \eqref{eq:11}.

By definition the estimate of $\Phi_1$ does not directly implies that for $\nabla_{y}^{k}\partial_{\theta}\chi_{\Omega}w.$ The decomposition of $v$ implies that
\begin{align}\label{eq:idenSimp}
\nabla_{y}^{k}\partial_{\theta}\chi_{\Omega}w=\nabla_{y}^{k}\partial_{\theta}\chi_{\Omega}\bigg(v-\big({\vec\beta}_2(\tau)\cdot y cos\theta + {\vec\beta}_3(\tau)\cdot y sin\theta +\alpha_1(\tau) cos\theta +\alpha_2(\tau) sin\theta\big)\bigg).
\end{align} This, together with the estimates in \eqref{eq:betaA} implies that, for some $C_4>0,$ 
\begin{align}
\Phi_1\geq \frac{1}{2}(100+|y|^2)^{-1} \|P_{\theta,\geq 2}\nabla_{y}^{k} \partial_{\theta}\chi_{\Omega} w\|_{L_{\theta}^2}^2-C_4\tau^{-4}.\label{eq:justi}
\end{align}
This, together with \eqref{eq:FPhi3}, directly implies the desired second estimate in \eqref{eq:12w0} .

\subsubsection{Proof of Proposition \ref{prop:2yThew}}\label{subsub:2yThew}
To prove \eqref{eq:estPsi22} we decompose $N_1$ into $N_{11}+N_{12}$ as in \eqref{eq:N112}. This makes 
\begin{align}
\Psi_{12}=D_1+D_2,
\end{align} where $D_1$ and $D_2$ are defined in terms of $N_{11}$ and $N_{12}$ respectively.

By arguing as in \eqref{eq:distribu3Thet}, we have that, for any $|k|=2,$
\begin{align}
|\nabla_{y}^{k}\partial_{\theta}\chi_{\Omega}N_{11}|\leq \delta\sum_{|l|=1,2}\chi_{\Omega}\Big[ |\nabla_{y}^{l}v|^2+  |\nabla_{y}^{l}\partial_{\theta}v|\Big]+\delta\Big[ |\nabla_{y}^{k}\chi_{\Omega}|+|\nabla_{y}\chi_{\Omega}|\Big].
\end{align}

Then we decompose $v$ and apply the same techniques as in proving \eqref{eq:y3nablaw2}-\eqref{eq:nablanablaw}, to find
\begin{align}
|D_1|\lesssim \delta \kappa(\epsilon)\Omega^{-2}P(M) \sqrt{\Phi_1}.
\end{align}

For $D_2$, we integrate by parts in $\theta$ to find that
\begin{align*}
D_2=-(100+|y|^2)^{-1}\langle v^{-1}\nabla_{y}^{k}\partial_{\theta}^2 \chi_{\Omega}v,\ v\nabla_{y}^{k}\chi_{\Omega} N_{12}\rangle_{\theta},
\end{align*} 
and then argue as in \eqref{eq:distribu3Thet2} to find
\begin{align}
v|\nabla_{y}^{k}\chi_{\Omega} N_{12}|\lesssim \delta \chi_{\Omega}\Big[|\nabla_{y}\partial_{\theta}v|+\sum_{l=1,2}|\partial_{\theta}^l v|\Big]+\delta \sum_{|l|=1,2}\Big|\nabla_{y}^{l}\chi_{\Omega}\Big|.
\end{align} What is left is to decompose $v$ and apply the techniques as in proving \eqref{eq:nablanablaw} to have
\begin{align}
|D_2|\lesssim \delta\kappa(\epsilon)\Omega^{-2} P(M)\ (100+|y|^2)^{-\frac{1}{2}}V_{a,B}^{-2} \|\nabla_{y}^{k} \partial_{\theta}^2 \chi_{\Omega} v\|_{L_{\theta}^2}.
\end{align}

Collect the estimates above and apply Young's inequality to obtain the desired \eqref{eq:estPsi22}.

Now we prove \eqref{eq:estPsi2}. 
The strategy is the same to those in proving \eqref{eq:03wNega} and \eqref{eq:1y2Nonlin}. A minor difference is that $P_{\theta,\geq 2}$ is not in the definition of $\Phi_1$, but it appears in the first term on the right hand side of \eqref{eq:estPsi2}. This is resulted by a cancellation, see \eqref{eq:canll} below.

We compute directly to find
\begin{align}
\Psi_{11}=\langle \partial_{\theta}^2\nabla_{y}^{k}\chi_{\Omega} v,\  v^{-2}\partial_{\theta}^2\nabla_{y}^{k} \chi_{\Omega}v\rangle_{\theta}-\langle \partial_{\theta}\nabla_{y}^{k} \chi_{\Omega} v,\ v^{-2}\partial_{\theta}\nabla_{y}^{k}\chi_{\Omega} v\rangle_{\theta}.
\end{align}
Decompose $\partial_{\theta}\chi_{\Omega}v$ as $$\partial_{\theta}\chi_{\Omega}v=\partial_{\theta}\Big[\frac{1}{2\pi}\sum_{m=\pm 1}e^{im\theta}\chi_{\Omega}v_{m}+P_{\theta,\geq 2}\chi_{\Omega}v\Big],$$ with $v_{m}:=\frac{1}{2\pi}\langle v,\ e^{im\theta}\rangle_{\theta}$, and observe a cancellation,
\begin{align}\label{eq:canll}
\langle \partial_{\theta}^2\nabla_{y}^{k}\chi_{\Omega} e^{im\theta}v_{m},\  v^{-2}\partial_{\theta}^2\nabla_{y}^{k} \chi_{\Omega}e^{im\theta}v_{m}\rangle_{\theta}-\langle \partial_{\theta}\nabla_{y}^{k} \chi_{\Omega} e^{im\theta}v_{m},\ v^{-2}\partial_{\theta}\nabla_{y}^{k}\chi_{\Omega} e^{im\theta}v_{m}\rangle_{\theta}=0,
\end{align}
and hence find
\begin{align}\label{eq:defW}
\Psi_{11}=-\Big[W_1+W_2\Big],
\end{align} where the term $W_1$ contains some positive terms and is defined as
$$W_1:=\langle P_{\theta,
\geq 2}\partial_{\theta}^2\nabla_{y}^{k}\chi_{\Omega} v,  v^{-2}P_{\theta,
\geq 2}\partial_{\theta}^2\nabla_{y}^{k} \chi_{\Omega}v\rangle_{\theta}-\langle P_{\theta,
\geq 2}\partial_{\theta}\nabla_{y}^{k} \chi_{\Omega} v, v^{-2}P_{\theta,
\geq 2}\partial_{\theta}\nabla_{y}^{k}\chi_{\Omega} v\rangle_{\theta}$$
and $W_2$ collecting all the others and satisfying the estimate, 
\begin{align*}
|W_2|\lesssim &|\langle P_{\theta,\geq 2}\partial_{\theta}^2\nabla_{y}^{k}\chi_{\Omega} v,\  v^{-2} e^{i\theta}\nabla_{y}^{k} \chi_{\Omega}v_{1}\rangle_{\theta}|+|\langle P_{\theta,\geq 2}\partial_{\theta}\nabla_{y}^{k} \chi_{\Omega} v,\ v^{-2}e^{i\theta}\nabla_{y}^{k}\chi_{\Omega} v_{1}\rangle_{\theta}| \\
&+|\langle e^{2i\theta}\nabla_{y}^{k}\chi_{\Omega} v_{1},\ v^{-2}\nabla_{y}^{k}\chi_{\Omega} v_{-1}\rangle_{\theta}|,
\end{align*} where we use that $v_{-1}=\overline{v_1}$ resulted by that $v$ is a real function.

Similar to proving \eqref{eq:03wNega} and \eqref{eq:1y2Nonlin}, $W_1$ is bounded from below, for some $C>0,$
\begin{align}
W_1\geq (\frac{3}{4}-C\delta) V_{a,B}^{-2}\langle P_{\theta,
\geq 2}\partial_{\theta}^2\nabla_{y}^{k}\chi_{\Omega} v,\  P_{\theta,
\geq 2}\partial_{\theta}^2\nabla_{y}^{k} \chi_{\Omega}v\rangle_{\theta}.
\end{align} 

For $W_2$, we control the first two terms as in the proofs of \eqref{eq:deltaK} and \eqref{eq:estKdelta}. For the last term, we also use that 
\begin{align*}
|\nabla_{y}^{k}\chi_{\Omega} v_{\pm 1}|\leq \|\partial_{\theta}^2\nabla_{y}^{k}\chi_{\Omega} v\|_{L_{\theta}^2}.
\end{align*}
Consequently, we have,
\begin{align}
\begin{split}
(100+|y|^2)^{-1}|W_2|\lesssim &\delta \kappa(\epsilon)\Omega^{-2} (1+\mathcal{M}_3)(100+|y|^2)^{-\frac{1}{2}}\|v^{-1}P_{\theta,\geq 2}\partial_{\theta}^2\nabla_{y}^{k}\chi_{\Omega} v\|_{L_{\theta}^2} +  \delta \Phi_1.
\end{split}
\end{align}

Collect the estimates above and apply Young's inequality to prove \eqref{eq:estPsi2}.

The proof of \eqref{eq:estP23} is easier than those of \eqref{eq:03Lambda} and \eqref{eq:1y2Lambda}, because the needed decay estimate is slower, and $|\nabla_{y}^{l}\chi_{\Omega}|$ decays faster as $|l|$ increases. Hence we skip the details.

\appendix
\section{Proof of Proposition \ref{prop:332211}}\label{sec:improved}

In the proof we need some results proved in \cite{GZ2017}.
Recall that we proved that, if $ \tau_0$ is sufficiently large, then for $\tau\geq\tau_0$, we can decompose $v$ as
\begin{align}
\begin{split}
v(y,\ \theta,\tau)=V_{a(\tau), B(\tau)}(y)+\vec\beta_1(\tau)\cdot y&+{\vec\beta}_2(\tau)\cdot y cos\theta + {\vec\beta}_3(\tau)\cdot y sin\theta\\
& +\alpha_1(\tau) cos\theta +\alpha_2(\tau) sin\theta+w(y,\ \theta,\tau).\label{eq:decomVToW100}
\end{split}
\end{align} 

We take the following results from \cite{GZ2017}: recall the definition of $R$ from \eqref{eq:defRTau}.
\begin{proposition}
There exists a small constant $\delta$ and a large time $\tau_1$ such that when $|y|\leq (1+\epsilon)R(\tau)$, and $\tau\geq \tau_1$, 
and for $ |k|+l=1,\cdots,4$,
\begin{align}
\Big|v(\cdot,\tau)-\sqrt{2}\Big|,\  |\nabla_{y}^{k}\partial_{\theta}^{l} v(\cdot,\tau)|\leq \delta.\label{eq:smallRw}
\end{align}

The scalar function $a$, $\vec\beta_k$, $\alpha_l$, $k=1,2,3,$ $l=1,2$ and the matrix $B$ satisfy the estimates in \eqref{eq:Best}-\eqref{eq:betaA}, and for the function $w$ there exists a constant $C$ such that, 
\begin{align}
\begin{split}\label{eq:prev2}
\|\langle y\rangle^{-3}\chi_{R}w(\cdot,\tau)\|_{\infty}\leq &C \kappa(\epsilon)R^{-4}(\tau),\\
\|\langle y\rangle^{-2}\nabla_{y}^{m}\partial_{\theta}^{n}\chi_{R}w(\cdot,\tau)\|_{\infty}\leq &C \kappa(\epsilon)R^{-3}(\tau),\ |m|+n=1,\\
\|\langle y\rangle^{-1}\nabla_{y}^{m}\partial_{\theta}^{n}\chi_{R}w(\cdot,\tau)\|_{\infty}\leq &C \kappa(\epsilon)R^{-2}(\tau),\ |m|+n=2.
\end{split}
\end{align}
\end{proposition}

Before proving Proposition \ref{prop:332211}, we recall the equation for the remainder $\chi_{R}w$:
\begin{align}
\partial_{\tau}\chi_{R}w=&-L \ \chi_{\Omega}w
+\chi_{R}\Big[F+G+N_1(v)+N_2(\eta)\Big]+\mu_{R}(w),\label{eq:preliRW}
\end{align} where the operator $L$ is defined as
\begin{align}
L:=-\Delta_{y}+\frac{1}{2}y\cdot \nabla_{y}-\frac{1}{2}-V_{a,B}^{-2},
\end{align}
the functions $F$, $G$, $N_1$ and $N_2$ are defined in \eqref{eq:wEqn},
and $\mu_{R}(w)$ is defined as
\begin{align*}
\mu_{R}(w):=\frac{1}{2}\big(y\cdot\nabla_{y}\chi_{R}\big)w+\big(\partial_{\tau}\chi_{R}\big)w-\big(\Delta_{y}\chi_{R}\big)w-2\nabla_{y}\chi_{R}\cdot  \nabla_{y}w.
\end{align*}

As in \eqref{eq:fourParts} we define three functions $w_{m},\ m=-1,0,1,$ by decomposing $w,$
\begin{align}
w(y,\theta,\tau)=w_0(y,\tau)+e^{i\theta} w_{1}(y,\tau)+e^{-i\theta}w_{-1}(y,\tau)+P_{\theta,\geq 2} w(y,\theta,\tau).
\end{align}

Impose $P_{\theta,\geq 2}$ on both sides of \eqref{eq:preliRW} and use that $P_{\theta,\geq 2}(F+G)=0$ to find
\begin{align}
\partial_{\tau}(P_{\theta,\geq 2}\chi_{R}w)=&-L(P_{\theta,\geq 2}\chi_{\Omega}w)
+P_{\theta,\geq 2}\Big[N_1(\eta)\chi_{R}+N_2(v)\chi_{R}+\mu_{R}(w)\Big]. \label{eq:chRw}
\end{align}

In the rest of the section we prove Proposition \ref{prop:332211}, based on \eqref{eq:chRw}.

The following two reasons make proving Proposition \ref{prop:332211} considerably easier than estimating the functions $\Big\|(100+|y|)^{-3+|k|} \|P_{\theta,\geq 2} \partial_{\theta}^{3-|k|}\nabla_{y}^{k}\chi_{\Omega}v\|_{L_{\theta}^2}\Big\|_{\infty},\ |k|=0,1,2,$ in the previous sections: 
\begin{itemize}
\item[(1)] the needed decay estimates are slower and hence it is relatively easier to obtain, 
\item[(2)] in the presently considered region $|y|\leq (1+\epsilon)R(\tau)=\mathcal{O} (\sqrt{ln \tau})$, we have $V_{a,B}\approx \sqrt{2}$ implied by its definition \eqref{eq:defVaB}, while in the region $|y|\leq (1+\epsilon)\Omega(\tau)=\mathcal{O}(\tau^{\frac{1}{2}+\frac{1}{20}})$, $V_{a,B}$ might be (adversely large) for large $|y|$. 
\end{itemize}

Based on these two reasons we often skip the details.

\subsection{Proof of \eqref{eq:33}}

The main tool is the maximum principle. Now we derive an equation to make it applicable.

For notational purpose we define $$\tilde\Phi_3(y,\tau):=(100+|y|^2)^{-3}\|P_{\theta,\geq 2}\partial_{\theta}^3 \chi_{R}w(y,\cdot,\tau)\|_{L^{2}_{\theta}}^2,$$ 
and derive an equation for it from \eqref{eq:chRw},
\begin{align}
\begin{split}\label{eq:Phi3Maxi}
\partial_{\tau}\tilde\Phi_3=&-(L_3+W_3)\tilde\Phi_3-2 (100+|y|^2)^{-3}\Big[V_{a,B}^{-2}\|P_{\theta,\geq 2}\partial_{\theta}^4 \chi_{R}w\|_{L^{2}_{\theta}}^2+\|P_{\theta,\geq 2}\partial_{\theta}^3\nabla_{y} \chi_{R}w\|_{L^{2}_{\theta}}^2\Big]\\
&+2(100+|y|^2)^{-3}D.
\end{split}
\end{align} Here $L_3$ is a differential operator, and $W_3$ is a multiplier, defined as
\begin{align}
\begin{split}
L_3:=&-\Delta+\frac{1}{2}y\cdot \nabla_{y}-2(100+|y|^2)^{-3} \Big(\nabla_y (100+|y|^2)^{3}\Big)\cdot \nabla_{y},\\
W_3:=&-1+\frac{3|y|^2}{100+|y|^2}-\frac{18}{100+|y|^2}-\frac{24|y|^2}{(100+|y|^2)^2}-2V_{a,B}^{-2},
\end{split}
\end{align}
and the term $D$ is defined as
\begin{align*}
D:=&\langle P_{\theta,\geq 2}\partial_{\theta}^3 \chi_{R}w, \ \partial_{\theta}^{3}\Big(\chi_{R}\big(N_{1}(v)+N_{2}(\eta)\big)+\mu_{R}(w)\Big)\rangle_{\theta}.
\end{align*}

We claim that $D$ satisfies the estimate, for large $\tau$,
\begin{align}\label{eq:estD3R}
(100+|y|^2)^{-3} D\leq \frac{1}{50}(100+|y|^2)^{-3}\Big[\|P_{\theta,\geq 2}\partial_{\theta}^4 \chi_{R}w\|_{L^{2}_{\theta}}^2+\|P_{\theta,\geq 2}\partial_{\theta}^3\nabla_{y} \chi_{R}w\|_{L^{2}_{\theta}}^2\Big]+C\delta^2\kappa^2(\epsilon)R^{-8}.
\end{align}

Suppose the claim holds, then its first two positive terms are cancelled by the negative ones in \eqref{eq:Phi3Maxi}. This, together with the facts $\|P_{\theta,\geq 2}\partial_{\theta}^4 \chi_{R}w\|_{L^{2}_{\theta}}^2\geq 4 \|P_{\theta,\geq 2}\partial_{\theta}^3 \chi_{R}w\|_{L^{2}_{\theta}}^2$ and that $V_{a,B}^{-2}=\frac{1}{2}+\mathcal{O}(\tau^{-\frac{1}{2}})$ if $|y|\leq (1+\epsilon)R$, implies that, for some $C>0,$
\begin{align}
\partial_{\tau}\tilde\Phi_3\leq -(L_3+\frac{1}{2})\tilde\Phi_3+C \delta^2 \kappa^2(\epsilon) R^{-8}.\label{eq:tPhi3}
\end{align} 
Before applying the maximum principle we need to check the boundary condition. The cutoff function $\chi_{R}$ makes $\tilde\Phi_3(y,\tau)=0$ if $|y|\geq (1+\epsilon)R(\tau)$.

Apply the maximum principle to have that, for some $C_1>0,$
\begin{align}
\tilde{\Phi}_3(\tau)\leq e^{-\frac{1}{2}(\tau-\tau_1)}\tilde\Phi_3(\tau_1)+C_1 \delta^2 \kappa^2(\epsilon) R^{-8}.
\end{align}

This, together with $\tilde\Phi_3(\tau_1)\leq \delta$ if $\tau\geq \tau_0$ implied by \eqref{eq:smallRw} implies the desired estimate, provided that $\tau$ is sufficiently large.

What is left is to prove the claim \eqref{eq:estD3R}.

\subsubsection{Proof of \eqref{eq:estD3R}}
We decompose $D$ into two terms
\begin{align}
D:= D_{1}+D_{2},\label{eq:DD1D2}
\end{align} where $D_{1}$ is defined as
\begin{align*}
D_{1}:= \langle P_{\theta,\geq 2}\partial_{\theta}^3 \chi_{R}w, \ \partial_{\theta}^{3} \chi_{R}\big(N_{1}+N_{2}\big)\rangle_{\theta}=-\langle P_{\theta,\geq 2}\partial_{\theta}^4 \chi_{R}w, \ \chi_{R} \partial_{\theta}^{2} \big(N_{1}+N_{2}\big)\rangle_{\theta},
\end{align*} here we integrate by parts in $\theta$ in the second step, and $D_2$ is defined as
\begin{align*}
D_2:=\langle P_{\theta,\geq 2}\partial_{\theta}^3 \chi_{R}w, \ \mu_{R}(\partial_{\theta}^{3}w)\rangle_{\theta}.
\end{align*}

For $D_1$ we observe that
\begin{align}
\chi_{R}|\partial_{\theta}^{2}N_{1}|\lesssim &\delta \chi_{R}\sum_{0\leq l\leq 2} \Big[|\partial_{\theta}^{l}\nabla_{y}v|^2+\sum_{0\leq l
\leq 2}|\partial_{\theta}^{l+1}v|\Big],\label{eq:2ThetN1}\\
\chi_{R}\|\partial_{\theta}^2 N_2\|_{L_{\theta}^2}\lesssim& \delta \chi_{R}\sum_{l=1,2,3,4}\|\partial_{\theta}^l \eta\|_{L_{\theta}^2}.\label{eq:2ThetN2}
\end{align}
Here in deriving \eqref{eq:2ThetN1} we use the definition of $N_1$
and the estimates in \eqref{eq:smallRw} to find that
\begin{align}
\begin{split}\label{eq:distribuRTheta}
|\partial_{\theta}^2 N_1|\lesssim \sum_{m=0}^{2} &\Big[\sum_{k,l=1,2,3}|\partial_{\theta}^m (\partial_{y_k}v\ \partial_{y_l}v \ \partial_{y_k}\partial_{y_l}v)|+\sum_{k=1,2,3}|\partial_{\theta}^m (\partial_{y_k}v\ \partial_{\theta}v\ \partial_{y_k}\partial_{\theta}v)|\\
&+|\partial_{\theta}^m \big((\partial_{\theta}v)^2\ \partial_{\theta}^2v\big) |+|\partial_{\theta}^m (\partial_{\theta}v)^2 |\Big],
\end{split}
\end{align}
then we consider the distribution of the derivatives $\partial_{\theta}^m$ among the terms.
And \eqref{eq:2ThetN2} is derived similarly, besides using that $\partial_{\theta}v=\partial_{\theta}\eta$.

For the first term in \eqref{eq:2ThetN1}, we decompose $v$ to find
\begin{align}
\chi_{R} |\partial_{\theta}^{l}\nabla_y v|^2\lesssim (1+|y|)^2 \tau^{-1}+\sum_{m=\pm 1}|\nabla_{y}\chi_{R}w_{m}||\nabla_{y}w_{m}|+\delta |\nabla_{y}\chi_{R}|+\delta |\partial_{\theta}^{l}\nabla_{y}P_{\theta,\geq 2}\chi_{R}w|\label{eq:a177}
\end{align} where $\nabla_{y}\chi_{R}$ is produced in changing the order of $\nabla_{y}$ and $\chi_{\Omega}$, and we use $\delta$ to bound terms in \eqref{eq:smallRw}. For the term $|\nabla_{y}\chi_{R}w_{m}||\nabla_{y}w_{m}|,\ m=\pm 1,$ we have
\begin{align}
\begin{split}
&(100+|y|^2)^{-\frac{3}{2}}\sum_{m=\pm 1}\Big\|   |\nabla_{y}\chi_{R}w_{m}||\nabla_{y}w_{m}|  \Big\|_{L_{\theta}^2}\\
\lesssim &\|\langle y\rangle^{-2}\nabla_{y}\chi_{R}w_{m}\|_{\infty}\Big[\|\langle y\rangle^{-1}\nabla_y \chi_{R}w_{m}\|_{\infty}+\|1_{\leq (1+\epsilon)R}\langle y\rangle^{-1}\nabla_y(1-\chi_{R}) w_{m}\|_{\infty}\Big]\\
\lesssim &\|\langle y\rangle^{-2}\nabla_{y}\chi_{R}w\|_{\infty}\Big[R^2 \|\langle y\rangle^{-3}\nabla_y \chi_{R}w\|_{\infty}+\|1_{\leq (1+\epsilon)R}\langle y\rangle^{-1}\nabla_y(1-\chi_{R}) w\|_{\infty}\Big]\\
\lesssim & \kappa(\epsilon) R^{-4}\Big(\kappa(\epsilon)R^{-1}+\delta\Big),
\end{split}
\end{align} 
where in the second step, we insert $1=\chi_{R}+(1-\chi_{R})$ before $w$ in the second factor, and then apply the estimates in \eqref{eq:prev2} in the last step;
we also use that the function $\nabla_y(1-\chi_{R}) w$ is supported in $|y|\geq R$, hence $\langle y\rangle^{-1}\lesssim R^{-1}$ here.

The second term in \eqref{eq:2ThetN1} can be controlled more easily since $\chi_{\Omega}$ and $\partial_{\theta}$ commute. These, together with that $\|\partial_{\theta}f\|_{L_{\theta}^2}\leq \|\partial_{\theta}^2f\|_{L_{\theta}^2}$ for any smooth function $f$ and that 
\begin{align}
\tau^{-\frac{1}{2}}\leq \kappa(\epsilon)R^{-10}(\tau)=\mathcal{O}((\ln\tau)^{-5}),
\end{align} 
and that, since $|\nabla_{y}\chi_{R}|=R^{-1}\Big|\chi^{'}(\frac{|y|}{R})\Big|$ and it is supported by the set $|y|\geq R,$
\begin{align}
\langle y\rangle^{-3}|\nabla_{y}\chi_{R}|\lesssim \kappa(\epsilon)R^{-4}
\end{align}
makes 
\begin{align}
\begin{split}\label{eq:a17}
&(100+|y|^2)^{-\frac{3}{2}}\chi_{R}\|\partial_{\theta}^{2}\chi_{R}N_{1}(v)\|_{L_{\theta}^2}\\
\lesssim &\delta(100+|y|^2)^{-\frac{3}{2}}\Big[  \|\nabla_{y}\partial_{\theta}^3 P_{\theta,\geq 2} \chi_{R}w\|_{L_{\theta}^2}+  \|\partial_{\theta}^4\chi_{R}P_{\theta,\geq 2} w\|_{L_{\theta}^2}\Big]+\delta \kappa(\epsilon)R^{-4}.
\end{split}
\end{align}

Apply similar techniques on \eqref{eq:2ThetN2} to find that,
\begin{align}
&(100+|y|^2)^{-\frac{3}{2}}\chi_{R}\|\partial_{\theta}^{2}\chi_{R}N_{2}\|_{L_{\theta}^2}
\lesssim \delta(100+|y|^2)^{-\frac{3}{2}}  \|\partial_{\theta}^4\chi_{R}P_{\theta,\geq 2} w\|_{L_{\theta}^2}+\delta \kappa(\epsilon)R^{-4}.
\end{align}

Collect the estimates above and apply Young's inequality to have, for some $C>0,$
\begin{align}
\begin{split}\label{eq:estD1Phi3}
&(100+|y|^2)^{-3}|D_{1}|\\
\leq &\frac{1}{100}(100+|y|^2)^{-3}\Big[\|P_{\theta,\geq 2}\partial_{\theta}^4 \chi_{R}w\|_{L^{2}_{\theta}}^2+ \|P_{\theta,\geq 2}\partial_{\theta}^{3}\chi_{R}N_{2}(\eta)\|_{L^{2}_{\theta}}^2\Big]+C\delta^2 \kappa^2 (\epsilon)R^{-8}.
\end{split}
\end{align}

For $D_{2}$, the term $\frac{1}{2}(y\nabla_{y}\chi_{R})$ in the definition of $\mu_{R}(\partial_{\theta}^3P_{\theta,\geq 2} w)$ is of order $\mathcal{O}(1)$, but it has a favorable non-positive sign. This makes
\begin{align}
D_{2}\leq \Big\langle P_{\theta,\geq 2}\partial_{\theta}^3 \chi_{R}w, \ P_{\theta,\geq 2}\partial_{\theta}^3\Big(\big(\partial_{\tau}\chi_{R}\big)w-\big(\Delta_{y}\chi_{R}\big)w-2\nabla_{y}\chi_{R}\cdot  \nabla_{y} w\Big)\Big\rangle_{\theta}.\label{eq:estD33}
\end{align} 

Here the decay estimates are from the derivatives of $\chi_{R}$ and that they are supported by the set $|y|\geq R$. We use \eqref{eq:smallRw}, to have that, for some $C>0,$
\begin{align}
(100+|y|^2)^{-3} D_{2}\leq \frac{1}{100} \tilde\Phi_3+C\delta^2 \kappa^2 (\epsilon)R^{-8}.\label{eq:estD2Phi3}
\end{align}

Take the estimates in \eqref{eq:estD1Phi3}, \eqref{eq:estD2Phi3} to \eqref{eq:DD1D2}, and obtain the desired estimate \eqref{eq:estD3R}.

\subsection{Proof of \eqref{eq:22}}
Compute directly to find that the function $\tilde\Phi_2$, defined as 
\begin{align}
\tilde\Phi_2:=(100+|y|^2)^{-2}\|\nabla_{y} \partial_{\theta}^2P_{
\theta,\geq 2}\chi_{R}w(y,\cdot,\tau)\|_{L^{2}_{\theta}}^2,
\end{align} satisfies the equation
\begin{align}
\begin{split}
\partial_{\tau}\tilde\Phi_2=&-(L_2+W_2)\tilde\Phi_2\\
&-2 (100+|y|^2)^{-2}\Big[V_{a,B}^{-2}\|P_{\theta,\geq 2}\partial_{\theta}^3\nabla_y \chi_{R}w\|_{L^{2}_{\theta}}^2+\sum_{l=1}^3\|P_{\theta,\geq 2}\partial_{\theta}^2\nabla_{y}\partial_{y_l} \chi_{R}w\|_{L^{2}_{\theta}}^2\Big]\\
&+2(100+|y|^2)^{-2}(U_{1}+U_{2}+U_{3}),
\end{split}
\end{align} where $L_2$ is a differential operator and $W_2$ is a multiplier, defined as,
\begin{align*}
L_2:=&-\Delta+\frac{1}{2}y\cdot \nabla_{y}-2 (100+y^2)^{-2} \big(\nabla_{y} (100+|y|^2)^{2}\big)\cdot \nabla_y,\\
W_2:=&\frac{2|y|^2}{100+|y|^2}-\frac{12}{100+|y|^2}-\frac{8|y|^2}{(100+|y|^2)^2}-2V_{a,B}^{-2},
\end{align*}
and we use that $\partial_{y_k} y\cdot \nabla_y g=(y\cdot \nabla_y+1)\partial_{y_k}g$. 
The terms $U_{l},\ l=1,2,3,$ are defined as
\begin{align*}
U_{1}:=&\langle \nabla_{y} \partial_{\theta}^2P_{
\theta,\geq 2}\chi_{R}w,\ (\nabla_{y}V_{a,B}) \partial_{\theta}^2 \chi_{R}w\rangle_{\theta},\\
U_{2}:=&\langle \nabla_{y} \partial_{\theta}^2P_{
\theta,\geq 2}\chi_{R}w,\ \nabla_{y} \partial_{\theta}^2 \chi_{R}\big(N_{1}(v)+N_{2}(\eta)\big)\rangle_{\theta},\\
U_3:=&\langle \nabla_{y} \partial_{\theta}^2P_{
\theta,\geq 2}\chi_{R}w,\ \nabla_{y} \partial_{\theta}^2\mu_{R}(w)\rangle_{\theta}.
\end{align*}

We claim that, for some $C>0,$ and $\tau$ is sufficiently large,
\begin{align}
\begin{split}\label{eq:D212}
|U_{1}|\leq &C\delta \tau^{-\frac{1}{2}} \|\nabla_{y} \partial_{\theta}^2P_{
\theta,\geq 2}\chi_{R}w\|_{L_{\theta}^2},\\
(100+|y|^2)^{-2}|U_{2}|\leq &\frac{1}{100}(100+|y|^2)^{-1}\|P_{\theta,\geq 2}\nabla_{y}\partial_{\theta}^3 \chi_{R}w\|_{L_{\theta}^{2}}+C\delta^2 \kappa^2 (\epsilon)R^{-6},\\
(100+|y|^2)^{-2}U_3\leq & C\delta^2 \kappa^2 (\epsilon)R^{-6}.
\end{split}
\end{align}
The claims will be proved in subsubsection \ref{subsub:D212}.

Suppose the claims hold, then as proving \eqref{eq:tPhi3}, we find that, for some $C>0,$
\begin{align}
\partial_{\tau}\tilde\Phi_2\leq -(L_2+\frac{1}{2})\tilde\Phi_2+C \delta^2 \kappa^2(\epsilon) R^{-6}.
\end{align} 
Apply the maximum principle to have that, for some $C_1>0,$
\begin{align}
\tilde{\Phi}_2(\tau)\leq e^{-\frac{1}{2}(\tau-\tau_1)}\tilde\Phi_2(\tau_1)+C_1 \delta^2 \kappa^2(\epsilon) R^{-6}.
\end{align}

We observe that $\tilde\Phi_2(\tau_1)\leq \delta$ for $\tau_1\geq \tau_0$ implied by \eqref{eq:smallRw}. Hence if $\tau$ is sufficiently large, we have the desired estimate.

\subsubsection{Proof of \eqref{eq:D212}}\label{subsub:D212}
It is easy to prove the estimate for $U_1$ since $|\nabla_{y}V_{a,B}|\lesssim \tau^{-\frac{1}{2}}$ and $|\partial_{\theta}^2\chi_{R}w|\lesssim \delta.$

To prove the second estimate, we follow the steps in \eqref{eq:distribuRTheta} to find
\begin{align}
\begin{split}
|\partial_{\theta}\nabla_y\chi_{R} N_1|\lesssim & \delta \Big[\sum_{l=0,1,2, 3} |\partial_{\theta}^{l}\nabla_{y}v|+\sum_{l=1,2}|\partial_{\theta}^{l}v|\Big]+\delta |\nabla_{y}\chi_{R}|,\\
|\partial_{\theta}\nabla_y\chi_{R} N_2|\lesssim & \delta \sum_{|k|,l=0,1} |\nabla_{y}^k\partial_{\theta}^l \eta|+\delta |\nabla_{y}\chi_{R}|.
\end{split}
\end{align}

This, together with the techniques in proving \eqref{eq:a17}, implies that
\begin{align}\label{eq:2RThetaY}
 \langle y\rangle^{-2}\|\partial_{\theta}\nabla_y \chi_{R}N_1\|_{L_{\theta}^2},\ \langle y\rangle^{-2} \|\partial_{\theta}\nabla_y\chi_{R} N_2\|_{L_{\theta}^2}\lesssim \delta \|\langle y\rangle^{-2}\nabla_{y} \partial_{\theta}^3 P_{
\theta,\geq 2}\chi_{R}w\|_{L^{2}_{\theta}}+\delta \kappa(\epsilon)R^{-3}+\delta\tau^{-1}.
\end{align}

Collect the estimates above to find the desired estimate for $U_2$ in \eqref{eq:D212}.

For $U_3$, the method is the same as that in \eqref{eq:estD33}, and we choose to skip the details here.

\subsection{Proof of \eqref{eq:11}}

Compute directly to find that the function $\tilde\Phi_l$, $ l\in (\mathbb{N}\cup \{0\})^3$ and $|l|=2,$ defined as
\begin{align}
\tilde\Phi_{l}(y,\tau):=(100+|y|^2)^{-1}\langle P_{\theta,\geq 2} \nabla_{y}^{l} \partial_{\theta}\chi_{R}w, \ \nabla_{y}^{l} \partial_{\theta}\chi_{R}w\rangle_{\theta},
\end{align} satisfies the equation
\begin{align}
\begin{split}
\partial_{\tau}\tilde\Phi_{l}=&-(L_1+W_1)\tilde\Phi_l-2(100+|y|^2)^{-1}\Big( V_{a,B}^{-2}\|P_{\theta,\geq 2}\partial_{\theta}^2\nabla_{y}^{l} \chi_{R}w\|_{L^{2}_{\theta}}^2+\|P_{\theta,\geq 2}\partial_{\theta}\nabla_{y} \nabla_{y}^{l}\chi_{R}w\|_{L^{2}_{\theta}}^2\Big)\\
&+2(100+|y|^2)^{-1}\sum_{k=1,2}\Upsilon_{k},
\end{split}
\end{align} where the linear operators $L_1$ and $W_1$ are defined as
\begin{align}
\begin{split}
L_1:=&-\Delta+\frac{1}{2}y\cdot\nabla_{y}-\frac{2y}{100+|y|^2}\cdot \nabla_{y}\\
W_1:=&1+\frac{|y|^2}{100+|y|^2}-\frac{6}{100+|y|^2}-2V_{a,B}^{-2},
\end{split}
\end{align}
and we use the relation that, for any function $h$, $\partial_{y_k}(\frac{1}{2}y\cdot \nabla_{y}h)=\big(\frac{1}{2}y\cdot \nabla_{y}+\frac{1}{2}\big)\partial_{y_k}h$, and the terms $\Upsilon_{k},\ k=1,2,$ are defined as
\begin{align*}
\Upsilon_{1}:=&\Big\langle P_{\theta,\geq 2}\nabla_{y}^{l}\partial_{\theta} \chi_{R}w, \ \nabla_{y}^{l}\partial_{\theta}V_{a,B}^{-2}\chi_{R}w-V_{a,B}^{-2}\nabla_{y}^{l}\partial_{\theta}\chi_{R}w  \Big\rangle_{\theta},\\
\Upsilon_{2}:=&\Big\langle P_{\theta,\geq 2}\nabla_{y}^{l}\partial_{\theta} \chi_{R}w, \ \nabla_{y}^{l}\partial_{\theta}\Big(\chi_{R}\big(N_{1}(v)+N_{2}(\eta)\big)+\mu_{R}(w)\Big)\Big\rangle_{\theta}\\
=&-\Big\langle P_{\theta,\geq 2}\nabla_{y}^{l}\partial_{\theta}^2 \chi_{R}w, \ \nabla_{y}^{l}\Big(\chi_{R}\big(N_{1}(v)+N_{2}(\eta)\big)+\mu_{R}(w)\Big)\Big\rangle_{\theta}
\end{align*}

We claim that, for some $C>0$,
\begin{align}
|\Upsilon_1|\leq &\delta\tau^{-\frac{1}{2}} \|P_{\theta,\geq 2}\nabla_{y}^{l}\partial_{\theta} \chi_{R}w\|_{L_{\theta}^{2}}, \label{eq:W1}\\
(100+|y|^2)^{-1}\Upsilon_{2}\leq &\frac{1}{100}(100+|y|^2)^{-1}\sum_{|l|=2}\|P_{\theta,\geq 2}\nabla_{y}^{l}\partial_{\theta}^2 \chi_{R}w\|_{L_{\theta}^{2}}^2+C\delta^2 \kappa^2 (\epsilon)R^{-4}.\label{eq:W2}
\end{align}
The claim will be proved in the subsubsection \ref{subsub:D112}.

Suppose the claims hold, then as proving \eqref{eq:tPhi3}, we have that, for some $C>0,$
\begin{align}
\partial_{\tau}\sum_{|l|=2}\tilde\Phi_l\leq -(L_1+\frac{1}{2})\sum_{|l|=2}\tilde\Phi_l+C \delta^2 \kappa^2(\epsilon) R^{-4}.
\end{align} 
Apply the maximum principle to have that, for some $C_1>0,$
\begin{align}
\sum_{|l|=2}\tilde{\Phi}_l(\tau)\leq e^{-\frac{1}{2}(\tau-\tau_1)}\sum_{|l|=2}\tilde\Phi_l(\tau_1)+C_1 \delta^2 \kappa^2(\epsilon) R^{-4}.
\end{align}

This, together with $\tilde\Phi_l(\tau_1)\leq \delta$ if $\tau_1\geq \tau_0$ implied by \eqref{eq:smallRw} and choosing $\tau$ to be sufficiently large, implies the desired estimate.

\subsubsection{Proofs of \eqref{eq:W1} and \eqref{eq:W2}}\label{subsub:D112}
It is easy to prove \eqref{eq:W1}, by the fact that $|\nabla_{y}^{l}V_{a,B}|\lesssim \tau^{-\frac{1}{2}},\ |l|=1,2.$

Next we prove \eqref{eq:W2}.
For $N_1$ we use the idea in \eqref{eq:distribuRTheta} to find that, for any $|k|=2,$
\begin{align}
|\nabla_{y}^k N_1|\lesssim &\delta\Big[\sum_{|n|=1,2,\ l=0,1,2} |\partial_{\theta}^l \nabla_{y}^{n}v|+ \sum_{l=1,2}|\partial_{\theta}^l v|\Big],
\end{align} for $N_2$, compute directly to have
\begin{align}
|\nabla_{y}^k  N_2|\lesssim &\delta \sum_{|l|=0,1,2} |\nabla_{y}^l \eta|.
\end{align}

Apply the same techniques as in proving \eqref{eq:a17} to find
\begin{align*}
\langle y\rangle^{-1} \chi_{R}\|\nabla_{y}^k N_1\|_{L_{\theta}^2}\lesssim &\delta \langle y\rangle^{-1} \sum_{|n|=2}\|\nabla_{y}^n\partial_{\theta}^2P_{\theta,\geq 2}\chi_{R}w\|_{L_{\theta}^2}+\delta\kappa(\epsilon)R^{-2}+\delta\tau^{-1},\\
\langle y\rangle^{-1}\chi_{R}\|\nabla_{y}^k N_2\|_{L_{\theta}^2}\lesssim &\delta \kappa(\epsilon)R^{-2}+\delta\tau^{-1}.
\end{align*}

We estimate the $\mu_{R}-$term by the same methods as that in proving \eqref{eq:estD33}.

Take these estimates above to the definition of $\Upsilon_2$ and apply Young's inequality, we find the desired estimate.

\def\cprime{$'$}

\end{document}